\documentclass[12pt]{amsart}

\usepackage{amssymb,amsthm,amsmath}

\usepackage{mathtools}

\RequirePackage[dvipsnames,usenames]{color}

\makeatletter
\def\@tocline#1#2#3#4#5#6#7{\relax
  \ifnum #1>\c@tocdepth 
  \else
    \par \addpenalty\@secpenalty\addvspace{#2}%
    \begingroup \hyphenpenalty\@M
    \@ifempty{#4}{%
      \@tempdima\csname r@tocindent\number#1\endcsname\relax
    }{%
      \@tempdima#4\relax
    }%
    \parindent\z@ \leftskip#3\relax \advance\leftskip\@tempdima\relax
    \rightskip\@pnumwidth plus4em \parfillskip-\@pnumwidth
    #5\leavevmode\hskip-\@tempdima
      \ifcase #1
       \or\or \hskip 1em \or \hskip 2em \else \hskip 3em \fi%
      #6\nobreak\relax
    \hfill\hbox to\@pnumwidth{\@tocpagenum{#7}}\par
    \nobreak
    \endgroup
  \fi}
\makeatother

\input{kmacros3.sty}

\usepackage{tikz}
\usepackage{tikz-cd}
\usetikzlibrary{cd}
\usepackage{graphicx}
\usepackage[all,cmtip]{xy}

\usepackage{mabliautoref}
\numberwithin{equation}{theorem}

\DeclareRobustCommand{\bigplus}{\pmb{+}}

\usepackage{bm}
\usepackage{pifont}
\usepackage{upgreek}

\usepackage{eucal}
\usepackage{ulem}
\usepackage{stmaryrd}

\numberwithin{equation}{theorem}

\usepackage{fullpage}

\usepackage{enumerate}
\usepackage{calc}

\usepackage{verbatim}
\usepackage{alltt}
\usepackage{scalerel,stackengine}

\def\todo#1{\textcolor{red}%
{\footnotesize\newline{\color{red}\fbox{\parbox{\textwidth-15pt}{\textbf{todo: } #1}}}\newline}}

\newcommand{\perf}{\textnormal{perf}}
\newcommand{\perfd}{\textnormal{perfd}}
\newcommand{\relperfd}{\textnormal{rel, perfd}}
\newcommand{\ratperfd}{\textnormal{rat, perfd}}
\newcommand{\Ainfty}{A_{\infty}}
\renewcommand{\m}{\mathfrak{m}}

\stackMath
\newcommand\reallywidehat[1]{%
\savestack{\tmpbox}{\stretchto{%
  \scaleto{%
    \scalerel*[\widthof{\ensuremath{#1}}]{\kern.1pt\mathchar"0362\kern.1pt}%
    {\rule{0ex}{\textheight}}
  }{\textheight}%
}{2.4ex}}%
\stackon[-6.9pt]{#1}{\tmpbox}%
}
\begin{document}

\title{Perfectoid signature, perfectoid Hilbert-Kunz multiplicity, and an application to local fundamental groups}

\author{Hanlin Cai}
\address{Department of Mathematics, University of Utah, Salt Lake City, UT 84112, USA}
\email{cai@math.utah.edu}
\thanks{Cai was supported in part by NSF FRG Grant \#1952522.}

\author{Seungsu Lee}
\address{Department of Mathematics, University of Utah, Salt Lake City, UT 84112, USA}
\email{slee@math.utah.edu}
\thanks{Lee was supported in part by NSF Grant  \#2101800.}

\author{Linquan Ma}
\address{Department of Mathematics, Purdue University, West Lafayette, IN 47907, USA}
\email{ma326@purdue.edu}
\thanks{Ma was supported in part by NSF Grant \#2302430, and was supported by \#1901672, NSF FRG Grant DMS \#1952366, and a fellowship from the Sloan Foundation when writing this article.
}

\author{Karl Schwede}
\address{Department of Mathematics, University of Utah, Salt Lake City, UT 84112, USA}
\email{schwede@math.utah.edu}
\thanks{Schwede was supported in part by NSF Grants \#1801849 and \#2101800, NSF FRG Grant \#1952522 and a
fellowship from the Simons Foundation.
}

\author{Kevin Tucker}
\address{Department of Mathematics, University of Illinois at Chicago, Chicago, IL 60607, USA}
\email{kftucker@uic.edu}
\thanks{Tucker was supported in part by NSF Grant \#2200716.}
\begin{abstract}
  We define a (perfectoid) mixed characteristic version of $F$-signature and Hilbert-Kunz multiplicity by utilizing the perfectoidization functor of Bhatt-Scholze and Faltings' normalized length (also developed in the work of Gabber-Ramero).  We show that these definitions coincide with the classical theory in equal characteristic $p > 0$.  We prove that a ring is regular if and only if either its perfectoid signature or perfectoid Hilbert-Kunz multiplicity is 1 and we show that perfectoid Hilbert-Kunz multiplicity characterizes BCM closure and extended plus closure of $\m$-primary ideals.  We demonstrate that perfectoid signature detects BCM-regularity and transforms similarly to $F$-signature or normalized volume under quasi-\'etale maps. As a consequence, we prove that BCM-regular rings have finite local \'etale fundamental group and also finite torsion part of their divisor class groups.  Finally, we also define a mixed characteristic version of relative rational signature, and show it characterizes BCM-rational singularities.
\end{abstract}
\maketitle
\renewcommand{\baselinestretch}{0.25}\normalsize
\tableofcontents
\renewcommand{\baselinestretch}{1.0}\normalsize
\section{Introduction}

    In \cite{KunzCharacterizationsOfRegularLocalRings}, Kunz showed that a Noetherian ring of characteristic $p > 0$ is regular if and only if its Frobenius endomorphism is flat.  Building directly upon that work, Kunz studied length of the quotient $R/\m^{[p^e]}$ as $e$ varies for $R$ a Noetherian ring of characteristic $p > 0$, \cite{KunzOnNoetherianRingsOfCharP}.  
    In \cite{MonskyHKFunction}, Monsky showed that the limit 
    \[
        \lim_{e \to \infty}{\length_R(R/\m^{[p^e]}) \over p^{ed}}
    \]
    exists and coined the limit as \emph{Hilbert-Kunz multiplicity of $R$}, denoted $e_{\HK}(R)$.  It is not difficult to see that the Hilbert-Kunz multiplicity is at least $1$.  Furthermore, $e_{\HK}(R) = 1$ if and only if $R$ is regular \cite{WatanabeYoshidaHKMultAndInequality}.  Roughly speaking, Hilbert-Kunz multiplicity is a larger number when the singularities of $R$ are ``more severe''.  
    
    This mysterious number, which can be irrational \cite{BrennerIrrationalHK} and is expected to sometimes be transcendental \cite{MonskyTranscendenceOfSomeHKMults}, has been extensively studied in characteristic $p > 0$ commutative algebra, see for instance
    \cite{HanMonskySomeSurprisingHKFunctions,
    HunekeYaoMinimalHKMultiplicity,
    HanesNotesOnHKFunction,
    WatanabeYoshidaMinimalRelative,
    HunekeMcDermottHkForNormal,
    MonskyTeixeiraPFractals1, 
    BlickleEnescuOnRingsWithSmallHKMult,
    TrivediHKMultAndReduction,
    MonskyRationalityOfSomeHKMultsCounterexample, 
    VraciuANewVersionOfTightClosure,
    AberbachEnescuLowerBoundsForHK, 
    AberbachEnescuNewEstimatesOfHKMult,
    CelibkasDaoHunekeBoundsOnHK,
    SmirnovUpperSemiContinuityOfHK,
    EpsteinYaoSomeExtensionsOfHKMult,    
    NunezBetancourtHKMultAndThresholds,    
    TrivediWatanabeHilbertKunzDensityFunctions}.  
    It is connected deeply with tight closure theory \cite{HochsterHunekeTC1,WatanabeYoshidaMinimalRelative,BrennerMonsky} and has played an important role in recent advances on Lech's conjecture and Lech's inequality \cite{MaLechsConjectureDimensionThree,HunekeMaQuySmirnovAsymptoticLech}.

    In a similar direction, building upon work of Smith and Van den Bergh \cite{SmithVanDenBerghSimplicityOfDiff}, Huneke and Leuschke coined the term $F$-signature \cite{HunekeLeuschkeTwoTheoremsAboutMaximal}.  There, assuming that $(R, \m)$ is a Noetherian local domain of characteristic $p > 0$ with perfect residue field (for simplicity), they defined $a_e(R)$ to be the maximal number of simultaneous free-summands of $R^{1/p^e}$ as an $R$-module (the \emph{free-rank of $R^{1/p^e}$}).  The $F$-signature of $R$ is then defined to be:
    \[
        s(R) := \lim_{e \to \infty} {a_e(R) \over p^{ed}}
    \]
    although it was not until \cite{TuckerFsigExists} that the limit was shown to exist in general (it is expected also to be sometimes irrational).  Notice that the $F$-signature is at most $1$ (with equality if and only if $R$ is regular \cite{WatanabeYoshidaHKMultAndInequality}, and see \cite{HunekeLeuschkeTwoTheoremsAboutMaximal,HunekeYaoMinimalHKMultiplicity}) and it was shown in \cite{AberbachLeuschke} that $R$ is strongly $F$-regular if and only if $s(R) > 0$.  This time, more severe singularities have \emph{smaller} $F$-signature.
    
    Another way to interpret the numbers $a_e(R)$ is to observe that if 
    \[
        I_e = \{x \in R \;|\; \phi(x^{1/p^e}) \in \m \text{ for all } \phi \in \Hom_R(R^{1/p^e}, R) \}
    \]
    then we have $a_e(R) = \length_R(R/I_e)$ so that 
    \begin{equation}
        \label{eq.SignatureViaIeColength}\tag{*}
        s(R) = \lim_{e \to \infty} {\length_R(R/I_e) \over p^{ed}}.
    \end{equation}
    This latter description will be important for us shortly.  Again, $F$-signature has been extensively studied within commutative algebra, see for instance
    \cite{
    SinghFSignatureOfAffineSemigroup,  
    AberbachExistenceOfTheFSignature,   
    AberbachEnescuWhenDoesTheFSignatureExist,
    YaoObservationsAboutTheFSignature,   
    EnescuYaoTheLowerSemicont,      
    VonKorffFSigOfAffineToric,  
    BlickleSchwedeTuckerFSigPairs1,  
    HashimotoNakajimaGeneralizedFSignature,
    SannaiDualFSignature,  
    HernandezJeffriesLocalOkunkovBodies,
    PolstraTuckerCombinedApproach,  
    DeStefaniPolstraYaoGlobalizing,  
    CarvajalRojasStablerFiniteMapsSplittingPrimes,
    BrennerJeffriesNunezBetancourtQuantifyingSingularities,  
    PolstraATheoremAboutMCM,  
    MartinTheNumberOfTorsionDivisors}.  One notable application of $F$-signature has been towards the \'etale fundamental group, a problem we now discuss the history of.

    In \cite[Question 26]{KollarNewExamples}, and partially inspired by \cite{GurjarZhangPi1OfTheSmoothPointsOfLogDelPezzo,FujikiKobayashiLuOnFundamentalGroupOfCertainOpenNormalSurfaces,KeelMcKernanRationalCurvesSurfaces} see also \cite{ZhangTheFundamentalGroupOfTheSmoothPartOfALogFanoVariety}, Koll\'ar conjectured that for any log terminal (klt) singularity $x \in X$ over $\bC$, the fundamental group of the link of the singularity\footnote{Intersect a small Euclidean sphere with the singularity.} is finite.  While this was recently solved by Braun in \cite{BraunTheLocalFundamentalGroupOfAKLT}, earlier in \cite{XuEtaleFundamentalGroup}, C.~Xu showed that the \'etale fundamental group of the punctured spectrum of a KLT singularity is finite by using boundedness results of \cite{HaconMcKernanXuACCForLCThresholds}, also see \cite{GrebKebekusPeternellEtaleFundamentalGroupsKLT,TianXuFinitenessOfFundamentalGroups}.  In \cite{CarvajalRojasSchwedeTuckerFundamentalGroup}, inspired by Xu's result, the authors proved that the \'etale fundamental group of the regular locus of a strongly $F$-regular singularity is finite, see also \cite{BhattCarvajalRojasGrafSchwedeTuckerFunGroups,StibitzEtaleCoversAndLocalFunGroups,BrennerJeffriesNunezBetancourtQuantifyingSingularities,SmirnovJeffriesTransformationRuleForNaturalMult,BhattGabberOlssonFinitenessOfFundamental}.  There, the key point was the following facts: for a Noetherian complete local ring $(R, \m)$ with algebraically closed residue field, we have
    \begin{enumerate}
        \item $s(R) \leq 1$;
        \item If $R$ is strongly $F$-regular, then $s(R) > 0$ (see \cite{AberbachLeuschke});
        \item If $R \subseteq S$ is quasi-\'etale (i.e., \'etale in codimension 1) and $R$ is strongly $F$-regular, then $s(S) = [K(S) : K(R)] \cdot s(R)$  (see \cite{VonKorffToricVarieties,CarvajalRojasSchwedeTuckerFundamentalGroup}).    
    \end{enumerate}
    These three facts together easily give an explicit bound on the size of the fundamental group of the regular locus $U$ of $\Spec R$.
    \[
        |\pi_1^{\text{\'et}}(U)| \leq 1/s(R).
    \]
    Normalized volume, which can be thought as a characteristic zero invariant closely related to $F$-signature, see \cite{LiLiuXuAGuidedTourNormalizedVolume}, also provides an explicit upper bound on the size of the fundamental group of the regular locus (which, since it is finite, equals the \'etale fundamental group).  Indeed, \cite{LiuXuKStabilityOfCubicThreefolds} established that the normalized volume is bounded above by $d^d$ where $d = \dim R$, the analog of (a) above.  The normalized volume variant of (b), for klt singularities, is positive by \cite[Corollary 3.4]{LiMinimizingNormalizedVolumes} \cf \cite{BlumExistenceOfValuationsWithSmallestNormalized}.  Finally, \cite{XuZhuangUniquenessOfTheMinimizerOfTheNormalizedVolume} proved that normalized volume satisfies a transformation rule analogous to (c) above.  Note differential signature and a similar strategy can also be used to prove finiteness of local \'etale fundamental groups for certain classes of singularities including summands of regular rings in characteristic zero (or strongly $F$-regular rings in characteristic $p > 0$), see \cite{SmirnovJeffriesTransformationRuleForNaturalMult,BrennerJeffriesNunezBetancourtQuantifyingSingularities}.  Such summands of regular rings include many klt singularities \cite{SchoutensLogTerminalSingularitiesAndVanishingTheoremsViaNonstandard,ZhuangDirectSummandsOfKLT}, also see \cite[Remark 8.4]{MalloryBignessOfTheTangentBundle}.

    The goal of this paper is to extend these results, and many basic facts about Hilbert-Kunz multiplicity and $F$-signature, into mixed characteristic.  The first obstruction one sees to mimicking the characteristic $p > 0$ approach, of course, is that there is no Frobenius map.  However, in mixed characteristic, there is an analog of the perfection map 
    \[ 
        R \to \varinjlim_e R^{1/p^e} =: R_{\perf}
    \] 
    from positive characteristic.  Indeed, suppose $R$ is a complete local Noetherian domain of mixed characteristic $(0, p> 0)$, that is $R$ has characteristic zero, while the residue field has characteristic $p > 0$.  Further suppose that $R$ has perfect residue field and that  
    \[
        A = W(k)\llbracket x_2, \dots, x_d\rrbracket \subseteq R
    \]
    is a Noether-Cohen normalization (in other words, $R$ is a finite $A$-module).  Set $\Ainfty$ to be the $p$-adic completion of the integral extension $A[p^{1/p^{\infty}}, x_2^{1/p^{\infty}}, \dots, x_d^{1/p^{\infty}}]$. Then $\Ainfty$ is a perfectoid ring.  Then Bhatt-Scholze \cite{BhattScholzepPrismaticCohomology} defines the \emph{perfectoidization} of $R \otimes_A \Ainfty$:
    \[
        R^{{\Ainfty}}_{\perfd} := \big(R \otimes_A \Ainfty \big)_{\perfd}.    
    \]
    This ring is, roughly speaking, the smallest perfectoid ring that $R \otimes_A \Ainfty$ maps to.

    There is another obstruction however.  The ring $R^{{\Ainfty}}_{\perfd}$ is non-Noetherian and although, if we want to mimic \autoref{eq.SignatureViaIeColength}, we can define the ideal 
    \[
        I_{\infty} = \big\{ z \in R^{{\Ainfty}}_{\perfd} \;|\; R \xrightarrow{1 \mapsto z} R^{{\Ainfty}}_{\perfd} \text{ does not split} \big\},
    \]
    the quotient $R^{{\Ainfty}}_{\perfd}/I_{\infty}$ has infinite length (even in equal characteristic $p > 0$).  In particular, there is no obvious limit to take.  To solve this issue we use normalized length $\lambda_{\infty}$ in the sense of Faltings \cite{faltings2002almost}, see also \cite{GabberRameroFoundationsAlmostRingTheory}.

    We define \emph{perfectoid signature of $R$} with respect to the choice of regular system of parameters of $A$, $\underline{x} = (x_1 = p, x_2, \dots, x_d)$ to be
    \[
        s^{\underline{x}}_{\perfd}(R) := \lambda_{\infty}(R^{{\Ainfty}}_{\perfd} / I_{\infty}).
    \]    
    This definition is in analogy with our alternate definition of $F$-signature from \autoref{eq.SignatureViaIeColength}.  Note in equal characteristic $p > 0$, $R^{{\Ainfty}}_{\perfd}$ coincides with $R_{\perf} = \bigcup_{e > 0} R^{1/p^e}$ and $I_{\infty} = \bigcup_{e > 0} I_e^{1/p^e}$, see \autoref{prop.I_inftyeEqualstoUnion}.  Still in characteristic $p > 0$, we then show in \autoref{thm.SignatureHKviaNormalizedLength} that  
    \[
        s(R) = s^{\underline{x}}_{\perfd}(R).
    \]
    In particular, our perfectoid signature generalizes the $F$-signature from characteristic $p > 0$.


    We prove the following results about perfectoid signature.  All of our results, or rather their analogs, were previously known in positive characteristic.  While our methods in most cases could be adapted to the equal characteristic $p > 0$ case using perfection instead of perfectoidization, this would not be an efficient way to prove these results.

    \begin{theoremA*}[{\autoref{prop.sFboundby1}, \autoref{cor.PerfdSignature1IsRegular}, \cf \cite{HunekeLeuschkeTwoTheoremsAboutMaximal,WatanabeYoshidaHKMultAndInequality,HunekeYaoMinimalHKMultiplicity}}]
        Suppose $R$ is a Noetherian complete local domain of mixed characteristic $(0, p> 0)$ with perfect residue field.  We then have that $0 \leq s_{\perfd}^{\underline{x}}(R) \leq 1$.  Furthermore $s_{\perfd}^{\underline{x}}(R) = 1$ (for one, or equivalently, all choices of $\underline{x}$) if and only if $R$ is regular.
    \end{theoremA*}

    We next recall the notion of BCM-regular rings.  We say that $R$ is \emph{weakly BCM-regular} if for every perfectoid big Cohen-Macaulay $R^+$-algebra $B$ the map $R \to B$ is split (equivalently the map is pure, since $R$ is complete).  If $R$ is $\bQ$-Gorenstein, we call $R$ \emph{BCM-regular}.  More generally, if $\Delta \geq 0$ is a $\bQ$-divisor such that $K_R + \Delta$ is $\bQ$-Cartier, we say that $(R, \Delta)$ is \emph{BCM-regular} if $R \to R^+(f^* \Delta) \otimes_{R^+} B$ is split for every perfectoid big Cohen-Macaulay $R^+$-algebra $B$ (here $f : \Spec R^+ \to \Spec R$ is the canonical map).  We refer to \cite{MaSchwedeSingularitiesMixedCharBCM} and \autoref{sec.BCM-regularrityPerfdSignature} for more discussion, see also \cite{PerezRG}.  Note that perfectoid big Cohen-Macaulay algebras exist in mixed characteristic by \cite{AndreWeaklyFunctorialBigCM}; indeed even $\widehat{R^+}$ is itself a perfectoid big Cohen-Macaulay algebra by \cite{BhattAbsoluteIntegralClosure} (this latter fact will not be used in any crucial way in this paper).

    \begin{theoremB*}[{\autoref{prop.dH-BCMRegularImpliesPerfdPositive}, \autoref{prop.SignaturePositiveImpliesPerturbationBCMRegular}, \cf \cite{AberbachLeuschke,HunekeLeuschkeTwoTheoremsAboutMaximal}}]
        Suppose that $R$ is a Noetherian complete local domain of mixed characteristic $(0, p> 0)$ with perfect residue field and that $(R, \Delta)$ is a BCM-regular pair (for instance if $R$ is $\bQ$-Gorenstein and BCM-regular).  Then $s_{\perfd}^{\underline{x}}(R) > 0$.  
        
        Conversely if $s_{\perfd}^{\underline{x}}(R) > 0$ then for every $0 \neq g \in R$, there exists $e > 0$, such that for any perfectoid big Cohen-Macaulay $R^+$-algebra $B$, the map $R \xrightarrow{1 \mapsto g^{1/p^e}} B$ splits.  In particular, $R$ is weakly BCM-regular.
    \end{theoremB*}

    We expect that if $s_{\perfd}(R) > 0$ then there exists a $\bQ$-divisor $\Delta \geq 0$ such that $(R, \Delta)$ is BCM-regular, but we do not know how to show this.  For additional discussion and other intermediate notions of (strong) BCM-regularity, and their comparison with the condition $s^{\underline{x}}_{\perfd}(R) > 0$, see \autoref{subsec.StrongBCMRegularity}.

    We also obtain the following transformation rule for perfectoid signature under quasi-\'etale maps.

    \begin{theoremC*}[{\autoref{thm.TransformationRule}, \cf \cite{HunekeLeuschkeTwoTheoremsAboutMaximal,VonKorffToricVarieties,CarvajalRojasSchwedeTuckerFundamentalGroup}}]
        Suppose $(R, \m) \subseteq (S, \frn)$ is a split\footnote{Note that $s_{\perfd}^{\underline{x}} > 0$ implies $R\to R^+$ is pure, and thus $R \to S$ splits automatically.} quasi-\'etale finite extension of Noetherian complete local domains of mixed characteristic $(0, p> 0)$ with algebraically closed residue field.  Then 
        \[
            s_{\perfd}^{\underline{x}}(S) = [K(S) : K(R)] \cdot s_{\perfd}^{\underline{x}}(R).
        \]
    \end{theoremC*}
    Furthermore, if the regular system of parameters of the Noether-Cohen normalization $A \subseteq R$ is chosen with more care, then the condition that $R \subseteq S$ is quasi-\'etale above can be weakened to simply requiring that:
    \begin{enumerate}
        \item $R \subseteq S$ is finite and split, 
        \item there exists $\phi \in \Hom_R(S, R)$ generating the $\Hom$-set as an $S$-module, and
        \item $\phi(\frn) \subseteq \m$.
    \end{enumerate}
    In the quasi-\'etale case we take $\phi$ to be the trace map.  This improvement gives us a transformation rule for perfectoid signature even for cyclic covers of index divisible by $p > 0$, \cf \cite{CarvajalRojasFiniteTorsors} for even stronger results in characteristic $p > 0$.

    As a consequence of Theorems A, B, and C we obtain the following result.
    \begin{theoremD*}[{\autoref{thm.FundamentalGroupApplication}, \autoref{thm.DivisorClassGroupApplication}, \cf \cite{CarvajalRojasSchwedeTuckerFundamentalGroup, CarvajalRojasFiniteTorsors}}]
        Suppose that $R$ is a Noetherian complete local domain of mixed characteristic $(0,p>0)$ with algebraically closed residue field.  If $U \subseteq \Spec R$ is the nonsingular locus, then  
        \[
            |\pi_1^{\text{\'et}}(U)| \leq {1 \over s_{\perfd}^{\underline{x}}(R)}
        \]
        where we interpret $1/0$ as infinity in the case that $s_{\perfd}^{\underline{x}}(R) = 0$ (for instance, if $R$ is not weakly BCM-regular).  
        Furthermore, 
        \[
            |\mathrm{Cl}(R)_{\mathrm{tors}}| \leq {1 \over s_{\perfd}^{\underline{x}}(R)}
        \]
        where $\mathrm{Cl}(R)_{\mathrm{tors}}$ is the torsion part of the divisor class group of $R$.
        In particular, if $(R, \Delta)$ is BCM-regular for some $\Delta \geq 0$, then $\pi_1^{\text{\'et}}(U)$ and $\mathrm{Cl}(R)_{\mathrm{tors}}$ are finite. 
 Finally $|\pi_1^{\text{\'et}}(U)|$ is not divisible by $p > 0$.
    \end{theoremD*}

    Using the work of Stibitz \cite{StibitzEtaleCoversAndLocalFunGroups} we can use this to recover a mixed characteristic version of the main result of \cite{GrebKebekusPeternellEtaleFundamentalGroupsKLT} (\cf \cite{BhattGabberOlssonFinitenessOfFundamental,BhattCarvajalRojasGrafSchwedeTuckerFunGroups}), see \autoref{thm.GrebKebekusPeternellVersion}.  
    
    As another corollary of this result, by taking cones, we obtain finiteness of the \'etale fundamental group of big open sets in log Fano pairs $(Y, \Delta)$ that are globally $\bigplus$-regular, in the sense of \cite{BMPSTWW-MMP} \cf \cite{TakamatsuYoshikawaMMP}, and remain so after small perturbations of $\Delta$, see \autoref{thm.FinitenessFunGroupsForBigOpenInLogFano}.

    We also introduce a mixed characteristic version of Hilbert-Kunz multiplicity.  We show that, if $(R, \m)$ is a complete Noetherian local domain of characteristic $p > 0$ and $I \subseteq R$ is $\m$-primary then
    \[
        e_{\HK}(I) = \lambda_{\infty}\big(R_{\perf} / I R_{\perf}\big)
    \]
    in \autoref{thm.SignatureHKviaNormalizedLength}.  Hence for a complete local domain $(R, \m)$ with perfect residue field of characteristic $p>0$ and Noether-Cohen normalization $W(k)\llbracket x_2, \dots, x_d \rrbracket = A \subseteq R$ as before, we \emph{define} the \emph{perfectoid Hilbert-Kunz multiplicity of an $\m$-primary ideal $I \subseteq R$} to be:
    \[
        e^{\underline{x}}_{\perfd}(I) := \lambda_{\infty} \big(R^{\Ainfty}_{\perfd}/ I R^{\Ainfty}_{\perfd}\big).
    \] 
    In the case that $I = \m$, we write $e^{\underline{x}}_{\perfd}(R)$ for $e^{\underline{x}}_{\perfd}(\m)$.

    We obtain the following result characterizing regularity in mixed characteristic, \cf \cite{WatanabeYoshidaHKMultAndInequality} and \cite{HunekeYaoMinimalHKMultiplicity} in characteristic $p>0$.

    \begin{theoremE*}[{\autoref{prop.eFHKgreaterthan1}, \autoref{prop.eFHKequals1regular}, \cf \cite{KunzOnNoetherianRingsOfCharP,WatanabeYoshidaMinimalRelative}}]
    Suppose $(R, \m)$ is a Noetherian complete local domain of mixed characteristic $(0, p> 0)$ with perfect residue field.  Then $e_{\perfd}^{\underline{x}}(R) \geq 1$ and we have equality (for one, or equivalently, all choices of $\underline{x}$) if and only if $R$ is regular.   
\end{theoremE*}

    We can also use perfectoid Hilbert-Kunz multiplicity to characterize BCM closure and extended plus closure of ideals in analogy with the fact that Hilbert-Kunz multiplicity characterizes tight closure \cite[Theorem 8.17]{HochsterHunekeTC1}, see also \cite[Section 5]{HunekeHilbertKunzMultiplicityAndFSignature}.

\begin{theoremF*}[{\autoref{prop.HKVsBCMClosureVsepfClosure}, \cf \cite{HochsterHunekeTC1}}]
    Suppose $(R, \m)$ is a Noetherian complete local domain of mixed characteristic $(0, p> 0)$ with perfect residue field. Then for every $\underline{x}$ and every pair of $\m$-primary ideals $I\subseteq J$, the following are equivalent: 
    \begin{enumerate}
        \item $e^{\underline{x}}_{\perfd}(I) = e^{\underline{x}}_{\perfd}(J)$.
        \item There exists a perfectoid big Cohen-Macaulay $R^+$-algebra $B$ such that $IB=JB$.
        \item $I^{epf}=J^{epf}$ where $(-)^{epf}$ denotes (full) extended plus closure.
    \end{enumerate}
    In particular, $R$ is weakly BCM-regular if and only if every $\m$-primary ideal $I$ is extended plus closed (\ie, $I^{epf}=I$).
    \end{theoremF*}
    
    Several other common facts about Hilbert-Kunz multiplicity also generalize to our mixed characteristic setting.  For instance, we show perfectoid Hilbert-Kunz multiplicity is bounded above by the Hilbert Samuel multiplicity $e_{\perfd}^{\underline{x}}(I) \leq e(I)$,  with equality when $I$ is generated by a system of parameters of $R$, see \autoref{cor.perfdHKparameterideal}.

    Finally, in analogy with the recent work on $F$-rational and relative $F$-rational signature found in \cite{HochsterYaoFratSignature,SmirnovTuckerRelativeFSignature} \cf \cite{SannaiDualFSignature}, we define the perfectoid rational and relative rational signature respectively as
    \begin{equation*}
        \begin{array}{crclc}
            \;\;\;\;\; & s^{\underline{x}}_{\ratperfd}(R) & = & {\displaystyle\inf_{\substack{\underline{y} \\ z \in ((\underline{y}):_R \m) \setminus (\underline{y})} } e^{\underline{x}}_{\perfd}(\underline{y}) - e^{\underline{x}}_{\perfd}(\underline{y},z)} & \text {and}\\\\
            \;\;\;\;\; &  s^{\underline{x}}_{\relperfd}(R)& = &{\displaystyle \inf_{\substack{\underline{y} \\ \underline{y}\subsetneq J \subseteq ((\underline{y}):_R \m)}} \frac{e^{\underline{x}}_{\perfd}(\underline{y}) - e^{\underline{x}}_{\perfd}(J)}{\length_R(J/(\underline{y}))}}
        \end{array}
    \end{equation*}
    We show they both characterize 
    BCM-rationality in the sense of \cite{MaSchwedeSingularitiesMixedCharBCM}.  Recall a Noetherian complete local ring $(R,\m)$ of dimension $d$ and residue characteristic $p>0$ is BCM-rational if and only if $R$ is Cohen-Macaulay and for every perfectoid big Cohen-Macaulay algebra $B$, we have that $H^d_{\m}(R) \to H^d_{\m}(B)$ is injective.
    
    \begin{theoremG*}[{\autoref{thm.sratpositiveandbcmregular}, \cf \cite{SmirnovTuckerRelativeFSignature}}]
        Suppose $R$ is a Noetherian complete local domain of mixed characteristic $(0,p>0)$ with perfect residue field.  The following are equivalent:
        \begin{enumerate}
            \item $R$ is BCM-rational; 
            \item $s^{\underline{x}}_{\ratperfd}(R) > 0$; 
            \item $s^{\underline{x}}_{\relperfd}(R) > 0$. 
        \end{enumerate}
    \end{theoremG*}
    
    Finally, some open questions on perfectoid signature and further directions for research are described in \autoref{sec.FurtherQuestions}.

\fakesubsection{Acknowledgements} The authors would like to thank Bhargav Bhatt for numerous valuable conversations and a great deal of encouragement.  We would also like to thank Neil Epstein, Anurag Singh and Ilya Smirnov for valuable conversations.  We thank Bhargav Bhatt, Javier Carvajal-Rojas, Luis Nu\~nez Betancourt, and Shravan Patankar for valuable comments on a previous draft.  Finally, we thank the anonymous referee who provided a great deal of valuable feedback and pointed out many typos.

\section{Preliminaries}

All rings considered are commutative with unity.  We frequently leave the Noetherian category however.  Note, for a complete Noetherian local domain $(R, \m)$ of mixed characteristic $(0, p > 0)$ and residue field $k=R/\m$, we use the term \emph{Noether-Cohen normalization} to describe a subring $V\llbracket x_1, \dots, x_d \rrbracket =: A \subseteq R$ with $R$ a finite $A$-module and $V$ a complete and unramified mixed characteristic DVR (typically we can take $V = W(k)$ since our residue field $k$ is usually perfect, see \cite[Theorem 84]{MatsumuraCommutativeAlgebra} or \cite[Corollary 22]{Hochster.StructureOfCompleteLocal}).  These exist by the Cohen structure theorem, \cf \cite[\href{https://stacks.math.columbia.edu/tag/0323}{Tag 0323}]{stacks-project}, \cite{Hochster.StructureOfCompleteLocal}, and \cite[Section 37]{MatsumuraCommutativeAlgebra}.

\subsection{Pure maps, split maps, and trace}    
\label{subsec:PureSplitMap}
We first recall that a map of $R$-modules $N\to M$ is pure if after tensoring with any $R$-module $K$, the map $N\otimes_RK\to M\otimes_RK$ is injective. Clearly, a split map is pure and a pure map is injective.

     Suppose $(R,\m)$ is a Noetherian local ring and $M$ is an $R$-module, not necessarily finitely generated. We collect the following basic facts.
    \begin{enumerate}
        \item A map $R\to M$ is pure if and only if the induced map $E\to E\otimes_RM$ is injective, where $E:=E_R(k)$ denotes
the injective hull of the residue field $k=R/\m$, see \cite[Lemma (2.1) (e)]{HochsterHunekeApplicationsofBigCM}.
\item If $R$ is complete, then a map $R\to M$ is pure if and only if it is split, see \cite[Lemma 1.2]{FedderFPureRat}.
    \end{enumerate}

We next recall that a Noetherian local ring $(R,\m)$ is approximately Gorenstein if there exists a decreasing
sequence of $\m$-primary ideals $I_1\supseteq I_2\supseteq \cdots \supseteq I_t \supseteq \cdots $ such that every $R/I_t$
is a Gorenstein ring and that for every $n>0$, there exists $t$ such that $I_t\subseteq \m^n$. Evidently, a Gorenstein local ring is approximately Gorenstein, since we may take $I_t=(x_1^t, \dots, x_d^t)$ for any system of parameters $x_1,\dots,x_d$ of $R$. But the approximately Gorenstein condition is much weaker: in fact, every Noetherian complete reduced ring is approximately Gorenstein, see \cite[Theorem 1.6]{HochsterCyclicPurity}. 

In this paper,  we will need the following well-known fact about splitting criteria for approximately Gorenstein rings, and we will typically apply this when $(R,\m)$ is a complete local domain. 
    
    \begin{lemma}
    \label{lem.SplittingApproxGorenstein}
        Let $(R,\m)$ be a Noetherian local ring that is approximately Gorenstein and let $M$ be an $R$-module. Consider the submodule: 
        \[ N := \{x \in M \;|\; R \xrightarrow{1 \mapsto x} M \text{ is not pure} \}. \]
        Then we have $$N= \bigcup_t (I_tM :_M u_t),$$
        where $u_t$ is a socle representative of $R/I_t$.
    \end{lemma}
    \begin{proof}
Let $E=E_R(k)$ denote the injective hull of the residue field $k=R/\m$. The point is that we have $$E=\bigcup_t \Ann_E{I_t}\cong \varinjlim_{t} E_{R/I_t} \cong \varinjlim_{t} R/I_t.$$ 
Since $R\to M$ is pure if and only if $E\to M\otimes_RE$ is injective, it follows that $R\to M$ is pure if and only if $R/I_t \to M/I_tM$ is injective for all $t$, which happens if and only if the image of $u_t$ in $M/I_tM$ is nonzero for all $t$. Thus $x\in N$ if and only if $u_tx\in I_tM$ for some $t$, i.e, $x\in \bigcup_t (I_tM :_M u_t)$. 
    \end{proof}
    
    \begin{remark}
    \label{rmk.ApproxGorensteinChain}
    In \autoref{lem.SplittingApproxGorenstein}, the submodules $(I_tM :_M u_t)$ of $M$ actually form an ascending chain. To see this, note that we have $R$-linear maps $R/I_t\cong \Ann_EI_t\subseteq \Ann_EI_{t+1}\cong R/I_{t+1}$. Suppose it sends $1\in R/I_t$ to $r_t\in R/I_{t+1}$, then we must have $r_tI_t\subseteq I_{t+1}$, and up to units $u_{t+1}=r_tu_t$. Thus if $u_tx\in I_tM$ for some $x\in M$, then $u_{t+1}x=r_tu_tx\in r_tI_tM\subseteq I_{t+1}M$.
    \end{remark}

We also record the following lemma, which is well-known to experts, but we do not know a good reference.

\begin{lemma}
    \label{lem.IinfinityForFlat}
    Suppose $(R, \m)$ is a Noetherian local ring and $M$ is a flat $R$-module.  Consider the submodule
    \[
        N := \{x \in M \;|\; R \xrightarrow{1 \mapsto x} M \text{ is not pure.} \}
    \]
    Then $N = \m M$.  
\end{lemma}
\begin{proof}
First of all, it is clear that $\m M\subseteq N$ since if $x\in \m M$, then the map $R\xrightarrow{1 \mapsto x} M$ induces the zero map $R/\m \to M/\m M$ and thus is not pure.

    For the reverse containment, suppose that $x \in M \setminus \m M$ and consider the map $R \xrightarrow{1 \mapsto x} M$.  Tensoring with $k = R/\m$ we see the induced map $k \xrightarrow{1 \mapsto x} k \otimes M = M/\m M$ is nonzero and hence injective.  Next, since $M$ is flat, if $E := E_R(k)$ is the injective hull, we see that $k \otimes M \to E \otimes M$ is injective.  We then have the following commutative diagram:
    \[
        \xymatrix{
            k \ar@{^{(}->}[d]_{1 \mapsto x} \ar[r]^{\subseteq} & E \ar[d]^{1 \mapsto x}\\
            k \otimes M \ar@{^{(}->}[r] & E \otimes M.
        }
    \]  
    Since the image of $k$ in $E$ is the socle, it follows that the right vertical map $E \xrightarrow{1 \mapsto x} E \otimes M$ injects on the socle, and so it itself injects as well.  Thus $R \xrightarrow{1 \mapsto x} M$ is pure, as desired.
\end{proof}

A common way to induce splittings is via trace.  Recall that if $A \to B$ is a map of rings and $B$ is a finite free $A$-module, then for each $b \in B$, we can view the multiplication-by-$b$ map as an $A$-linear transformation on $B$ and so associate to $b$ the trace of the corresponding matrix.  We denote this value by $\Tr(b)$, \cite[\href{https://stacks.math.columbia.edu/tag/0BSY}{Tag 0BSY}]{stacks-project}. 

Now if $R\to S$ is a finite extension of Noetherian normal domains with fraction fields $K$ and $L$ respectively, then $\Tr_{L/K}(-)$ sends $S$ to $R$ and thus induces a map $\Tr\in \Hom_{R}(S, R)$ (see \cite[Pages 33-35]{HochsterFoundations}).
We collect some well known facts about trace for future use.  

\begin{lemma}
\label{lem.TraceForNormalDomains}
    Suppose that $R \subseteq S$ is a finite extension of normal Noetherian domains.  
    
    \begin{enumerate}
        \item $\Tr \in \Hom_R(S, R) \cong \omega_{S/R}$, viewed as a section, induces the (effective) ramification divisor $\Ram$ of $\Spec (S) \to \Spec (R)$.  
    In particular, if $R \subseteq S$ is \'etale in codimension one, then $\Tr$ generates $\Hom_R(S, R)$ as an $S$-module. 

    \item $\Tr : S \to R$ extends to $\omega_{S/R} \cong S(\Ram) \to R$.  By twisting by $\omega_R$ and reflexifying, the resulting map also coincides with the Grothendieck dual to the inclusion $R \subseteq S$ on canonical modules, $\omega_S \to \omega_R$, up to a unit of $S$.
    \item If $I \subseteq R$ is a radical ideal, then $\Tr(\sqrt{IS}) \subseteq I$.
    \end{enumerate}
\end{lemma}
\begin{proof}
    For the first statement, see for instance \cite[Proposition 4.8]{SchwedeTuckerTestIdealFiniteMaps} as well as \cite{MoriyaTheorieDerivationen,SchejaStorchUberSpurfunktionen,deSmitTheDifferentAndDifferentials}.

    For the second statement, for any map $\phi \in \Hom_R(S, R)$, there is an induced extension $\phi : S(D_{\phi}) \to R$ where $D_{\phi}$ is the effective divisor associated to $\phi$ and furthermore $\phi$ generates $\Hom_R(S(D_{\phi}), R)$ as an $S$-module (one can check this in codimension 1 where $R$ becomes a DVR).  Hence the  obtained twisted $\omega_S \cong (S(\Ram) \otimes \omega_R)^{**} \to \omega_R$ also generates the corresponding $\Hom$-set.  But the Grothendieck dual map $\omega_S = \Hom_R(S, \omega_R) \xrightarrow{\text{evaluation at $1$}} \omega_R$ also generates $\Hom_R(\omega_S, \omega_R)$ (again, this can be verified explicitly in codimension 1) and so the claim follows as $\Hom_R(\omega_S, \omega_R) \cong S$ as $S$ is ($\mathrm{S}_2$).

    For the third statement, 
see \cite[Lemma 2.10]{CarvajalRojasSchwedeTuckerFundamentalGroup} or \cite[Lemma 9]{Speyer.FrobeniusSplitSubvarsPullBackInAlmostAll}.
\end{proof}

\subsection{Normalized length}
\label{section:normlength}

In this section we recall the notion of normalized length, introduced by Faltings, and some results of \cite{faltings2002almost,ShimomotoFrobeniusActionLocalCohomology,GabberRameroFoundationsAlmostRingTheory}.

In what follows, $(A,\m_A,k)$ will be a complete regular local ring of dimension $d$ with perfect residue field $k$ of characteristic $p>0$.  We will only consider $A$ that is of one of the following two forms.  
\begin{itemize}
    \item{} If $\mathrm{char} A = p > 0$, the \emph{equicharacteristic $p > 0$ case}, then we will set $A = \Lambda\llbracket x_1,\dots,x_d\rrbracket$ where $\Lambda = k$.  Notice that $x_1, \dots, x_d$ is a regular system of parameters.
    \item{} If $\mathrm{char} A = 0$, the \emph{unramified mixed characteristic case}, then we define $A = \Lambda\llbracket x_2, \dots, x_d \rrbracket$ where $\Lambda = W(k)$.  In this case we set $x_1 = p$ and notice that $x_1, \dots, x_d$ form a regular system of parameters.
\end{itemize}

In either case we define 
\[
    A_n:=A[x_1^{1/{p^n}},x_2^{1/{p^n}},\dots,x_d^{1/{p^n}}]. 
\] 
We observe that $(A_n\mid n\in\mathbb{N})$ forms a filtered system, and so, by taking the filtered colimit, we define 
\[ 
    \Ainfty^{nc}:=\varinjlim_{n} A_n.
\]  
Here $(-)^{nc}$ indicates that it is not necessarily $p$-complete.  Note that in the equicharacteristic $p > 0$ case we have that $\Ainfty^{nc} = A_{\perf}$.  However, in mixed characteristic, $\Ainfty^{nc}$ \emph{depends} on our choice of regular system of parameters ${\underline{x}}= p, x_2, \dots, x_d$.

We denote by $\Ainfty^{nc}\hyphen\text{Mod}_{\m_A}$ the full subcategory of $\Ainfty^{nc}\hyphen\Mod$ consisting of all the $\Ainfty^{nc}$-modules $M$ which are $\m_A$-power torsion, i.e. annihilated by $\m_A^N$ for some $N\in \mathbb{N}$.

\begin{definition}[Faltings' normalized length]
    \label{def.NormalizedLength}
    We define the normalized length in three steps.  First for finitely presented modules, then finitely generated modules, and finally in the general case.
    Suppose first that $M\in \Ainfty^{nc}\hyphen\text{Mod}_{\m_A}$ is finitely presented and so we can write $M\simeq M_n\otimes_{A_n}\Ainfty^{nc}$ for some finitely presented $A_n$-module $M_n$, which is also $\m_A$-power torsion, thus of finite length. We then define
    \[
        \lambda^{**}_\infty(M)\coloneqq ({1}/ p^{nd})\length_A(M_n).
    \]
    Next, assume that $M\in \Ainfty^{nc}\hyphen\text{Mod}_{\m_A}$ is finitely generated, we define
    \[
        \lambda^*_\infty(M)\coloneqq\inf\{\lambda^{**}_\infty(N)\mid N\twoheadrightarrow M,\ N\ \text{ is finitely presented} \}.
    \]
    In general, if $M$ is any object in $\Ainfty^{nc}\hyphen\text{Mod}_{\m_A}$, then we set
    \[ 
        \lambda_\infty(M)\coloneqq\sup\{\lambda^{*}_\infty(N)\mid N\subset M,\ N\ \text{is finitely generated} \}.
    \]
    We call $\lambda_\infty(M)$ the \emph{normalized length of $M$}.
\end{definition}

\begin{remark}
    Normalized length can be defined in more general settings. For more details, see \cite[Remark 14.5.69]{GabberRameroFoundationsAlmostRingTheory} and \cite[Appendix 2]{faltings2002almost}.
\end{remark}


\begin{proposition}[\cf {\cite[Lemma 14.5.74]{GabberRameroFoundationsAlmostRingTheory}, \cite[Lemma 2.2]{ShimomotoFrobeniusActionLocalCohomology}}]
\label{prop.NormalizedLengthexists}
    $\lambda_\infty$ is a well-defined function from $\Ainfty^{nc}\hyphen\Mod_{\m_A}$ to $\mathbb{R}\cup\{\infty \}$ and agrees with $\lambda_{\infty}^*$ and $\lambda_{\infty}^{**}$ on finitely generated and finitely presented modules respectively.
\end{proposition}

\begin{proof}
    Suppose first that $M$ is finitely presented.  We must check that different presentations of $M$ give the same normalized length. Suppose $M\simeq M_n\otimes_{A_n} \Ainfty^{nc}\simeq M_m\otimes_{A_m} \Ainfty^{nc}$. Then there exists some finite $q \ge m,n$ such that $A_q\otimes_{A_m}M_m\simeq A_q\otimes_{A_n}M_n $ and the assertion follows from the computation on finite levels since $A_q$ is a free $A_n$-module of rank $p^{(n-q)d}$.  Furthermore, if $M$ is finitely presented then $\lambda_\infty^{**}(M)=\lambda_\infty^{*}(M)$ since $M$ surjects onto itself. 

    The more difficult step is to show that when $M$ is \emph{finitely generated} then $\lambda_\infty^*(M)=\lambda_\infty(M)$. This is in \cite[Lemma 14.5.74]{GabberRameroFoundationsAlmostRingTheory} or \cite[Lemma 2.2]{ShimomotoFrobeniusActionLocalCohomology}, but we reproduce the proof here for the convenience of the reader. By definition, $\lambda_\infty(M)\ge \lambda_\infty^*(M)$. Hence it suffices to show for any $M'\subset M$ finitely generated submodule, we have $\lambda_\infty^*(M')\le\lambda_\infty^*(M)$. We choose finite set of generators $\Sigma'\subset M'$ and $\Sigma\subset M$ with $\Sigma'\subset \Sigma$. We clearly have $\length_A(\Sigma' A_n)\le \length_A(\Sigma A_n)$ for $n\gg 0$, hence the proof is complete by \autoref{lem.Normalizedlengthfglimit} below.
\end{proof}

\begin{lemma}[\cf {\cite[Appendix 2]{faltings2002almost}, \cite[Lemma 2.3]{ShimomotoFrobeniusActionLocalCohomology}}]
    \label{lem.Normalizedlengthfglimit}
    Let $M\in \Ainfty^{nc}\hyphen\Mod_{\m_A}$ be finitely generated, $\Sigma$ any finite set of generators of $M$.  For each $n \gg 0$ let $M_n=\Sigma A_n \subseteq M$ be the $A_n$-submodule of $M$ generated by $\Sigma$.  We then have that 
    \[ 
        \lambda^*_\infty(M)=\lim_{n\to \infty}\lambda_\infty^*(M_n\otimes_{A_n}\Ainfty^{nc})=\lim_{n\to \infty}(1/p^{nd})\length_A(M_n).
    \]
\end{lemma}
\begin{proof}
Note that $M_{n+m}$ is a quotient of $M_n\otimes_{A_n}A_{n+m}$ and the sequence $\lambda_\infty^*(M_n\otimes_{A_n}\Ainfty^{nc})$ is decreasing, so the limit makes sense. The first equality is done in \cite[Lemma 2.3]{ShimomotoFrobeniusActionLocalCohomology} and a more general proof of the second equality is done by \cite[Appendix 2, Claim]{faltings2002almost}. However, in our setup, since $M_n$ is finitely generated over $A_n$, it is finitely presented, so using the previous consequence of $\lambda_\infty^*(-)=\lambda_\infty^{**}(-)$ for finitely presented modules we see that the terms in the limit actually coincide.

We now sketch the first equality. It suffices to show that $\{ M_n\otimes_{A_n}\Ainfty^{nc}\}$ is cofinal with the images of the set of finitely presented modules $N$ surjecting onto $M$. But this is clear since for every finitely presented $N\simeq N_m\otimes_{A_m}\Ainfty^{nc}$, we may choose $n$ large so that $M_n$ contains the image of $N_m$.
\end{proof}

\begin{corollary}
\label{cor.Normalizedlengthfglimit}
Let $M\in \Ainfty^{nc}\hyphen\Mod_{\m_A}$ be finitely generated. Assume $\{ N_i\mid i\in I\}$ is a filtered system of finitely generated modules in $\Ainfty^{nc}\hyphen\Mod_{\m_A}$ with surjective transition maps such that $\colim_{i\in I}N_i\simeq M$. Then we have
$$\lambda_\infty^*(M)=\inf_{i\in I}\lambda_\infty^*(N_i).$$
\end{corollary}
\begin{proof}
Assume without loss of generality that $i_0$ is the smallest index in $I$. We fix a surjection $(\Ainfty^{nc})^{\oplus n}\twoheadrightarrow N_{i_0}$. For each $i\in I$, let $K_i$ be the kernel of the induced surjection to $N_i$. Now consider the filtered system of $\{ L_j\mid j\in J\}$ containing all finitely generated $\Ainfty^{nc}$-submodules $L_j$ of $(\Ainfty^{nc})^{\oplus n}$ such that $L_j\subset K_i$ for some $i\in I$. Then we have $\colim_{j\in J}(\Ainfty^{nc})^{\oplus n}/L_j \simeq M $. Since $(\Ainfty^{nc})^{\oplus n}/L_j$ is finitely presented, it is base changed from some finitely generated module $M_{k_j}$ over $A_{k_j}$ for some $k_j$, then by essentially the same argument as used in the proof of 
 \autoref{lem.Normalizedlengthfglimit}, we have $$\lambda_\infty^*(M)=\inf_{j\in J}\lambda^*_\infty(M_{k_j}\otimes_{A_{k_j}}A_\infty^{nc})=\inf_{j\in J} \lambda_\infty^* ((\Ainfty^{nc})^{\oplus n}/L_j).$$ 
Clearly we have $\lambda_\infty^*(M)\le \lambda_\infty^*(N_i)\le \lambda_\infty^*((\Ainfty^{nc})^{\oplus n}/L_j)$ which completes the proof.
\end{proof}

\begin{proposition}
    \label{prop.NormalizedLengthProperties}
    Continuing to use the notation from above, we have
    \begin{enumerate}
        \item 
        \label{prop.NormalizedLengthProperties.a}
        Suppose $0\to M'\to M\to M''\to 0$ is a short exact sequence in $\Ainfty^{nc}\hyphen\Mod_{\m_A}$. Then we have 
            \[ 
                \lambda_\infty(M)=\lambda_\infty(M')+\lambda_\infty(M'').
            \]
            
        \item Let $(M_i\mid i\in I)$ be a filtered system of objects in $\Ainfty^{nc}\hyphen\Mod_{\m_A}$ and suppose that either
        \begin{enumerate}
            \item all transition maps are injective, or
            \label{prop.NormalizedLengthProperties.b}
            \item all transition maps are surjective and
            \label{prop.NormalizedLengthProperties.c}
            $\lambda_\infty(M_i)<\infty$ for every $i$.
        \end{enumerate}
        Then we have
            \[ 
                \lambda_\infty(\colim_{i\in I} M_i)=\lim_{i\in I}\lambda_\infty(M_i).
            \]
    \end{enumerate}
\end{proposition}
\begin{proof}
For \autoref{prop.NormalizedLengthProperties.a}, 
we refer the reader to \cite[Appendix 2, Lemma]{faltings2002almost}. One can also find proofs in \cite[Proposition 2.8]{ShimomotoFrobeniusActionLocalCohomology} or \cite[Theorem 14.5.75]{GabberRameroFoundationsAlmostRingTheory}.

    For \autoref{prop.NormalizedLengthProperties.b} and \autoref{prop.NormalizedLengthProperties.c}, we follow the strategy as in \cite[Proposition 14.5.12(iii)(b)]{GabberRameroFoundationsAlmostRingTheory} and  \cite[Proposition 14.5.75(i)(b)]{GabberRameroFoundationsAlmostRingTheory}. Assume the situation of \autoref{prop.NormalizedLengthProperties.b}. Let $M = \colim_{i\in I} M_i$.  Certainly $\lambda_{\infty}(M) \geq \lambda_{\infty}(M_i)$ for each $M_i$ and so we have $\leq$.  Now, for any finitely generated submodule $N \subseteq M$, we have that $N \subseteq M_i$ for some $i$ and so the reverse inequality follows.

    Now assume the situation of \autoref{prop.NormalizedLengthProperties.c}. Again we let $M=\colim_{i\in I} M_i$. Clearly we have the obvious inequality
    $$ \lim_{i\in I}\lambda_\infty(M_i)\ge \lambda_\infty(M)$$
    As for the other direction, fix $\epsilon>0$. Assume without loss of generality that $i_0$ is the smallest index in $I$. Since $\lambda_\infty(M_{i_0})<\infty$, we may find some finitely generated submodule $N_{i_0}\subset M_{i_0}$ such that $\lambda_\infty(M_{i_0})-\lambda_\infty(N_{i_0})<\epsilon$. Let $N_i\subset M_i$ be the image of $N_{i_0}$ and $N=\colim_{i\in I }N_i$. Note that $(N_i\mid i\in I)$ is also a filtered system and $M_i/N_i$ is a quotient of $M_{i_0}/N_{i_0}$ for each $i\in I$. Then \autoref{prop.NormalizedLengthProperties.a} gives us that $\lambda_\infty(M_i)-\lambda_\infty(N_i)<\epsilon$ for every $i\in I$. Then by \autoref{cor.Normalizedlengthfglimit} (after identifying $\lambda_\infty(-)$ with $\lambda_\infty^*(-)$ for $N$), we have $\lambda_\infty(N)=\inf_{i\in I}\lambda_\infty(N_i)$, hence we have 
    $$\lambda_\infty(M)\ge \lambda_\infty(N)=\inf_{i\in I}\lambda_\infty(N_i)\ge \inf_{i\in I}\lambda_\infty(M_i)-\epsilon.$$
    This completes the proof.
\end{proof}

\subsection{Derived completion and flatness}

In this subsection, we briefly review the theory of derived completion for the convenience of the reader. Our treatment of derived completion follows from \cite{bhatt2013pro}, \cite{BhattMorrowScholzeIHES}, and \cite[\href{https://stacks.math.columbia.edu/tag/091N}{Tag 091N}]{stacks-project}. As for $I$-complete flatness, the standard reference is \cite[Section 4.1]{bhatt2019topological}. Throughout this section, $R$ is a commutative ring with unit and $I$ is a finitely generated ideal.

\begin{definition}
    A complex $K\in D(R)$ is called derived $I$-complete if for every $f\in I$, the derived limit 
    $$\mathbf{R}\lim (\cdots\xrightarrow{f} K\xrightarrow{f} K\xrightarrow{f} K)$$
    vanishes. A module $M$ is called derived $I$-complete if $M[0]$ is derived $I$-complete. We denote $D_{\comp}(R)$ the full subcategory of derived $I$-complete objects of $D(R)$.
\end{definition}

It can be deduced from the definition that to check derived $I$-completeness it suffices to check derived $f_i$-completeness for a generating set $\{f_1,\dots,f_n\}$ of $I$, see \cite[\href{https://stacks.math.columbia.edu/tag/091Q}{Tag 091Q}]{stacks-project}. Hence we can usually use induction and assume $I=(f)$ is principal (in fact, we will mostly work only with derived $p$-completion in this article). We start with the following description of derived completion of principal ideals.

\begin{proposition}[{\cite[Lemma 3.4.8, Lemma 3.4.9]{bhatt2013pro}}] Suppose $I=(f)$ is a principal ideal of $R$. 
    Then the forgetful functor $D_{\comp}(R)\to D(R)$ admits an idempotent left adjoint given by 
    $$K\mapsto \widehat{K}\coloneqq\mathbf{R}\lim K\otimes^L_{\mathbb
Z[x]} \mathbb{Z}[x]/(x^k)$$
where on the right and side, $K$ is viewed as an object in  $D(\mathbb{Z}[x])$ under the natural map $\mathbb{Z}[x]\to R$ sending $x \mapsto f$.  
\end{proposition}


 Hence for any $K\in D(R)$, we call $\widehat{K}$ the derived $f$-completion of $K$. In particular, $K$ is derived $f$-complete if and only if the natural map $K\to \widehat{K}$ is an equivalence. Using induction on generators, one can obtain a general formula for derived completion, see \cite[\href{https://stacks.math.columbia.edu/tag/091V}{Tag 091V}]{stacks-project}.  For $K_1,K_2\in D(R)$, we also define 
$$K_1\widehat{\otimes}^L_RK_2\coloneqq \widehat{K_1{\otimes}^L_RK_2} \in D_{\comp}(R).$$
This gives a symmetric monoidal structure on $D_{\comp}(R)$.

The following basic facts about derived completion will be used throughout the article without further comment.


\begin{enumerate}[(1)]
\item(Derived Nakayama) Let $K\in D(R)$ be derived $I$-complete. Then $K\otimes^L_R R/I \simeq 0$ if and only if $K\simeq 0$, see \cite[\href{https://stacks.math.columbia.edu/tag/0G1U}{Tag 0G1U}]{stacks-project}.
\item $K\in D(R)$ is derived $I$-complete if and only if $H^i(K)$ is derived $I$-complete for all $i$, see \cite[\href{https://stacks.math.columbia.edu/tag/091U}{Tag 091U}]{stacks-project}.
\item An $R$-module $M$ is classically $I$-complete if and only if it is derived $I$-complete and $I$-adically separated, see \cite[\href{https://stacks.math.columbia.edu/tag/091T}{Tag 091T}]{stacks-project}.
\item Derived $I$-complete modules form a weak Serre subcategory of the category of all modules \cite[\href{https://stacks.math.columbia.edu/tag/091U}{Tag 091U}]{stacks-project}, i.e., it is a strictly full subcategory that is closed under kernels, cokernerls and extensions (\cite[\href{https://stacks.math.columbia.edu/tag/0754}{Tag 0754}]{stacks-project}).
\end{enumerate}

As mentioned above, we will mostly use derived completion in the case where $I=(f)$ is principal. Due to the following proposition, most of the examples that we encounter in the context will be classically $I$-complete. Recall that every pro-system can be viewed as an object in the pro-category, and two pro-systems are pro-isomorphic to each other if they are isomorphic in the pro-category, see \cite[\href{https://stacks.math.columbia.edu/tag/05PT}{Tag 05PT}]{stacks-project} for more details on pro-categories.

\begin{proposition}[{\cite[III, Lemma 2.4]{BhattPrismaticLectureNotes}}]
\label{prop.bounded torsion is classical complete}
    If an $R$-module $M$ has bounded $f^\infty$-torsion, i.e., $M[f^n]= M[f^{n+1}]= \cdots$ for some $n$, then the the pro-systems $\{M/f^n\}_n$ and $\{M\otimes^L_R R/f^n\}_n$ are pro-isomorphic. As a consequence, the derived $f$-completion of $M$ is discrete (meaning it has cohomology in a single degree) and given by its classical $f$-completion.
\end{proposition}


We next record the definition of $I$-complete flatness.

\begin{definition}
    We say that $K\in D(R)$ is $I$-completely (faithfully) flat if $K{\otimes}^L_RR/I$ is discrete and its degree zero cohomology is a (faithfully) flat $R/I$-module. 
\end{definition}

For more properties of $I$-complete flatness and more generalities on $I$-complete Tor amplitude, we refer the readers to \cite[Section 4.1]{bhatt2019topological}. Here we only record a simple lemma that will be used later.

\begin{lemma}
\label{lem.completed tensor product of p-completely flat algebra}
    Let $R\to S$ be a $p$-completely flat morphism of derived $p$-complete rings. Suppose $M$ is an $R$-module with bounded $p^\infty$-torsion. Then we have the following identification $M\widehat\otimes^L_R S \simeq (M\otimes_R S)^{\wedge c}$ where the former term is the derived $p$-completion of $M\otimes_RS$ and the latter term is the classical $p$-completion of $M\otimes_RS$.
\end{lemma}
\begin{proof}
    By definition, $M\widehat\otimes^L_R S$ is the derived limit of the pro-system $\{M\otimes^L_R  S \otimes^L_\mathbb{Z} \mathbb{Z}/p^n \}_n$. We have the following pro-isomorphisms 
    $$\{M\otimes^L_R  S \otimes^L_\mathbb{Z} \mathbb{Z}/p^n \}_n \simeq \{M/p^n\otimes^L_R  S  \}_n \simeq \{M/p^n\otimes^L_{R/p^n} R/p^n\otimes^L_R  S  \}_n$$
    where the first pro-isomorphism is obtain from \autoref{prop.bounded torsion is classical complete} and the assumption that $M$ has bounded $p^\infty$-torsion. However, the third term is pro-isomorphic to $\{M/p^n\otimes_{R/p^n} S/p^n\}_n$ by the assumption that $S$ is $p$-completely flat over $R$. Since the pro-system $\{M/p^n\otimes_{R/p^n} S/p^n\}_n$ clearly computes the classical $p$-completion of $M\otimes_R S$, the result follows.
\end{proof}

\begin{convention}
    From now on, unless otherwise specified, we will use $\widehat{(-)}$ to denote the functor of derived $p$-completion, we will use $-\widehat{\otimes}^L_R-$ to denote the derived $p$-completion of the tensor product. 
\end{convention}
\subsection{Perfectoid rings and perfectoidization}
\label{subsec:PefectoidAlgebras}
In this section, we collect some facts about perfectoid rings that will be needed for this article. We refer the readers to \cite[Section 3]{BhattMorrowScholzeIHES} and \cite{BhattScholzepPrismaticCohomology} for more details on perfectoid rings and the purity theorems. We start with the definition of perfectoid rings following \cite{BhattIyengarMaRegularRingsPerfectoid} (which is equivalent to the definition in \cite{BhattMorrowScholzeIHES}). 
We fix a prime number $p$.

\begin{definition}
A ring $S$ is \emph{perfectoid} if it satisfies the following:
\begin{itemize}
    \item $S$ is $p$-adically complete.
    \item There exists an element $\varpi\in S$ so that $p=u\varpi^p$ for some unit $u\in S$.
    \item The Frobenius map on $S/p$ is surjective.
    \item The kernel of the Fontaine's map $\theta$: $W(S^\flat)\to S$ is principal,\footnote{We refer to \cite[Section 3]{BhattMorrowScholzeIHES} for detailed definition of $\theta$: essentially, this is the unique map lifting the natural surjection $S^\flat\to S/p$.} where $S^\flat=\varprojlim_F S/p$.
\end{itemize}
\end{definition}


In our context, it is often the case that either $p=0$ in $S$ or $p$ is a nonzerodivisor in $S$. Therefore we point out that if $S$ has characteristic $p$, then a perfectoid ring is the same as a perfect ring, see \cite[Example 3.15]{BhattMorrowScholzeIHES}. On the other hand, if $S$ is $p$-torsion free and $p\in S$ admits a $p$-power root $\varpi=p^{1/p}$, then $S$ is perfectoid if and only if it is $p$-adically complete and the Frobenius map $S/\varpi\to S/p$ is bijective, see \cite[Lemma 3.10]{BhattMorrowScholzeIHES}. This is compatible with the definition given in  \cite[2.2]{AndreWeaklyFunctorialBigCM} or \cite[Definition 2.2]{MaSchwedeSingularitiesMixedCharBCM}. 

We list two main examples of perfectoid rings that will be relevant in this paper, see \cite[Example 3.8]{BhattIyengarMaRegularRingsPerfectoid} for more details.
\begin{example}
\label{example:PerfectoidRings}
\begin{enumerate}[(1)]
    \item $\Ainfty$, the $p$-adic completion of the ring $\Ainfty^{nc}$ in \autoref{section:normlength}, is perfectoid. 
    \item Let $(R,\m)$ be a Noetherian complete local domain of residue characteristic $p>0$. Then $\widehat{R^+}$, the $p$-adic completion of the absolute integral closure of $R$, is perfectoid.
\end{enumerate}
\end{example}

We next note that if $I$ is an ideal of a perfectoid ring $R$, then $R/I$ need not be $p$-adically separated. We define $I^{-} := \bigcap_n (I+p^n)$, which is the closure of $I$ in the $p$-adic topology. Then the $p$-adic completion of $R/I$ is isomorphic to $R/I^{-}$.  Finally, suppose $\{f_i\}_{i=1}^n$ is a sequence of elements in $R$, each of which admits a compatible system of $p$-power roots $\{f_i^{1/p^e}\}$, and set $I=(f_1^{1/p^\infty},\dots,f_n^{1/p^\infty})$, the ideal generated by $f_i^{1/p^e}$ for all $i$ and $e$.  In this case, one checks directly that $R/I^-$ is perfectoid (for example see \cite[Example 6.2.11]{BhattLectureNotesPerfectoidSpaces}).


If $R$ is a perfectoid ring and $S$ is a derived $p$-complete $R$-algebra, then \cite{BhattScholzepPrismaticCohomology} defined the {\it perfectoidization} $S_{\perfd}$ of $S$ using (derived) prismatic cohomology. In general $S_{\perfd}$ only lives in $D^{\geq0}(R)$, but it is an honest perfectoid ring in all the cases that we consider. We will not give the precise definition here but we point out the following results that will be relevant to us later. 
\begin{enumerate}[(1)]
  \item[(0)] Perfectoid rings are always reduced \cite[Paragraph after (2.1.3.1)]{CesnaviciusScholzePurityflatCohomology}. 
  \item In characteristic $p>0$, $S_{\perfd}$ is the usual perfection $\varinjlim_eF^e_*S$ \cite[Example 8.3]{BhattScholzepPrismaticCohomology}.
  \item If $S$ is a derived $p$-complete quotient of $R$ (e.g., $S=R/J$ for a finitely generated ideal $J\subseteq R$), then $S_{\perfd}$ is a perfectoid ring and is a quotient of $S$ \cite[Theorem 7.4]{BhattScholzepPrismaticCohomology}.
  \item If $R\to S$ is the $p$-adic completion of an integral map, then $S_{\perfd}$ is a perfectoid ring \cite[Theorem 10.11]{BhattScholzepPrismaticCohomology}.
  \item $S_{\perfd}$ can be characterized as the derived limit of $R'$ over all maps from $S$ to perfectoid rings $R'$, and it does not depend on the choice of $R$ \cite[Proposition 8.5]{BhattScholzepPrismaticCohomology}. In particular, if $S_{\perfd}$ is a perfectoid ring then $S\to S_{\perfd}$ is the universal map to a perfectoid ring.
\end{enumerate}

As a consequence of $(0)$, we see that a perfectoid ring $R$ has bounded $g^\infty$-torsion for any $g\in R$. In fact, this is true for any reduced ring $R$: for if $g^nx=0$ for some $x\in R$ and some $n\in \mathbb{N}$, then $(gx)^{n}=0$ and thus $gx=0$ by reducedness, so any $g^\infty$-torsion is $g$-torsion. 

As a consequence of $(2)$, for any finitely generated ideal $J\subseteq R$, we can define an ideal $J_{\perfd}=\ker(R\to (R/J)_{\perfd})$. It turns out that we have a well-behaved almost mathematics theory with respect to $J_{\perfd}$ for {\it derived $p$-complete} $R$-modules, see \cite[Section 10]{BhattScholzepPrismaticCohomology} (the essential point that lurks behind is Andr\'{e}'s flatness lemma, see \cite[Theorem 7.12 and Theorem 7.4]{BhattScholzepPrismaticCohomology}). Moreover, for each $n$, after modulo $p^n$, the pair $(R/p^n, J_{\perfd}R/p^n)$ defines a \emph{basic setup} for almost mathematics as in \cite[2.1.1]{GabberRameroAlmostringtheory}. 

We will frequently work in the above setup when $J=(g)$ is a principal ideal. Even in this case, describing $(g)_\perfd$ explicitly is not easy (to the best of the authors' knowledge). On the other hand, we record following well-known fact that describes $(g)_\perfd$ when $g$ admits a compatible system of $p$-power roots.

\begin{lemma}
\label{lem.gPerfd=AllpPowerRootsofg}
    Suppose $R \to S$ is a map of perfectoid rings. Suppose $g \in R$ is such that the image of $g$ has a compatible system $\{ g^{1/p^e} \}$ of $p$-power roots in $S$.  Then 
$$((g)_\perfd S)^- = (gS)_\perfd  = (g^{1/p^\infty})^-,$$
where $(g^{1/p^\infty})$ denotes the ideal in $S$ generated by $g^{1/p^e}$ for all $e>0$. 

In particularly, if we have two compatible systems $\{ g^{1/p^e}\}$ and $\{ \widetilde{g}^{1/p^e}\}$ of $p$-power roots of $g$ in $S$, then $(g^{1/p^\infty})^-=(\widetilde{g}^{1/p^\infty})^-$.
\end{lemma}

\begin{proof}
The second equality follows directly from the universal property of perfectoidization: if $S/gS$ maps to a perfectoid ring $S'$, then as $S'$ is reduced and $p$-adically complete, this map automatically factors through $S/(g^{1/p^\infty})^-$, which is already a perfectoid ring. Thus we have $S/(g^{1/p^\infty})^-=(S/gS)_\perfd= S/(gS)_\perfd$. 

To prove the first equality, we note that by Andr\'e's flatness lemma \cite[Theorem 7.14]{BhattScholzepPrismaticCohomology}, $R\to S$ factors through a $p$-completely faithfully flat perfectoid $R$-algebra $\widetilde{R}$ such that $g$ admits a compatible system of $p$-power roots $\{g^{1/p^e}\}$ in $\widetilde{R}$.\footnote{More precisely, we are applying \cite[VIII, Theorem 3.1]{BhattPrismaticLectureNotes} here, i.e., we are only formally adjoining all $p$-power roots of $g$ to $R$ and not adjoining the roots of all monic polynomials over $R$. More specifically, $\tilde{R}=(R[x^{1/p^\infty}]/(x-g))_\perfd$. Clearly, by mapping $\{x^{1/p^e}\}$ to the compatible system of $p$-power roots of $g$ in $S$ we get a map from $R[x^{1/p^\infty}]/(x-g) \to S$. Since $S$ is perfectoid, by the universal property of perfectoidization, we obtain a map $\tilde{R}\to S$.}  If we can prove the first equality for $\widetilde{R}$, then by the already established second equality we have $((g)_{\perfd}\widetilde{R})^-=(g^{1/p^\infty})^-$ in $\widetilde{R}$. In particular, $((g)_{\perfd}\widetilde{R})^-$ and thus $((g)_{\perfd}S)^-$ contains $(g^{1/p^\infty})$. Since it is closed in the $p$-adic topology, it contains $(g^{1/p^\infty})^-=(gS)_{\perfd}$. As the other direction $((g)_\perfd S)^-\subseteq (gS)_{\perfd}$ is obvious, we have $((g)_\perfd S)^-= (gS)_{\perfd}$ as wanted. Therefore, by replacing $S$ by $\widetilde{R}$, we may assume that $S$ is $p$-completely faithfully flat over $R$ in order to prove the first equality. In this case, we have
$${(S/(g)_\perfd S)}^{\wedge c}\cong (R/(g)_\perfd)\widehat{\otimes}^L_RS\cong \left((R/(g))\widehat{\otimes}^L_RS\right)_\perfd\cong (S/gS)_\perfd\cong S/(gS)_\perfd,$$
where the first term is the classical $p$-completion of $S/(g)_\perfd$. The first equality is obtained by applying \autoref{lem.completed tensor product of p-completely flat algebra} to $M=R/(g)_\perfd$ (since $R/(g)_\perfd$ is perfectoid it has bounded $p^\infty$-torsion), the second equality follows from \cite[Proposition 8.13]{BhattScholzepPrismaticCohomology} and the third equality uses the fact that perfectoidization only depends on $\pi_0$ of an animated ring by \cite[Proposition 8.5]{BhattScholzepPrismaticCohomology}. Since the classical $p$-adic completion of $S/(g)_\perfd S$ is $S/((g)_\perfd S)^-$, we have $((g)_\perfd S)^-=(gS)_\perfd$ as wanted.
\end{proof}

\begin{caution}
We caution the reader that, throughout the rest of this article, whenever we use almost mathematics with respect to $(g)_\perfd$ or $J_\perfd$ (e.g., whenever we say a module $M$ over a perfectoid ring $R$ is $g$-almost zero or $J$-almost zero), the module $M$ involved is tacitly assumed to be derived $p$-complete: these include all $p^n$-torsion $R$-modules or more generally, all (classically) $p$-adically complete $R$-modules.  
\end{caution}

Our work crucially relies on the almost purity theorem of Bhatt-Scholze \cite{BhattScholzepPrismaticCohomology}. For earlier versions of the almost purity theorem, see  \cite{faltings2002almost,ScholzePerfectoidspaces,KedlayaLiuRelativepadicHodgefoundations,AndrePerfectoidAbhyankarLemma}. 

\begin{theorem}[{\cite[Theorem 10.9]{BhattScholzepPrismaticCohomology}}]
\label{thm.BhattScholzeAlmostPurity}
Let $R$ be a perfectoid ring, $J\subseteq R$ a finitely generated ideal. Let $S$ be a finitely presented finite $R$-algebra such that $\Spec(S)\to\Spec(R)$ is finite \etale outside $V(J)$. Then $S_\perfd$ is a perfectoid ring, and the map $S\to S_\perfd$ is an isomorphism away from $V(J)$. Moreover, for every $n>0$, the map $R/p^n\to S_\perfd/p^n$ is $J$-almost finite \'etale, that is, $S_\perfd/p^n$ is $J$-almost finite projective and $J$-almost unramified over $R/p^n$. 

In addition, if $S$ admits a $G$-action for some finite group $G$ such that $\Spec(S)\to \Spec(R)$ is a $G$-Galois cover outside $V(J)$, then $R\to S_\perfd$ is a $J$-almost $G$-Galois cover. That is, the maps $R\to \mathbf{R}\Gamma(G, S_\perfd)$ and $S_\perfd\widehat{\otimes}^L_RS_\perfd\to \prod_G S_\perfd$ are $J$-almost isomorphisms.
\end{theorem}

We refer the readers to \cite{GabberRameroAlmostringtheory} (and also \cite[Section 10]{BhattScholzepPrismaticCohomology}) for basic language of almost mathematics, including the definitions of almost finitely generated, almost projective and almost unramified (we will recall the definition of almost projective in \autoref{sec.TransformationRule}).

To prove what we need in the sequel, we record the following remark which follows from Bhatt-Scholze's proof of \autoref{thm.BhattScholzeAlmostPurity}.

\begin{remark}
\label{rmk.BhattSchozeAlmostPurityAIC}
With notation as in \autoref{thm.BhattScholzeAlmostPurity}, and assume additionally that $R$ is absolutely integrally closed and $S$ has constant rank $r$ outside $V(J)$. Then by \cite[Lemma 1.9.2]{AndrePerfectoidAbhyankarLemma}, we can find $R\to S\to S'$ such that $\Spec(S')\to \Spec(R)$ is $\Sigma_r$-Galois outside $V(J)$ and $\Spec(S')\to\Spec(S)$ is $\Sigma_{r-1}$-Galois outside $V(J)$ (here $\Sigma_r$ is the symmetric group on $r$ elements). It then follows as in the proof of \cite[Theorem 10.9]{BhattScholzepPrismaticCohomology} that we can find generators $f_1,\dots,f_n$ of $J$ such that for each $i$, we have an $f_i$-almost isomorphism $S'_\perfd \xrightarrow{\phi_i}\prod_{\Sigma_r}R$. Therefore applying $\mathbf{R}\Gamma(\Sigma_{r-1}, -)$, we obtain $f_i$-almost isomorphisms
$$S_\perfd \cong \mathbf{R}\Gamma(\Sigma_{r-1}, S'_\perfd) \cong \prod_{\Sigma_r/\Sigma_{r-1}}R.$$
Note that if $r=0$, the above should be interpreted as saying that $S_\perfd$ is $f_i$-almost zero.
\end{remark}

We now prove \autoref{prop.gisogeny} and \autoref{prop.BhattScholzeAlmostPuritypcompletefaithfullyflat}, which are addendums to \autoref{thm.BhattScholzeAlmostPurity}. These results are well-known to experts, and more general results will appear in forthcoming work \cite{BhattLuriepadicRHmodp}. We give detailed arguments here for the sake of completeness. In our applications, we will only need these propositions in the case when $R$ and $S$ are $p$-torsion free, but we state and prove them in slightly more general setup.

\begin{proposition}
\label{prop.gisogeny}
With notation as in \autoref{thm.BhattScholzeAlmostPurity}, and assume $S$ has bounded $p^\infty$-torsion, then the map $\phi$: $S\to S_\perfd$ is a $J$-isogeny, i.e., for every $g\in J$, there exists $N$ such that the map $\cone(\phi)\xrightarrow{\cdot g^N}\cone(\phi)$ induced by multiplication by $g^N$ is the zero map in $D(R)$.
\end{proposition}
\begin{proof}
We can work with each element of $J$ to assume $J=(g)$ is principal. Set $C:= \cone(\phi)$. We first prove that, to show multiplication by a power of $g$ induces the zero map on $C$ in $D(R)$, it is enough to show this after a $p$-completely faithfully flat base change. Suppose $R\to \widetilde{R}$ is $p$-completely faithfully flat and that multiplication by $g^N$ is zero on $C\widehat{\otimes}_R^L\widetilde{R}$. Then we know that 
$$(\Z/p^n)\otimes_{\Z}^L C{\otimes}_R^L\widetilde{R} \xrightarrow{\cdot g^N} (\Z/p^n)\otimes_{\Z}^L C{\otimes}_R^L\widetilde{R}$$
is the zero map. Since $\widehat{R}$ is $p$-completely faithfully flat over $R$, $\widetilde{R}\otimes^L_R R/p^n$ is discrete and faithfully flat over $R/p^n$. Now $\mathbb{Z}/p^n\otimes_\mathbb{Z}^L C$ being $p^n$-torsion implies that 
$$H^i(\mathbb{Z}/p^n\otimes^L_\mathbb{Z}C\otimes^L_R\widetilde{R})\simeq H^i(\mathbb{Z}/p^n\otimes^L_\mathbb{Z}C)\otimes_{R/p^n}(R/p^n \otimes^L_R\widetilde{R}).$$ 
It follows that $g^N$ annihilates $H^i((\Z/p^n)\otimes_{\Z}^L C)$ for each $i$ by faithfully flat descent. Note that we have a short exact sequence by \cite[\href{https://stacks.math.columbia.edu/tag/0CQE}{Tag 0CQE}]{stacks-project}: 
$$0\to \mathbf{R}^1\varprojlim_n H^{i-1}(C\otimes^L_\Z (\Z/p^n)) \to H^i(\myR\varprojlim_n C{\otimes}_{\Z}^L(\Z/p^n)) \to \varprojlim_nH^i(C{\otimes}_{\Z}^L(\Z/p^n))\to 0.$$
It follows that $g^{2N}$ annihilates $H^i(\myR\varprojlim_n C{\otimes}_{\Z}^L(\Z/p^n))\cong H^i(C)$ for all $i$, since $C$ is derived $p$-complete. By \cite[Lemma 3.2]{BhattDerivedDirectSummand}, since $C$ only has two cohomology groups, multiplication by $g^{4N}$ induces the zero map on $C$ as desired. 

By Andr\'e's flatness lemma \cite[Theorem 7.14]{BhattScholzepPrismaticCohomology}, we have a $p$-completely faithfully flat map $R\to \widetilde{R}$ of perfectoid rings such that $\widetilde{R}$ is absolutely integrally closed. So after derived $p$-complete base change to $\widetilde{R}$ we have 
$$S\widehat{\otimes}^L_R\widetilde{R} \to S_\perfd\widehat{\otimes}^L_R\widetilde{R} \to C\widehat{\otimes}^L_R\widetilde{R} \xrightarrow{+1}.$$
Next we note that, since $S$ has bounded $p^\infty$-torsion by assumption, 
by \autoref{lem.completed tensor product of p-completely flat algebra}, we have $S\widehat{\otimes}^L_R\widetilde{R}\simeq (S\otimes_R \widetilde{R})^{{\wedge c}}$ where the latter is the classical $p$-completion of $S\otimes_R \widetilde{R}$.
Moreover, it follows from \cite[Proposition 8.13]{BhattScholzepPrismaticCohomology} that $$S_\perfd\widehat{\otimes}^L_R\widetilde{R} \cong (S\widehat{\otimes}^L_R\widetilde{R})_\perfd\cong \left((S\otimes_R \widetilde{R})^{{\wedge c}}\right)_{\perfd} \cong (S\otimes_R\widetilde{R})_\perfd,$$
where the last isomorphism follows from the fact that $(\widetilde{R}\otimes_RS)_\perfd$, being a perfectoid ring by \autoref{thm.BhattScholzeAlmostPurity}, factors through the classical $p$-completion of $S\otimes_R\widetilde{R}$. Also note that, since $S$ is finitely presented as an $R$-module, $S\otimes_R\widetilde{R}$ is a quotient of a finite free $\widetilde{R}$-module $F$ by a finitely generated submodule $G$ of $F$, and thus it surjects onto its classical $p$-completion (the classical $p$-completion is $F/G^-$, where $G^-$ denotes the closure of $G$ under the $p$-adic topology on $F$). Therefore we have a commutative diagram of derived $p$-complete rings
\[
\xymatrix{
(S\otimes_R \widetilde{R})^{{\wedge c}}\cong S\widehat{\otimes}^L_R\widetilde{R}  \ar[r]^-\alpha & S_\perfd\widehat{\otimes}^L_R\widetilde{R} \\
S\otimes_R\widetilde{R} \ar[r]^-\beta \ar@{->>}[u] & (S\otimes_R\widetilde{R})_\perfd \ar@{=}[u]
}
\]
An easy diagram chasing then shows that $\coker(\beta) \cong \coker(\alpha)$ and $\ker(\beta)\twoheadrightarrow \ker(\alpha)$. It follows that if we can show that $\cone(\beta)$ is annihilated by a power of $g$ in $D(\widetilde{R})$, then so is $\cone(\alpha)\cong C\widehat{\otimes}^L_R\widetilde{R}$ (using \cite[Lemma 3.2]{BhattDirectsummandandDerivedvariant}). Thus we can replace $R$ by $\widetilde{R}$ and $S$ by $S\otimes_R\widetilde{R}$ to assume $R$ is absolutely integrally closed. 

Now if $S$ has constant rank over $R$ outside $V(J)=V(g)$, then 
by \autoref{rmk.BhattSchozeAlmostPurityAIC} we can find $f_1,\dots,f_n$ so that $J=(g)=(gf_1,\dots, gf_n)$ and that there exists a $(gf_i)$-almost isomorphism $S_\perfd\xrightarrow{\phi_i} \prod_GR$\footnote{We caution the readers that, even we are in the case that $J=(g)$ is principal and that $S$ has constant rank outside $V(g)$, it is not clear that there is a single map $S_\perfd\xrightarrow{\phi} \prod_GR$ that is a $g$-almost isomorphism, i.e., we still need to write $(g)=(gf_1,\dots,gf_n)$ and work with each $gf_i$, see \cite[Proof of Theorem 10.9, footnote 18]{BhattScholzepPrismaticCohomology}.} for a finite group $G$. By \cite[Lemma 3.2]{BhattDirectsummandandDerivedvariant}, it is enough to show that the kernel and cokernel of $\phi$ are annihilated by a power of $g$. Thus we can work with each $gf_i$ individually. Therefore, after redefining $g$ as $gf_i$, we may replace $S_\perfd$ by $\prod_GR$ and work with $S\xrightarrow{\varphi}\prod_GR$. The conclusion in this case now follows from the next claim.
\begin{claim}
There exists $N$ such that multiplication by $g^N$ induces the zero map on $\cone(\varphi)$ in $D(R)$.
\end{claim}
\begin{proof}
By \cite[Lemma 3.2]{BhattDerivedDirectSummand}, it is enough to show that there exists $N$ such that $\ker(\varphi)$ and $\coker(\varphi)$
are annihilated by $g^{N}$. Since $\varphi[1/g]$ is an isomorphism, this is clear for $\coker(\varphi)$ (as it is finitely generated over $R$ and vanishes after inverting $g$). 

As for $\ker(\varphi)$, note that since $S[1/g]\cong (\prod_GR)[1/g]$ and $\prod_GR$ is clearly finitely presented over $R$, there exists a $R$-module morphism $\varphi'$: $\prod_GR\to S$ such that $\varphi'[1/g]$ is an isomorphism. It follows that $\coker(\varphi')$ is annihilated by $g^{N'}$ for some $N'$ as $S$ and thus $\coker(\varphi')$ is finitely generated as an $R$-module. Thus we have 
$$g^{N'}\cdot \ker(\varphi) \subseteq \im(\varphi')\cap\ker(\varphi)=\im(\varphi'|_{\ker(\varphi\circ\varphi')}).$$
But it is clear that $\ker(\varphi\circ\varphi')$ is annihilated by $g$ (in fact, it is annihilated by $(g)_\perfd$) because $\prod_GR$ is reduced and $\ker(\varphi\circ\varphi')$ is $g^{\infty}$-torsion. Thus setting $N=N'+1$ we see that $g^N\cdot\ker(\varphi)=0$. 
\end{proof}
Finally, in the general case, by \cite[last paragraph of proof of Theorem 10.9]{BhattScholzepPrismaticCohomology}, we have a $J$-almost isomorphism $R\to \prod_{i=1}^t(R_i)_\perfd$ such that $(R_i)_\perfd\to S_i:= (R_i)_\perfd\otimes S$ has constant rank and $S_\perfd$ is $J$-almost isomorphic to $\prod_{i=1}^t(S_i)_\perfd$. Therefore the conclusion follows from the constant rank case by working with each $S_i\to (S_i)_\perfd$.
\end{proof}

\begin{proposition}
\label{prop.BhattScholzeAlmostPuritypcompletefaithfullyflat}
With notation as in \autoref{thm.BhattScholzeAlmostPurity}, we have:
\begin{enumerate}
    \item $R\to S_\perfd$ is $p$-completely $J$-almost flat. That is, for every $n$ and every $p^n$-torsion $R$-module $M$, $M\otimes_R^LS_\perfd$ is $J$-almost concentrated in degree 0, i.e., $\Tor_{i}^{R}(M, S_\perfd)$ is $J$-almost zero for all $i>0$.
  \item If, in addition, $S$ has constant rank $r>0$ over $R$ outside $V(J)$, then $R\to S_\perfd$ is $p$-completely $J$-almost faithful in the following sense: for every $n$ and every $p^n$-torsion $R$-module $M$, if $M\otimes_RS_\perfd$ is $J$-almost zero, then $M$ is $J$-almost zero.
\end{enumerate}
As a consequence, if $S$ has constant rank $r>0$ over $R$ outside $V(J)$, then $R\to S_\perfd$ is $p$-completely $J$-almost faithfully flat and hence $p$-completely $J$-almost pure, i.e., for every $n$ and every $p^n$-torsion $R$-module $M$, we have $\ker(M\to M\otimes_RS_\perfd)$ is $J$-almost zero.
\end{proposition}
\begin{proof}
The argument essentially follows from the same strategy as in the proof of \autoref{prop.gisogeny} (i.e., reducing to the absolutely integrally closed case, using \autoref{rmk.BhattSchozeAlmostPurityAIC} in the constant rank case and then deducing the general case), so we provide fewer details.

To prove (a) and (b), note that by Andr\'e's flatness lemma (see \cite[Theorem 7.14]{BhattScholzepPrismaticCohomology}), we can perform a $p$-completely faithfully flat base change to assume $R$ is absolutely integrally closed. Now if $S$ has constant rank $r$ over $R$ outside $V(J)$, then by \autoref{rmk.BhattSchozeAlmostPurityAIC} there exists $f_1,\dots,f_n$ that generates $J$ such that $S$ is $f_i$-almost isomorphic to a product of $R$ as an $R$-module. Therefore both (a) and (b) follows since being $J$-almost zero is equivalent to being $f_i$-almost zero for each $i$ (note that for (b) we need $r>0$ to avoid the case that $S_\perfd$ is $J$-almost zero). For the general case of part (a), note that by \cite[last paragraph of proof of Theorem 10.9]{BhattScholzepPrismaticCohomology}, we have a $J$-almost isomorphism $R\to \prod_{i=1}^t(R_i)_\perfd$ such that $R_i\to S_i:= R_i\otimes S$ has constant rank. Thus we reduce to the constant rank case. 

To see the last conclusion, we set $K := \ker(M\to M\otimes_RS_\perfd)$. By part (a), after tensoring with $S_\perfd$, we have a $J$-almost exact sequence: 
$$0\to K\otimes_R S_\perfd \to M\otimes_RS_\perfd \to M\otimes_R S_\perfd\otimes_R S_\perfd.$$
Since the last map above is split (induced by the multiplication map), it follows that $K\otimes_R S_\perfd$ is $J$-almost zero and hence $K$ is $J$-almost zero by part (b). 
\end{proof}

\subsection{Big Cohen-Macaulay algebras}
\label{subsec:BigCohenMacaulayAlgebras}
Let $(R,\m)$ be a Noetherian local ring. An $R$-algebra $B$, not necessarily finitely generated over $R$, is called big Cohen-Macaulay (resp. balanced big Cohen-Macaulay) if some (resp. every) system of parameters $x_1,\dots,x_d$ of $R$ is a regular sequence on $B$ and $B/\m B\neq 0$. Historically, there are other variants of big Cohen-Macaulay algebras, and we refer to \cite[Section 2.2]{BMPSTWW-MMP} for a quick summary. For our purpose in this article (see \autoref{sec.BCM-regularrityPerfdSignature}), we can often ignore these distinctions thanks to the following basic fact: if $B$ is a (perfectoid) big Cohen-Macaulay $R$-algebra, then the $\m$-adic completion of $B$ is (perfectoid) balanced big Cohen-Macaulay, see \cite[Corollary 8.5.3]{BrunsHerzog} and \cite[Proposition 2.2.1]{AndreWeaklyFunctorialBigCM}.

Big Cohen-Macaulay algebras exist in the following weakly functorial sense: suppose $(R,\m)\to (S,\mathfrak{n})$ is a map of Noetherian complete local domains, then there exists a commutative diagram: 
\[\xymatrix{
R \ar[d]\ar[r] & S \ar[d]\\
R^+ \ar[r]\ar[d] & S^+ \ar[d]\\
B \ar[r] & C
}\]
such that $B$, $C$ are balanced big Cohen-Macaulay algebras over $R$ and $S$ respectively, see \cite{HochsterHunekeInfiniteIntegralExtensionsAndBigCM,HochsterHunekeApplicationsofBigCM,AndreWeaklyFunctorialBigCM,MurayamaSymbolicTestIdeal}. Furthermore, we have 
\begin{itemize}
    \item If $R/\m$ has characteristic $p>0$, then $B$ and $C$ can be taken to be perfectoid rings \cite{AndreWeaklyFunctorialBigCM}, and in fact, one can take $B$ and $C$ to be the $p$-adic completions of $R^+$ and $S^+$ respectively \cite{HochsterHunekeInfiniteIntegralExtensionsAndBigCM,BhattAbsoluteIntegralClosure}.
    \item If $R/\m=S/\mathfrak{n}$ has characteristic $p>0$, then for any given perfectoid big Cohen-Macaulay $R^+$-algebra $B$, one can find a perfectoid $S^+$-algebra $C$ making the above diagram commutes \cite{AndreWeaklyFunctorialBigCM,MaSchwedeTuckerWaldronWitaszekAdjoint} .
\end{itemize}

We recall and modify a result of Gabber to our context. 

\begin{lemma}[{\cite{GabberMSRINotes}, \cite[Section 17.5]{GabberRameroFoundationsAlmostRingTheory}}]
\label{thm.GabbersTrickVVersion}
Let $(R,\m)$ be a Noetherian complete local domain of residue characteristic $p>0$ and let $B$ be a perfectoid balanced big Cohen-Macaulay $R^+$-algebra. Suppose $v$ is a $\bR$-valuation on $\widehat{R^+}$ that is positive on $\m R^+$, and suppose that $\{c_j\}_{j\geq 0}$ is a sequence of elements of $\widehat{R^+}$ such that $v(c_j)\to 0$ as $j\to\infty$. 

Then if we set $B'':= {W^{-1}\prod_{j \geq 0} B}$ and $B':=\widehat{B''}$, where $W$ is the multiplicative set generated by ${\bf c} := (c_0, c_1, \dots)$, then $B'$ is a perfectoid balanced big Cohen-Macaulay $B$-algebra (and hence $R^+$-algebra) where $B \to B'$ is the diagonal map.  
\end{lemma}

\begin{proof}
        The idea of the proof is contained in \cite[Page 3]{GabberMSRINotes} but we record it for the convenience of the reader.         The fact that $B'$ is perfectoid follows from the fact that an arbitrary product of perfectoid rings is perfectoid \cite[Example 3.8 (8)]{BhattIyengarMaRegularRingsPerfectoid} and a $p$-completed localization of a perfectoid ring is perfectoid \cite[Example 3.8 (7)]{BhattIyengarMaRegularRingsPerfectoid}.

To see $B'$ is balanced big Cohen-Macaulay, by \cite[Theorem 2.8]{BMPSTWW-MMP}, it is enough to show the following:\footnote{We only carry out the case when $R$ has mixed characteristic, the case $R$ has equal characteristic $p$ is easier as the $p$-adic completion is doing nothing and so $B'=B''$ in this case.}
\begin{enumerate}[(a)]
    \item $p, x_2,\dots,x_d$ is a regular sequence on $B'$ for every system of parameters starting with $p$.
    \item $B'/\m B'\neq 0$.
\end{enumerate}

We explain part (a).  First note that $p ,x_2,\dots, x_d$ is a regular sequence on $B''$: it is a regular sequence on $B$ by assumption and thus it is a regular sequence after mapping to $\prod_{j\ge 0} B$ via the diagonal, and thus remains a regular sequence after localization. Next, since $p$ is a nonzerodivisor on $B''$, it remains a nonzerodivisor on $B'=\widehat{B''}$. Now, modulo $p$, we have $B'/p=B''/p$ and so (a) follows.

Now we prove part (b). 
Suppose $1 \in \m B'$, then $1\in \m B''$ too since $p\in \m$ so that $B'/\m \cong B''/\m$. In particular, there exists $m > 0$ such that ${\bf c}^m \in \m (\prod_{j\geq 0}B)$.  Meaning that $c_j^m \in \m B$ for all $j$.  
        
        \begin{claim}
            There exists a sequence of elements $d_j \in R^+$ such that $v(d_j) = v(c_j^m)$ and such that $d_j \in \m B$.
        \end{claim}
        \begin{proof}
            We may write $c_j^m = d_j + p^{N_j} h_j$ for some $h_j \in \widehat{R^+}$, $N_j \gg 0$, and $d_j \in R^+$.  Since $N_j \gg 0$ we have $v(d_j) = v(c_j^m)$.  Since $c_j \in\m B$ we see the same is true for $d_j$.
        \end{proof}

        Now, each $d_j \in S_j \subseteq R^+$ where $S_j$ is a module-finite domain extension of $R$, in particular $S_j$ is a Noetherian complete local domain.  Thus $d_j \in \m B \cap S_j$.  But since $B$ is big Cohen-Macaulay over $S_j$, it is a solid $S_j$-algebra by \cite[Corollary 2.4]{HochsterSolidClosure} and so $\m B \cap S_j$ is contained in the solid closure of $\m S_j$, which is contained in the integral closure of $\m S_j$ by \cite[Theorem 5.10]{HochsterSolidClosure}, also see \cite[Proof of Theorem 4.5]{MaSchwedeSingularitiesMixedCharBCM}.
        In particular, we have $v(d_j)\geq v(\m S_j) \geq v(\m R^+)$ since integral closure is computed valuatively.  However, this implies $$v(\m R^+) \leq \displaystyle\lim_{j \to \infty} v(d_j) =  \lim_{j \to \infty} v(c_j^m)= 0,$$ which is a contradiction as $v(\m R^+)>0$ since $\m R^+$ is finitely generated (and $v$ is positive on $\m R^+$). 
\end{proof}

    We will typically use the previous result in the following form.

    \begin{corollary}
    \label{cor.GabberTrick}
        With notation as in \autoref{thm.GabbersTrickVVersion}, suppose that $J \subseteq R$ is an ideal and $z \in B$ is such that the image of $c_j z$ in $B$ is contained in $JB$ for all $j \geq 0$.  Then the image of $z$ in $B'$ is contained in $JB'$ (where $B\to B'$ is the diagonal map). 

        As a consequence, if $d = \dim R$ and $\eta \in H^d_{\mathfrak{m}}(B)$ is such that $c_j \eta = 0$ for all $j \geq 0$ then $\eta \mapsto 0 \in H^d_{\m}(B')$.
    \end{corollary}
    \begin{proof}
        It is enough to show that the image of $z$ is contained in $JB''$. By hypothesis, we have ${\bf c}z \in JB''$, but since ${\bf c}$ is a unit in $B''$, we have that the image of $z$ in $B''$ is contained in $JB''$.  This proves the first statement.  

        For the second part, note that we can represent $\eta$ as a class $\big[ {z \over x_1^t \cdots x_d^t } \big]$, where $x_1,\dots,x_d$ is a system of parameters of $R$, $z \in B$, and $t > 0$.  Since $B$ is balanced big Cohen-Macaulay, and since $c_j \eta = 0$, we see that $c_j z \in (x_1^t, \dots, x_d^t)B$ for every $j \geq 0$.  By the first statement, the image of $z$ is contained in $(x_1^t, \dots, x_d^t)B'$ proving that $\eta \mapsto 0 \in H^d_{\m}(B')$ as desired.
    \end{proof}
\section{\texorpdfstring{$F$}{F}-signature and Hilbert-Kunz multiplicity via perfection in positive characteristic}
\label{section:positivecharacterstic}

In this section, we describe the $F$-signature and Hilbert-Kunz multiplicity in terms of normalized length and perfection. Intuitively, since both invariants are defined as limits involving iterated Frobenius twists, it is natural to try to encapsulate all Frobenius twists uniformly using perfection. However, the perfection is a non-Noetherian object, making the usual notion of length inapplicable. To address this, we use normalized length, which allows us to recover the original invariants, as shown in \autoref{thm.SignatureHKviaNormalizedLength}.

Suppose that $(R,\m)$ is a Noetherian complete local domain of equal characteristic $p>0$ with perfect residue field $k=R/\m$. By the Cohen-Gabber theorem \cite{IllusieLaszloOrgogozoGabber}, we may choose a coefficient field
$k \subseteq R$ and a system of parameters $x_1, \ldots, x_d$ so that $R$ is module finite and \emph{generically
separable} over the regular subring $A = k[[x_1, . . . , x_d]]$; see \cite{KuranoShimomotoElementaryProof} for an elementary proof. Set $A_e = A^{1/p^e}$ for $e \geq 0$ and let $\lambda_\infty(\blank)$ be the normalized length discussed in \autoref{section:normlength}, note that here we have $\Ainfty^{un}=A_\perf = \bigcup_e A_e$. Under these assumptions, there exists an element $g \in A$ such that
\begin{equation}
    \label{eq.gRootsContainedInRA_e}
    g \cdot R^{1/p^e} \subseteq R[A_e] = R \otimes_A A_e \subseteq R^{1/p^e}
\end{equation}
for all $e \geq 0$; see \cite[Lemma 2.3]{PolstraTuckerCombinedApproach} for further detail. 
In particular, it follows that the natural maps
\begin{equation*}
    R^{1/p^e} \otimes_{A_e} A_\perf \to R_\perf
\end{equation*}
are inclusions for each $e \geq 0$ with cokernels annihilated by $g^{1/p^e}$. This shows explicitly that, while $R_\perf$ is not finitely generated over $A_\perf$, it is $g$-almost finitely presented.

We first recall the definition of Frobenius non-splitting ideal $I_e$. 

\begin{definition} With notation as above, we define $I_e$ to be
\begin{align*}
\label{def.I_e}
{I}_e &=\{x\in R\mid R\to R^{1/p^e}\text{ such that }1\mapsto x^{1/p^e} \text{ is not pure}\}\\
&=\{x\in R\mid \phi(x^{1/p^e})\in \m \text{ for any }\phi\in\Hom(R^{1/p^e},R) \}\\
&= \left\{ x\in R \mid x^{1/p^e}\otimes \eta =0 \text{ in } R^{1/p^e} \otimes_R E_R(k) \text{ where } \eta \text{ generates the socle of } E_R(k) \right\}.
\end{align*}
See \autoref{subsec:PureSplitMap} for explanations why these are equivalent definitions.
\end{definition}

\begin{proposition}
\label{prop.IeProperties}
     With notation as above, for $I_e\subseteq R$ and $f \geq 0$, we have the following:
    \begin{enumerate}
        \item $\m^{[p^e]}\subseteq I_e$ \label{prop.IeProperties.a}
        \item $I_e^{[p^f]} \subseteq I_{e+f}$ \label{prop.IeProperties.b}
        
        \item for any $\varphi \in \Hom_R(R^{1/p^f}, R)$, $\varphi(I_{e+f}^{1/p^f}) \subseteq I_e$.\label{prop.IeProperties.c}
    \end{enumerate}
\end{proposition}
\begin{proof}
    One can find the proofs for (a) and (b) in \cite[Lemma 4.4]{TuckerFsigExists}. (c) follows from the definition, or see \cite[Lemma 3.4]{BlickleSchwedeTuckerFSigPairs1}.
\end{proof}

Inspired by the definition of $I_e$, it is natural to define the non-splitting ideal of $R_\perf $ as follows:

\begin{definition}
\label{def.IinftyPosChar}
 With notation as above, we define $I_\infty$ to be
\begin{align*}
I_{\infty} &=\{x\in R_{\perf}\mid R\to R_{\perf}\text{ such that }1\mapsto x \text{ is not pure}\}\\
&=\{x\in R_\perf\mid \phi(x)\in \m \text{ for any }\phi\in\Hom(R_\perf,R) \}\\
&= \left\{ x\in R_{\perf} \mid x\otimes \eta =0 \text{ in } R_\perf \otimes_k E_R(k) \text{ where } \eta \in E_R(k) \text{ generates the socle}  \right\}.
\end{align*}
Again, we refer to \autoref{subsec:PureSplitMap} for explanations why these are equivalent definitions.
\end{definition}

\begin{proposition}
\label{prop.I_inftyeEqualstoUnion}
    With notation as above, we have
    \[
        I_{\infty}=\bigcup_e I_e^{1/{p^e}}=\bigcup_e I_e^{1/{p^e}}R_{\perf}.
    \]
\end{proposition}
\begin{proof}
    Note first that if $R$ is not $F$-pure, then we have $I_e = R$ for all $e > 0$ and so $I_{\infty} = R_\perf$, from which both equalities above are clear. 

    We assume that $R$ is $F$-pure.
    We first prove that the second equality holds. 
    It is clear that $\bigcup_e I_e^{1/{p^e}}\subseteq \bigcup_e I_e^{1/{p^e}}R_\perf$. As for the other direction, note that by \autoref{prop.IeProperties} \autoref{prop.IeProperties.b}, we have $I_e^{1/{p^e}}R^{1/{p^{e+f}}}=(I_e^{[p^f]})^{1/{p^{e+f}}}\subseteq (I_{e+f})^{1/{p^{e+f}}}$. This proves the second equality. 
    
    For the first equality, since $R^{1/{p^e}}\to R_\perf$ is pure, we have $R\to R_\perf$ sending $1\to x^{1/{p^e}}$ is pure if and only if $R\to R^{1/{p^e}}$ sending $1\to x^{1/{p^e}}$ is pure. This completes the proof.
\end{proof}

\begin{definition}\cite{KunzOnNoetherianRingsOfCharP,MonskyHKFunction,SmithVanDenBerghSimplicityOfDiff,HunekeLeuschkeTwoTheoremsAboutMaximal}
     With notation as above, we define the \emph{$F$-signature} $s(R)$ to be
    \[
        s(R) := \lim_{e\to\infty} \frac{\length_R\left(R/I_e\right)}{p^{ed}}.
    \]
    Similarly, we define the \emph{Hilbert-Kunz multiplicity} of an $\m$-primary ideal $J$ to be 
    \[
        e_{\HK}(J, R) := \lim_{e\to\infty}\frac{\length_R\left(R/J^{[p^e]}\right)}{p^{ed}}.
    \]
    In the case that $J = \m$, we simply write $e_{\HK}(R)=e_{\HK}(\m, R)$.
\end{definition}

It is well-known that the above limits exist \cite{MonskyHKFunction,TuckerFsigExists,PolstraTuckerCombinedApproach}.  Our goal in this section is to identify $F$-signature and Hilbert-Kunz multiplicity as the normalized length of certain $A_{\perf}$-modules. Towards this end, we present a useful lemma:

\begin{lemma}
\label{lem.AlmostFinitelyPresented}
    With notation as above, we have 
        \[I_{\infty} \cap (A_\perf \otimes_{A_e}R^{1/p^e})\subseteq (I_e : g)^{1/p^e} (A_{\perf} \otimes_{A_e} R^{1/p^e}) = 
        A_{\perf} \otimes_{A_e} \left( I_e : g \right)^{1/p^e}.\]
    
\end{lemma}

\begin{proof}
    First of all, the equality follows from flatness of $A_\perf$ over $A^{1/p^e}$. Now we show the inclusion. Since $A$ is regular, $R^{1/p^e}\otimes_{A_e}A_\perf$ is a free $R^{1/p^e}$-module. Let $\varphi:R^{1/p^e}\otimes_{A_e} A_\perf \to R^{1/p^e}$ be an $R^{1/p^e}$-module map that is a projection onto a summand.
     Then, using \autoref{eq.gRootsContainedInRA_e}, we have the following diagram:
    \[
    \xymatrix{
    R^{1/p^e}\otimes_{A_e}A_\perf \ar[r]^-{\varphi} & R^{1/p^e} \\
    R_\perf \ar[u]^{\cdot g^{1/p^e}}\ar[ru]_{\Phi} & 
    }
    \]
    Here we define $\Phi(-) = \varphi ( g^{1/p^e} \cdot -) \in \Hom_{R^{1/p^e}}(R_\perf,R^{1/p^e})$. Now, for $x\in I_{\infty}\cap(A_\perf\otimes_{A_e}R^{1/p^e})$, $\Phi(x) = \varphi(g^{1/p^e}x) \in I_{e}^{1/p^e}$ because $x \in I_\infty$ can only be sent into $\m$ under any $R$-linear map from $R_\perf$ to $R$, and $I_e^{1/p^e}$ is the set of elements in $R^{1/p^e}$ that are sent into $\m$ (under any $R$-linear map from $R^{1/p^e}$ to $R$). Since $\varphi$ is arbitrary and $A_\perf$ is free over $A_e$, this proves that $g^{1/p^e}x\in I_e^{1/p^e}(A_\perf\otimes_{A_e}R^{1/p^e})$, and thus $$x\in I_e^{1/p^e}(A_\perf \otimes_{A_e}R^{1/p^e}):_{(A_\perf \otimes_{A_e}R^{1/p^e})}g^{1/p^e} = (I_e:g)^{1/p^e}(A_\perf \otimes_{A_e}R^{1/p^e})$$
    where the equality holds since $A_\perf \otimes_{A_e}R^{1/p^e}$ is flat over $R^{1/p^e}$. 
\end{proof}

\begin{theorem}
    \label{thm.SignatureHKviaNormalizedLength}
    With notation as above, we have
    \[
        s(R)=\lambda_\infty\left( R_\perf/I_{\infty} \right) \qquad \mbox{and} \qquad e_{\HK}(R,J) = \lambda_\infty(R_\perf / J R_\perf).
    \]
\end{theorem}

\begin{proof}
    We first prove that the $F$-signature of $R$ is equal to the normalized length of $R_\perf / I_{\infty}$. Without loss of generality, we may assume $R$ is $F$-pure. Consider the finitely generated $A_\perf$-modules $A_\perf\otimes_{A_e}R^{1/p^e}$ and observe that $\bigcup_e  A_\perf\otimes_{A_e}R^{1/p^e} = R_\perf$. Note that we have inclusions $A_\perf\otimes_{A_e}R^{1/p^e}\subseteq R_\perf$ by generic separateness and so 
    \[
        (A_\perf\otimes_{A_e}R^{1/p^e}\mid e\in \mathbb{N})
    \] 
    forms a system cofinal with the system of all finitely generated $A_\perf$-submodules of $R_\perf$. Hence, by \autoref{prop.NormalizedLengthProperties} \autoref{prop.NormalizedLengthProperties.b},
    \[ 
        \lambda_\infty \left( R_\perf/I_{\infty}\right) = \sup_e \lambda_\infty \left(\frac{A_\perf \otimes_{A_e}R^{1/p^e}}{I_{\infty}\cap (A_\perf \otimes_{A_e}R^{1/p^e}) }\right).
    \] 
    Note, if $N \subseteq N'$ are finitely generated submodules of $R_\perf/I_{\infty}$, then $\lambda_{\infty}(N) \leq \lambda_{\infty}(N')$ and so it suffices to show that
    \[
        s(R) = \lim_{e \to \infty} \lambda_\infty \left(\frac{A_\perf \otimes_{A_e}R^{1/p^e}}{I_{\infty}\cap (A_\perf \otimes_{A_e}R^{1/p^e}) }\right).
    \]
    
    Now, by \autoref{prop.I_inftyeEqualstoUnion}, 
    \[ 
        I_{\infty} \cap ( A_\perf \otimes_{A_e}R^{1/p^e})=\left(\bigcup_f I_f^{1/p^f} \right)\cap ( A_\perf \otimes_{A_e}R^{1/p^e}) \supseteq A_\perf \otimes_{A_e} (I_e)^{1/p^e}.
    \] 
    Hence, there exists a constant $C > 0$ so that  
    \[ 
        \begin{split}
        \lambda_\infty \left(\frac{A_\perf \otimes_{A_e}R^{1/p^e}}{I_{\infty}\cap (A_\perf \otimes_{A_e}R^{1/p^e}) }\right) & 
        \leq \lambda_{\infty}\left( A_\perf \otimes_{A_e} {R^{1/p^e} \over I_e^{1/p^e} } \right)\\ 
        & = (1/{p^{ed}})\cdot\length_A\left(R/I_e\right)\\
        & \leq s(R) + \frac{C}{p^{e}} 
        \end{split}
    \]
    for every $e > 0$, where the last inequality follows from \cite[
    Theorem 4.3]{PolstraTuckerCombinedApproach}.\footnote{As pointed out by the referee, without the last inequality above, simply taking the supremum over $e>0$ gives $\lambda_\infty(R_\perf/I_\infty) \leq s(R)$. However, using this last inequality and following through the rest of the proof,
    the stronger statement follows that there exists a constant $C$ so that 
    \begin{equation*}
        \left| \frac{1}{p^{ed}}\length_A(R/I_e) - \lambda_\infty \left(\frac{A_\perf \otimes_{A_e}R^{1/p^e}}{I_{\infty}\cap (A_\perf \otimes_{A_e}R^{1/p^e}) }\right) \right| \leq \frac{C}{p^e}
    \end{equation*}
    which may be of independent interest.} This proves $\lambda_\infty(R_\perf/I_{\infty}) \leq s(R)$. 
    
    Conversely, we will show $s(R) \leq \lambda_\infty(R_\perf/I_{\infty})$. 
    By \autoref{lem.AlmostFinitelyPresented} above, we have that 
    \[
        I_{\infty} \cap (A_\perf \otimes_{A_e}R^{1/p^e})\subseteq (I_e : g)^{1/p^e} (A_{\perf} \otimes_{A_e} R^{1/p^e}) = 
        A_{\perf} \otimes_{A_e} \left( I_e : g \right)^{1/p^e}
    \] 
    Therefore, there exists a constant $C' > 0$ so that
    \[ 
        \begin{split}
            \lambda_\infty \left(\frac{A_\perf \otimes_{A_e}R^{1/p^e}}{I_{\infty} \cap (A_\perf \otimes_{A_e}R^{1/p^e}) }\right) 
            &\geq \lambda_{\infty} \left(A_\perf \otimes_{A_e}  {R^{1/p^e} \over (I_e : g)^{1/p^e} }\right) \\
            &= (1/{p^{ed}})\cdot\length_A\left(\frac{R}{(I_e:g)}\right)\\
            &\ge s(R) - \frac{C'}{p^{e}}
        \end{split}
    \]
    for every $e > 0$, where the last inequality follows from \cite[Lemma 4.13]{PolstraTuckerCombinedApproach}.
    This proves the other inequality and hence we have $\lambda_\infty(R_\perf/I_{\infty}) = s(R)$.
    
    Now we prove that Hilbert-Kunz multiplicity $e_{\HK}(J, R)$ is equal to the normalized length of $R_\perf/J R_\perf$. As before, we have $(A_\perf \otimes_{A_e} R^{1/p^e}) /J R_\perf \cap (A_\perf \otimes_{A_e} R^{1/p^e}) \subseteq R_\perf / J R_\perf$. Hence we have
    
        \begin{align*}
        & \lambda_\infty \left( R_\perf/J R_\perf\right) = \sup_e \lambda_\infty \left(\frac{A_\perf \otimes_{A_e}R^{1/p^e}}{J R_\perf\cap (A_\perf \otimes_{A_e}R^{1/p^e}) }\right)\\ 
        =  & \lim_{e \to \infty} \lambda_\infty \left(\frac{A_\perf \otimes_{A_e}R^{1/p^e}}{J R_\perf\cap (A_\perf \otimes_{A_e}R^{1/p^e}) }\right).
        \end{align*}
    
    Since $ J R_\perf \cap ( A_\perf \otimes_{A_e}R^{1/p^e}) \supseteq J R^{1/p^e} \otimes_{A_e}A_\perf $, we know that there exists a constant $C > 0$ so that
    \[ 
        \begin{split}
            \lambda_\infty \left(\frac{A_\perf \otimes_{A_e}R^{1/p^e}}{J R_\perf \cap (A_\perf \otimes_{A_e}R^{1/p^e}) }\right) & \leq (1/{p^{ed}})\cdot\length_A\left(R/J^{[p^e]}\right)\\
        & \leq e_{\HK}(J, R) + \frac{C}{p^{e}} 
        \end{split}
    \]
    for every $e > 0$.
    Hence we obtain the inequality $\lambda_\infty(R_\perf/J R_\perf)\le e_{\HK}(J, R)$. 
    
   Conversely, since $g^{1/p^e}R_\perf \subseteq A_\perf \otimes_{A_e}R^{1/p^e}$, we have
    \[ 
        J R_\perf \cap (A_\perf \otimes_{A_e}R^{1/p^e})\subseteq \left(A_\perf\otimes_{A_e} J R^{1/p^e} \right) :_{(A_\perf \otimes_{A_e}R^{1/p^e})} g^{1/p^e}= A_\perf \otimes_{A_e} \left( J: g\right)^{1/p^e}
    \]
    where the equality holds since $A_\perf \otimes_{A_e}R^{1/p^e}$ is flat over $R^{1/p^e}$. It follows that there exists a constant $C' > 0$ so that
    \[ 
        \begin{split}
            \lambda_\infty \left(\frac{A_\perf \otimes_{A_e}R^{1/p^e}}{J R_\perf \cap (A_\perf \otimes_{A_e}R^{1/p^e}) }\right) &\ge (1/{p^{ed}})\cdot\length_A\left(\frac{R}{(J^{[p^e]}:g)}\right)\\
            &\ge e_{\HK}(J, R) - \frac{C'}{p^{e}} 
        \end{split}
    \]
    for each $e > 0$ by \cite[Lemma 4.13]{PolstraTuckerCombinedApproach}. This implies $\lambda_\infty(R_\perf/J R_\perf)\ge e_{\HK}(J, R)$.
\end{proof}

\section{Perfectoid signature and perfectoid Hilbert-Kunz multiplicity}
\label{section:mixedchardef}

In this section, we define the $F$-signature and Hilbert-Kunz multiplicity in mixed characteristic. We call them perfectoid signature and perfectoid Hilbert-Kunz multiplicity, and we establish their key properties.
\begin{itemize}
    \item The perfectoid signature is at most 1, while the perfectoid Hilbert-Kunz multiplicity is at least 1. These bounds are established in \autoref{prop.sFboundby1} and \autoref{prop.eFHKgreaterthan1}, respectively.
    \item A Noetherian complete local domain of mixed characteristic $(0,p>0)$ is regular if and only if its perfectoid signature is 1, which is also equivalent to its perfectoid Hilbert-Kunz multiplicity being 1. This is established in \autoref{thm.CharacterizationOfRegularLocalRings}.
\end{itemize}

We briefly outline the key ideas. To define a mixed characteristic version of the $F$-signature and Hilbert-Kunz multiplicity, we follow the approach from the previous section, where these invariants were expressed in terms of the normalized lengths of certain quotients of the perfection of the ring. In mixed characteristic, perfection is replaced by the more general notion of perfectoidization (see \autoref{subsec:PefectoidAlgebras}). Once these invariants are defined, we carefully analyze normalized length in this setting, using \autoref{prop.gisogeny} to control the perfectoidization and establish the desired bounds on the perfectoid signature and perfectoid Hilbert-Kunz multiplicity. To characterize regularity, we rely on the results of \cite{BhattIyengarMaRegularRingsPerfectoid}, which show that regularity is captured by (almost) flat maps to perfectoid rings.

Unless otherwise specified, we work under the following setup throughout this section. 
\begin{notation}
\label{not.setup}
    Let $(R,\m)$ be a Noetherian complete local domain of dimension $d$ and of mixed characteristic $(0,p)$ with perfect residue field $k=R/\m$. Let $\underline{x}:=p, x_2, \dots, x_d$ be a system of parameters of $R$. Then there exists $(A,\m_A):=W(k)[[x_2,\dots,x_d]] \subseteq R$ a module-finite extension by Cohen's structure theorem, note $A$ and $R$ have the same residue field.    
    Let $\Ainfty^{nc}$ be as in \autoref{section:normlength} and $\Ainfty$ be the $p$-adic completion of $\Ainfty^{nc}$. Then $\Ainfty$ is perfectoid (see \autoref{example:PerfectoidRings}). We set 
    $$R^{\Ainfty}_{\perfd}:=(R\otimes_A \Ainfty)_\perfd$$ 
    to be the perfectoidization of $R\otimes_A \Ainfty$. This is a perfectoid ring by \cite[Theorem 10.11]{BhattScholzepPrismaticCohomology} (since $R\otimes_A \Ainfty$ is finite and finitely presented over $\Ainfty$). Moreover, by the universal property of the perfectoidization functor, we have $$\Ainfty\to R^{\Ainfty}_{\perfd} \to \widehat{R^+}.$$
We further let $0\neq g\in A$ be such that $A[1/{g}]\to R[1/{g}]$ is finite \'etale and let $(g)_\perfd$ be the kernel of the canonical map $\Ainfty \to (\Ainfty/g)_\perfd$, see \autoref{subsec:PefectoidAlgebras}. We setup almost mathematics with respect to $(g)_\perfd$ as in \cite[Section 10]{BhattScholzepPrismaticCohomology}.  In particular, for each $n$, $(\Ainfty/p^n,(g)_{\perfd}\Ainfty/p^n)$ produces a classical almost mathematics setup in the sense of \cite[2.1.1]{GabberRameroAlmostringtheory}. In practice, most of our modules will be $p$-power-torsion (and frequently even $p$-torsion), so we may use the classical theory of almost mathematics.  
\end{notation}

\begin{caution}
\label{caut.independence}
    It is worth noting that $\Ainfty$, and hence $R^{\Ainfty}_{\perfd}$, depends on the choice of a regular system of parameters $\underline{x} = p, x_2, \dots, x_d \in A$. Also, the map $\Ainfty \to \widehat{R^+}$ depends on the choice of a compatible system of $p$-power roots of $x_i$ inside $R^+$. 
\end{caution}

We summarize some properties of $R^{\Ainfty}_{\perfd}$ that we will be used throughout, which are essentially consequences of \autoref{thm.BhattScholzeAlmostPurity} and \autoref{prop.BhattScholzeAlmostPuritypcompletefaithfullyflat}.

\begin{theorem}
\label{thm.RperfdAlmostFaithfullyFlatOverA}
With notation as in \autoref{not.setup}, we have 
\begin{enumerate}
    \item $R^{\Ainfty}_{\perfd}$ is $g$-almost flat over $A$, i.e., $(g)_{\perfd}\cdot \Tor^A_{i}(R^{\Ainfty}_{\perfd}, M)=0$ for all $A$-modules $M$ and all $i>0$.
    \item $R^{\Ainfty}_{\perfd}$ is $g$-almost faithful over $A$ in the following sense: if $M\neq 0$, then $(g)_{\perfd}\cdot (R^{\Ainfty}_{\perfd}\otimes_A M)\neq 0$.
\end{enumerate}
\end{theorem}
\begin{proof}
We first prove part (a). By \autoref{prop.BhattScholzeAlmostPuritypcompletefaithfullyflat} applied to $R\otimes_A \Ainfty\to R^{\Ainfty}_{\perfd}$, we see that $R^{\Ainfty}_{\perfd}$ is $p$-completely $g$-almost flat over $\Ainfty$. Since $\Ainfty$ is $p$-completely flat over $A$, we see that for every $n$ and every $p^n$-torsion $A$-module $M$, $(g)_{\perfd}\cdot \Tor^A_{i}(R^{\Ainfty}_{\perfd}, M)=0$ for all $i>0$. 

For the general case, it is enough to show that $R^{\Ainfty}_{\perfd} \otimes_A^L M$ is $g$-almost concentrated in degree zero for any finitely generated $A$-module $M$. Since $A$ is Noetherian and $M$ is finitely generated, we know that $M$ has bounded $p^\infty$-torsion, i.e., there exists $c$ such that $M[p^{\infty}]\cong M[p^c]$. It follows that the pro-systems $\{M/p^nM\}_n$ and $\{M\otimes^L_AA/p^n\}_n$ are pro-isomorphic. Thus we have:
$$R^{\Ainfty}_{\perfd}\otimes^L_AM\cong \mathbf{R}\varprojlim_n(R^{\Ainfty}_{\perfd}\otimes^L_AM\otimes_A^LA/p^n)\cong \mathbf{R}\varprojlim_n(R^{\Ainfty}_{\perfd}\otimes^L_AM/p^nM)$$
where the first isomorphism follows because $R^{\Ainfty}_{\perfd}\otimes^L_AM$ is derived $p$-complete: one can take a finite free resolution of $M$ over $A$ as $A$ is regular, and after tensoring with $R^{\Ainfty}_{\perfd}$ we see that $R^{\Ainfty}_{\perfd}\otimes^L_AM$ is a finite complex such that each term is derived $p$-complete. Now consider the short exact sequence by \cite[\href{https://stacks.math.columbia.edu/tag/0CQE}{Tag 0CQE}]{stacks-project}: 
$$0\to {\myR}^1\varprojlim_n H^{i-1}(R^{\Ainfty}_{\perfd}\otimes_A^L M/p^nM)\to H^i(R^{\Ainfty}_{\perfd}\otimes^L_AM)\to \varprojlim_n H^i(R^{\Ainfty}_{\perfd}\otimes_A^L M/p^nM)\to 0.$$
For every $i<0$, the first and third terms are both derived $p$-complete by \cite[\href{https://stacks.math.columbia.edu/tag/0ARD}{Tag 0ARD}]{stacks-project} and \cite[\href{https://stacks.math.columbia.edu/tag/091U}{Tag 091U}]{stacks-project}, and they are both $g$-almost zero by the $p^n$-torsion case we already established. Hence by \cite[Proposition 10.2]{BhattScholzepPrismaticCohomology}, the middle term is $g$-almost zero for every $i<0$. This completes the proof of part (a).

Now we prove part $(b)$. We pick $0\neq \eta\in M$ and consider $A\eta\cong A/I \hookrightarrow M$. By part $(a)$ we know that the kernel of $R^{\Ainfty}_{\perfd} \otimes_A A/I \to R^{\Ainfty}_{\perfd}\otimes_A M$ is annihilated by $(g)_{\perfd}$, thus if $(g)_{\perfd}\cdot (R^{\Ainfty}_{\perfd}\otimes_A M)=0$, then $R^{\Ainfty}_{\perfd} \otimes_A A/I$ is annihilated by $(g)_{\perfd}^2$. It follows that $R^{\Ainfty}_{\perfd} \otimes_A A/\m_A$ is annihilated by $(g)_{\perfd}^2$ and hence annihilated by $(g)_{\perfd}$ since $R^{\Ainfty}_{\perfd} \otimes A/\m_A$ is $p$-torsion. Thus it is enough to show that $R^{\Ainfty}_{\perfd}/\m_A R^{\Ainfty}_{\perfd}$ is not $g$-almost zero. But if $R^{\Ainfty}_{\perfd}/\m_A R^{\Ainfty}_{\perfd}$ is $g$-almost zero, then $(g)_{\perfd}\subseteq \m_A R^{\Ainfty}_{\perfd}$. Since $R^{\Ainfty}_{\perfd}$ maps to $\widehat{R^+}$, it follows that $(g)_{\perfd}\widehat{R^+} \subseteq \m_A \widehat{R^+}$ and thus $(g^{1/p^\infty})\subseteq \m_A\widehat{R^+}$ by \autoref{lem.gPerfd=AllpPowerRootsofg} and the fact that $p\in\m_A$. But if $g^{1/p^e}\in \m_A R^+$ for all $e$, then we have $$g\in \bigcap_e (\m_A^{p^e} R^+ \cap A)= \bigcap_e \m_A^{p^e}=0,$$ where the first equality follows from the direct summand theorem \cite{AndreDirectsummandconjecture} that $A\to R^+$ is split.\footnote{One does not need the full strength of the direct summand theorem here, only that $R^+$ is $\fram$-adically separated. See \cite[Lemma 4.2]{ShimomotoFCoherent} for a simple proof, or \cite{HochsterBigCohen-Macaulayalgebrasindimensionthree} for the original statement.} Thus $g=0$ which is a contradiction. 
\end{proof}

A direct consequence of \autoref{thm.RperfdAlmostFaithfullyFlatOverA} is the following:

\begin{corollary}
\label{cor.RperfdAlmostCMKoszulandLocal}
With notation as in \autoref{not.setup}, we have $(g)_{\perfd}\cdot H_\m^i(R^{\Ainfty}_{\perfd})=0$ for all $i<d$ and consequently for every system of parameters $\underline{y}$ of the local\footnote{Note that $A_n\otimes_A R$ is a local ring since it is $p$-adically complete and after modulo $p$ it is local (as it is a purely inseparable extension of $R/p$).} ring $A_n\otimes_A R$ we have $(g)_{\perfd}H_i(\underline{y}; R^{\Ainfty}_{\perfd})=0$ for all $i>0$. 
\end{corollary}
\begin{proof}
First note that we have 
$$H_\m^i(R^{\Ainfty}_{\perfd}) \cong \varinjlim_n H_{d-i}(p^n, x_2^n,\cdots, x_d^n; R^{\Ainfty}_{\perfd}) \cong \varinjlim_n\Tor_{d-i}^A(A/(p^n, x_2^n,\cdots, x_d^n), R^{\Ainfty}_{\perfd}),$$ \autoref{thm.RperfdAlmostFaithfullyFlatOverA} immediately implies that $(g)_{\perfd}\cdot H_\m^i(R^{\Ainfty}_{\perfd})=0$ when $i<d$. 

To see the second conclusion, we note that $\Kos(\underline{y}; R^{\Ainfty}_{\perfd})\cong \Kos(\underline{y}; \mathbf{R}\Gamma_\m R^{\Ainfty}_{\perfd})$. It follows that we have a spectral sequence:
$$H_{i+j}(\underline{y}; H_\m^j(R^{\Ainfty}_{\perfd})) \Rightarrow H_i(\underline{y}; R^{\Ainfty}_{\perfd}).$$
Note that if $i>0$, then $i+j>d$ when $j=d$ and thus $H_{i+j}(\underline{y}; -)$ vanish, while if $j<d$ then $H_\m^j(R^{\Ainfty}_{\perfd})$ is annihilated by $(g)_{\perfd}$ as already established. Thus when $i>0$, all the $E_2$-contributions of $H_i(\underline{y}; R^{\Ainfty}_{\perfd})$ are annihilated by $(g)_{\perfd}$. It follows that $H_i(\underline{y}; R^{\Ainfty}_{\perfd})$ is annihilated by a power of $(g)_{\perfd}$. But since $H_i(\underline{y}; R^{\Ainfty}_{\perfd})$ is $p$-power torsion so it is annihilated by $(g)_{\perfd}$.
\end{proof}

\begin{remark}
\label{rmk.AlmostInjectivityLocalCohomology}
With notation as in \autoref{not.setup}, if $\underline{y}=y_1,\dots,y_d$ is a system of parameters of $R$, then for every $d$-tuple $(a_1,\dots,a_d)$, we have a long exact sequence of Koszul homology (see \cite[Theorem 3.4 (i)]{BhattHochsterMaLimCM}):
\begin{align*}
\cdots \to & H_1(y_1^{a_1}, \dots, y_i, \dots, y_d^{a_d}; R^{\Ainfty}_{\perfd}) \to H_0(y_1^{a_1}, \dots, y_i^{a_i}, \dots, y_d^{a_d}; R^{\Ainfty}_{\perfd}) \xrightarrow{\cdot y_i} \\ & H_0(y_1^{a_1}, \dots, y_i^{a_i+1}, \dots, y_d^{a_d}; R^{\Ainfty}_{\perfd})\to H_0(y_1^{a_1}, \dots, y_i, \dots, y_d^{a_d}; R^{\Ainfty}_{\perfd})\to 0.
\end{align*}
In particular, we have
$$R^{\Ainfty}_{\perfd}/(y_1^{a_1},\dots, y_i^{a_i},\dots, y_d^{a_d}) \xrightarrow{\cdot y_i} R^{\Ainfty}_{\perfd}/(y_1^{a_1},\dots,y_i^{a_i+1},\dots, y_d^{a_d})$$
is $g$-almost injective since its kernel is a quotient of $H_1(y_1^{a_1}, \dots, y_i, \dots, y_d^{a_d}; R^{\Ainfty}_{\perfd})$, which is $g$-almost zero by \autoref{cor.RperfdAlmostCMKoszulandLocal}. Thus, by induction, we know that 
$$R^{\Ainfty}_{\perfd}/(y_1,\dots, y_d) \xrightarrow{\cdot (y_1\cdots y_d)^{t-1}} R^{\Ainfty}_{\perfd}/(y_1^{t},\dots, y_d^{t})$$
is $g$-almost injective for all $t$. It follows after passing to a direct limit that the natural map 
$$R^{\Ainfty}_{\perfd}/(y_1,\dots, y_d) \to H_\m^d(R^{\Ainfty}_{\perfd})$$
is $g$-almost injective.
\end{remark}

Inspired by the characteristic $p>0$ definition, we introduce the ideal $I_{\infty}$.
\begin{definition} With notation as in \autoref{not.setup}, we define 
    \label{def.Iinfty}
    \begin{align*}
    I^R_{\infty} &=\{x\in R^{\Ainfty}_{\perfd}\mid R\to R^{\Ainfty}_{\perfd}\text{ such that }1\mapsto x \text{ is not pure}\}\\
    &=\{x\in R^{\Ainfty}_{\perfd}\mid \phi(x)\in \m_R \text{ for any }\phi\in\Hom_R(R^{\Ainfty}_{\perfd},R) \}\\
    &= \left\{ x\in R^{\Ainfty}_{\perfd} \mid x\otimes \eta =0 \text{ in } R^{\Ainfty}_{\perfd} \otimes_R E_R(k) \text{ where } \eta \in E_R(k) \text{ generates the socle}  \right\}.
    \end{align*}
    We will typically write $I_{\infty}$ instead of $I^R_{\infty}$ when no confusion can occur. As mentioned in \autoref{caut.independence}, this notion depends on $A_\infty$, hence on the system of parameters we choose.
\end{definition}

We make the first observation on $I^R_{\infty}$.
\begin{lemma}
\label{lem.CharacterizationIinftyviaAppGorenstein}
With notation as in \autoref{not.setup}, we have 
$$I^R_{\infty}=\bigcup_t (I_tR^{\Ainfty}_{\perfd}: u_t)$$ where $I_1\supseteq I_2\supseteq \cdots \supseteq I_t \supseteq \cdots $ is a descending chain of $\m$-primary ideals cofinal with the powers of $\m$ such that $R/I_t$ is Gorenstein and $u_t$ generates the socle of $R/I_t$.

Moreover, if $R$ is Gorenstein, then for every system of parameters $\underline{y}=y_1,\dots,y_d$ of $R$, we have $I^R_{\infty} \overset{a}{\simeq} (y_1,\dots, y_d){R^{\Ainfty}_{\perfd}} : u$, that is, $$I^R_{\infty} \big/ \left((y_1,\dots, y_d){R^{\Ainfty}_{\perfd}} : u\right)$$ is $g$-almost zero where $u$ generates the socle of $R/(y_1,\dots,y_d)$. 
\end{lemma}
\begin{proof}
The first conclusion follows directly from \autoref{lem.SplittingApproxGorenstein}. If $R$ is Gorenstein and $\underline{y}$ is a system of parameters, then we can take $I_t=(y_1^t,\dots, y_d^t)$ and $u_t=(y_1\cdots y_d)^{t-1}u$. By \autoref{rmk.AlmostInjectivityLocalCohomology},  we know that 
$$R^{\Ainfty}_{\perfd}/(y_1,\dots, y_d) \xrightarrow{\cdot (y_1\cdots y_d)^{t-1}} R^{\Ainfty}_{\perfd}/(y_1^{t},\dots, y_d^{t})$$
is $g$-almost injective for all $t$. It follows that $\left(I_tR^{\Ainfty}_{\perfd}:u_t\right) \big/ \left((y_1,\dots,y_d)R^{\Ainfty}_{\perfd}:u\right)$ is $g$-almost zero for all $t$. Thus so is $I^R_{\infty} \big/ \left((y_1,\dots, y_d){R^{\Ainfty}_{\perfd}} : u\right)$.
\end{proof}

\begin{definition}[Perfectoid signature and perfectoid Hilbert-Kunz multiplicity]
\label{def.PerfdeHKandPerfdSignature}
    With notation as in \autoref{not.setup}, we define the \emph{perfectoid signature of $R$} to be the following:
    \[ 
        s^{\underline{x}}_{\perfd}(R)\coloneqq\lambda_\infty(R^{\Ainfty}_{\perfd}/I_{\infty}).
    \]
    Suppose $J$ is $\m$-primary, we define the \emph{perfectoid Hilbert-Kunz multiplicity of $J$ in $R$} to be:
    \[
        e^{\underline{x}}_{\perfd}(J; R)\coloneqq\lambda_\infty(R^{\Ainfty}_{\perfd}/J R^{\Ainfty}_{\perfd}).
    \]
    In the case that $J = \m$, we simply write $e^{\underline{x}}_{\perfd}(R) := e^{\underline{x}}_{\perfd}(\m; R)$.
    Note that, for both definitions, we include the choice of a regular system of parameters ${\underline{x}} \subseteq A$ in the notation.
\end{definition}

\begin{example}
    If $R=A$ (in which case $R$ is unramified regular), then $R^{\Ainfty}_{\perfd}=\Ainfty$ and $I_{\infty}=\m_A\Ainfty$ (for example by \autoref{lem.IinfinityForFlat}). Thus we have 
    $$
      s^{\underline{x}}_{\perfd}(R) = e^{\underline{x}}_{\perfd}(R) =  \lambda_\infty(\Ainfty/\m_A A_\infty)=1.
    $$
We will see later in this section that $s^{\underline{x}}_{\perfd}(R) = e^{\underline{x}}_{\perfd}(R) = 1$ for every choice of $\underline{x}$ as long as $R$ is regular (possibly ramified), and in fact, the perfectoid signature or the perfectoid Hilbert-Kunz multiplicity equals to $1$ characterize $R$ being regular.
\end{example}
    
When proving facts about Hilbert-Kunz multiplicity (and also $F$-signature), a common observation is that $\length(R/(\m^{[p^e]} : g)) =o( p^{e \dim R})$ for any nonzerodivisor $g \in R$.  We will frequently need similar results in mixed characteristic.  In fact, this can also be interpreted as saying we can define normalized length for almost modules.

\begin{proposition}
\label{prop.AlmostNormalizedLength}
    With notation as in \autoref{not.setup}, suppose that $M,N\in \Ainfty\hyphen\Mod_{\m_A}$ such that $M^a\simeq N^a$ with respect to $(g)_\perfd$\footnote{Here $(g)_\perfd$ actually is $(g)_{\perfd}\Ainfty/p^n$ for some $n$ since $M,N$ are both $\m_A$-power torsion, and so $p^n$-torsion for some $n$.}, then we have $\lambda_\infty(M)=\lambda_\infty(N)$.

\end{proposition}
\begin{proof}
    Using additivity (\autoref{prop.NormalizedLengthProperties} (a)), we may reduce to the case such that $M^a=0$, in which case we need to show that $\lambda_\infty(M)=0$. By definition, we may reduce further to the case where $M$ is finitely generated. By taking a filtration of $M$ and using additivity again, it suffices to show the case where $M$ is cyclic, i.e. $M\simeq \Ainfty/I$ for some ideal $I$. Since $M$ is $\m_A$-power-torsion and $g$-almost zero, $\m_A^N+(g)_\perfd\subseteq I$ and therefore $\lambda_\infty(\Ainfty/(\m_A^N+(g)_\perfd))\ge \lambda_\infty(\Ainfty/I)$. Hence, it suffices to show the case where $M=\Ainfty/(\m_A^N+(g)_\perfd)$. By filtering $M$ by $\m_A^k+(g)_\perfd$ so that each quotient is killed by $\m_A+(g)_\perfd$, we reduce to the case that $M = \Ainfty/(\m_A + (g)_\perfd)$. When $g=p$, then $(g)_\perfd=(p^{1/p^\infty})$ (see \autoref{lem.gPerfd=AllpPowerRootsofg}), and this is done by \cite[Appendix 2]{faltings2002almost} (see also \cite[Lemma 1.1]{MonskyHKFunction} for the characteristic $p>0$ analog). 

We next observe that, since $\Ainfty$ and $\Ainfty/(g)_\perfd$ are perfectoid, for every $n>0$ we have a commutative diagram:
\[\xymatrix{
\Ainfty/p^{1/p^n}  \ar[r]^{\sim} \ar@{->>}[d]  & \Ainfty/p \ar@{->>}[d]\\
\Ainfty/((g)_\perfd, p^{1/p^n})  \ar[r]^{\sim} & \Ainfty/((g)_\perfd, p) 
}\]
where the horizontal maps are isomorphisms induced via the $n$th-iterated Frobenius and the vertical maps are the natural surjections. In particular, chasing this diagram we see that  
     \begin{equation}
        \label{eq.RootsOfGInGPerfdModp}
        \bar{g}^{1/p^n}\in (p^{1/p^n},(g)_\perfd)/p^{1/p^n}
    \end{equation} 

Now if $p$ does not divide $g$,  then as $\left(A_n/((\m_A+(g)_\perfd)\cap A_n\right)\otimes \Ainfty$ surjects onto $\Ainfty/(\m_A+(g)_\perfd)$, we have 
    \begin{align*}
        \lambda_\infty(A_\infty/(\m_A+(g)_\perfd))&\le ({1/p^{nd}})\length_A\big(A_n/((\m_A+(g)_\perfd)\cap A_n)\big)\\
        &\le ({1/p^{n(d-1)}})\length_A\big(A_n/((p^{1/p^n},x_2,\dots,x_d)+(g)_\perfd)\cap A_n\big)\\
        &\le ({1/p^{n(d-1)}})\length_A(A_n/((p^{1/p^n},x_2,\dots,x_d)+(\bar{g}^{1/p^n})))\\
        &\le C/p^{n}
    \end{align*}
    for all $n>0$. Here the second inequality is by taking a filtration; the third inequality is \autoref{eq.RootsOfGInGPerfdModp}; the final inequality is because the penultimate line is computing the Hilbert-Kunz multiplicity of the characteristic $p > 0$ ring $A_n/(p^{1/n}, \overline{g}^{1/p^n}) \cong A/(p, g)$, a ring of dimension $n-2$ (as $p$ does not divide $g$). Let $n\to\infty$ we see that $\lambda_\infty(A_\infty/(\m_A+(g)_\perfd))=0$.

    Finally, in the general case, we can write $g=p^cg'$ for some $g'\in A$ such that $p$ does not divide $g'$. We have the following short exact sequence:
    \[
        0\to (g')_\perfd/(g)_\perfd\to \Ainfty/(g)_\perfd\to \Ainfty/(g')_\perfd\to 0.
    \]
    Tensoring with $\Ainfty/\m_A A_\infty$ over $A_\infty$, we have
    $$((g')_\perfd/(g)_\perfd)\otimes_{A_\infty} \Ainfty/\m_A A_\infty\to \Ainfty/(\m_A,(g)_\perfd)\to \Ainfty/(\m_A,(g')_\perfd)\to 0.$$
    Note that the first term is $p$-almost zero and the third term is $g'$-almost zero, hence both of them have normalized length zero by the already established cases. It follows from \autoref{prop.NormalizedLengthProperties} (a) that
    $$\lambda_\infty(\Ainfty/(\m_A,(g)_\perfd)) =0$$
    and the proposition is proven.
\end{proof}

The converse of the proposition is not true in general, but instead we have the following weaker result \autoref{lem.valmostzero}. Before we state the lemma, we recall the definition of being $v$-almost zero.
\begin{definition}
    \label{def.standardvaluationonAinftyzero}
    Let $v_0$ be the $\m_A$-adic valuation on $A$, and let $v$ be the $\bZ[1/p]$-valued extension of $v_0$ to $\Ainfty$ where for $f \in \Ainfty$, we define $v(f)$ to be the smallest degree of a monomial in $p, x_2, \dots, x_d$ that appears in $f$ (notice that this may be rational, since we have all the $p$-power roots of $p, x_2, \dots, x_d$).
    Then we say a $\Ainfty$-module $M$ is \emph{$v$-almost zero} if for any $m\in M$ there is a sequence $\{f_i\} \subseteq \Ainfty$ such that $f_im=0$ and $v(f_i)\to 0$.
\end{definition}

\begin{remark}
\label{rmk.valuationextendstoRplus}
Note that $v_0$ defined above extends to an $\bR$-valuation on the $\m_A$-adic completion $\widehat{A^+}^{_{\m_A}}$ (and in particular, it extends to an $\bR$-valuation on the $p$-adic completion $\widehat{A^+}$ as this is a subring of $\widehat{A^+}^{_{\m_A}}$, see \cite[Proposition 1.6]{HeitmannR+domain}). We believe this is well-known to experts and there are several ways to see it, for example see \cite[Lemma 1.3]{HeitmannR+domain}. We also include a short explanation here. We note that $v_0$ is determined by a map $(A, \m_A)\to (V, \m_V)$ where $(V,\m_V)$ is the corresponding valuation ring of $v_0$ inside the fraction field of $A$ (and $V$ is a DVR). It follows that we have a commutative diagram 
\[\xymatrix{
A \ar[r]\ar[d] & V\ar[d] \\
\widehat{A^+}^{_{\m_A}} \ar[r] & \widehat{V^+}^{_{\m_V}}
}
\]
induced by completion. Since $\widehat{V^+}^{_{\m_V}}$ is a valuation ring, it determines a valuation on $\widehat{A^+}^{_{\m_A}}$ (that is positive on the unique maximal ideal of $\widehat{A^+}^{_{\m_A}}$) that extends $v_0$. 
\end{remark}

We record the following result of Shimomoto for the convenience of the reader.

\begin{lemma}[{\cite[Proposition 2.15]{ShimomotoFrobeniusActionLocalCohomology}}]
    \label{lem.valmostzero}
    If $M\in \Ainfty\hyphen\Mod_{\m_A}$ is such that $\lambda_\infty(M)=0$, then $M$ is $v$-almost zero.
\end{lemma}
\begin{proof}
    It suffices to prove the lemma in the case where $M=\Ainfty/I$. Suppose for a contradiction that $M$ is not $v$-almost zero. Then there exists some $\epsilon>0$ such that 
    \[ 
        \inf\{v(b)\mid b\in \Ann(\Ainfty/I) = I \}\ge \epsilon.
    \]
    Consider $I_n=\{x\in A_n\mid v(x)\ge \epsilon \}$.  Since $I \cap A_n \subseteq I_n$, we see that $A_n/I\cap A_n$ surjects onto $A_n/I_n$.  By counting monomials we have that $\length_A(A_n/I_n)\ge p^{nd}\epsilon^d$ hence $\length_A( A_n/I\cap A_n)\ge p^{nd}\epsilon^d$. Then by \autoref{lem.Normalizedlengthfglimit}, we have $\lambda_\infty(A_\infty/I_\infty)\ge \epsilon^d$ which is a contradiction.
\end{proof}

The following is our key computation.

\begin{proposition}
\label{prop.Normalizedlengthequalsmultiplicity}
    With notation as in \autoref{not.setup}, let $\underline{y}=y_1,\dots,y_d$ be a system of parameters of $A_n\otimes_A R$.
    Then we have
    \[
 \lambda_\infty(R^{\Ainfty}_{\perfd}/(\underline{y})R^{\Ainfty}_{\perfd}) = \frac{1}{p^{nd}}e(\underline{y}, A_n\otimes_A R),
    \]
where $e(\underline{y}, A_n\otimes_A R)$ is the Hilbert-Samuel multiplicity of $\underline{y}$. In particular, $$\lambda_\infty(R^{\Ainfty}_{\perfd}/\m_AR^{\Ainfty}_{\perfd})=e(\m_A, R)=\rank_A(R).$$
\end{proposition}

\begin{proof}
    First of all, since $\Ainfty\otimes_A R$ is $g$-torsion free, we have a short exact sequence 
    \[
        0\to \Ainfty\otimes_A R\to R^{\Ainfty}_{\perfd} \to C\to 0
    \]
    such that $g^NC=0$ for some $N$ by \autoref{prop.gisogeny} (note that $A_\infty \otimes_AR$ is $p$-torsion free since $p$ is a nonzerodivisor on $R$ and $A_\infty$ if faithfully flat over $A$, so we can apply \autoref{prop.gisogeny}). 
    
Next note that we can pick $y_1',\dots, y_d' \in A_n\otimes_AR$ such that $(y_1',\dots, y_d')=(y_1,\dots,y_d)$ and such that $y_1',\dots, y_{d-1}', g^N$ is a system of parameters of $A_n\otimes_AR$ by prime avoidance. Therefore replacing $\underline{y}$ by $\underline{y}'$ if necessary, we may and we will assume that $y_1,\dots,y_{d-1}, g^N$ is a system of parameters of $A_n\otimes_AR$. Note that there exists $D>0$ such that $y_d^D\in (y_1,\dots, y_{d-1}, g^N)$. We claim the following:
\begin{claim}
$\lambda_\infty(R^{\Ainfty}_{\perfd}/(y_1,\dots,y_{d-1}, y_d^D)R^{\Ainfty}_{\perfd}) = D\cdot \lambda_\infty(R^{\Ainfty}_{\perfd}/(\underline{y})R^{\Ainfty}_{\perfd})$.
\end{claim}
\begin{proof}[Proof of Claim]
We use induction on $D$, $D=1$ is obvious. For $D>1$ we have 
\[
    \xymatrix@R=6pt@C=10pt{
    R^{\Ainfty}_{\perfd}/(y_1,\dots,y_{d-1}, y_d^{D-1})R^{\Ainfty}_{\perfd} \ar[r] & R^{\Ainfty}_{\perfd}/(y_1,\dots,y_{d-1}, y_d^D)R^{\Ainfty}_{\perfd} \ar[r] & R^{\Ainfty}_{\perfd}/(\underline{y})R^{\Ainfty}_{\perfd} \ar[r] & 0.
    }
\]
where $H_1(\underline{y}, R^{\Ainfty}_{\perfd})$ surjects onto the kernel.  
Since $H_1(\underline{y}, R^{\Ainfty}_{\perfd})$ is $g$-almost zero by \autoref{cor.RperfdAlmostCMKoszulandLocal}, we know that $\lambda_{\infty}(H_1(\underline{y}, R^{\Ainfty}_{\perfd}))=0$ by \autoref{prop.AlmostNormalizedLength}. Thus we are done by induction and \autoref{prop.NormalizedLengthProperties}. 
\end{proof}
By the above Claim and the fact that $e((y_1,\dots, y_{d-1}, y_d^D), A_n\otimes_AR)=D\cdot e(\underline{y}, A_n\otimes_AR)$, if we can prove the proposition for $y_1,\dots, y_{d-1}, y_d^D$, that is, if we can show 
$$\lambda_\infty(R^{\Ainfty}_{\perfd}/(y_1,\dots,y_{d-1}, y_d^D)R^{\Ainfty}_{\perfd}) = \frac{1}{p^{nd}}e((y_1,\dots, y_{d-1}, y_d^D),A_n\otimes_A R),$$
then it follows that 
$$\lambda_\infty(R^{\Ainfty}_{\perfd}/(\underline{y})R^{\Ainfty}_{\perfd}) = \frac{1}{p^{nd}}e(\underline{y},A_n\otimes_A R).$$
Thus replacing $\underline{y}$ by $y_1,\dots,y_{d-1}, y_d^D$, we may assume that $y_d\in(y_1,\dots,y_{d-1}, g^N)$. 
    
We next consider the long exact sequence of the induced Koszul homology: 
\begin{align}
\label{LES 1 in proof of key proposition}
 \cdots & \to H_2(\underline{y}, A_\infty\otimes_A R) \to H_2(\underline{y}, R^{\Ainfty}_{\perfd}) \to H_2(\underline{y}, C) \notag \\
  & \to H_1(\underline{y}, A_\infty\otimes_A R) \to H_1(\underline{y}, R^{\Ainfty}_{\perfd}) \to H_1(\underline{y}, C)   \\
 & \to  H_0(\underline{y}, A_\infty\otimes_A R) \to H_0(\underline{y}, R^{\Ainfty}_{\perfd}) \to H_0(\underline{y}, C) \to 0  \notag
\end{align}
Since $H_i(\underline{y},R^{\Ainfty}_{\perfd}))$ is $g$-almost zero for all $i\geq 1$ by \autoref{cor.RperfdAlmostCMKoszulandLocal} and thus has zero normalized length by \autoref{prop.AlmostNormalizedLength}, from the above long exact sequence and \autoref{prop.NormalizedLengthProperties} we know that 
\begin{equation}
\label{equation 1 in proof of key proposition}
\lambda_\infty(H_{i+1}(\underline{y}, C))=\lambda_\infty(H_i(\underline{y}, \Ainfty\otimes_{A}R))= \lambda_\infty(H_i(\underline{y}, A_n\otimes_AR) \otimes_{A_n}\Ainfty)<\infty
\end{equation}
for all $i\geq 1$, and that 
\begin{equation}
\label{equation 2 in proof of key proposition}
\lambda_\infty(H_1(\underline{y}, C)) \leq \lambda_\infty(H_0(\underline{y}, \Ainfty\otimes_AR)) <\infty,
\end{equation}
and that 
\begin{equation}
\label{equation 3 in proof of key proposition}
\lambda_\infty(H_0(\underline{y}, C))+ \lambda_\infty(H_0(\underline{y}, \Ainfty\otimes_AR)) = \lambda_\infty(H_0(\underline{y}, R^{\Ainfty}_{\perfd})))+ \lambda_\infty(H_1(\underline{y}, C)).
\end{equation}

We next consider another long exact sequence of Koszul homology (see \cite[Corollary 1.6.13 (a)]{BrunsHerzog} or \cite[Theorem 3.4 (c)]{BhattHochsterMaLimCM})
\begin{align}
\label{LES 2 in proof of key proposition}
 \cdots &  \to   H_2(y_1,\dots, y_{d-1}; C) \xrightarrow{\cdot y_d} H_2(y_1,\dots, y_{d-1}; C) \to H_2(\underline{y}, C)  \notag \\
 &  \to   H_1(y_1,\dots, y_{d-1}; C) \xrightarrow{\cdot y_d} H_1(y_1,\dots, y_{d-1}; C) \to H_1(\underline{y}, C)  \\
   & \to H_0(y_1,\dots, y_{d-1}; C) \xrightarrow{\cdot y_d} H_0(y_1,\dots, y_{d-1}; C) \to H_0(\underline{y}, C) \to 0.\notag
\end{align}
Since $y_d\in(y_1,\dots, y_{d-1}, g^N)$ and $g^NC=0$, we know that all the multiplication by $y_d$ maps are the zero maps. So the above long exact sequence breaks into short exact sequences, and combining with \autoref{equation 1 in proof of key proposition} and \autoref{equation 2 in proof of key proposition} it follows that 
$$\lambda_\infty(H_i(y_1,\dots, y_{d-1}; C))<\infty$$
for all $i\geq 0$ and that 
$$\lambda_\infty(H_0(\underline{y}, C))=\lambda_\infty(H_0(y_1,\dots, y_{d-1}; C)) <\infty.$$
But then by \autoref{equation 3 in proof of key proposition} we find that 
$\lambda_\infty(H_0(\underline{y}, R^{\Ainfty}_{\perfd})))<\infty,$
since we have shown all the other terms in \autoref{equation 3 in proof of key proposition} have finite normalized length. 

At this point, we have proved that all terms in \autoref{LES 1 in proof of key proposition} and \autoref{LES 2 in proof of key proposition} have finite normalized length. Taking the alternating sum of the normalized length in \autoref{LES 2 in proof of key proposition} (here we are using \autoref{prop.NormalizedLengthProperties}), we find that 
$$\sum_{i=0}^d (-1)^i\lambda_\infty(H_i(\underline{y}, C)) =0.$$
Finally, taking the alternating sum of the normalized length in \autoref{LES 1 in proof of key proposition} (and again using  \autoref{prop.NormalizedLengthProperties}), we have that 
\begin{align*}
    \lambda_\infty(R^{\Ainfty}_{\perfd}/(\underline{y})R^{\Ainfty}_{\perfd})& = \lambda_\infty(H_0(\underline{y}, R^{\Ainfty}_{\perfd})) = \sum_{i=0}^d (-1)^i\lambda_\infty(H_i(\underline{y}, R^{\Ainfty}_{\perfd})) \\
    & = \sum_{i=0}^d (-1)^i\lambda_\infty(H_i(\underline{y}, \Ainfty\otimes_A R)) + \sum_{i=0}^d (-1)^i\lambda_\infty(H_i(\underline{y}, C))  \\
    & = \sum_{i=0}^d (-1)^i\lambda_\infty(H_i(\underline{y}, \Ainfty\otimes_A R)) + 0 \\ 
    & = \sum_{i=0}^d (-1)^i\lambda_\infty(H_i(\underline{y}, A_{n}\otimes_A R) \otimes_{A_n}\Ainfty)\\
    & = \sum_{i=0}^d (-1)^i\frac{\length(H_i(\underline{y}, A_{n}\otimes_A R))}{p^{nd}} =\frac{1}{p^{nd}}e(\underline{y}, A_n\otimes_AR),
\end{align*}
where the last equality follows from \cite[Theorem 4.7.6]{BrunsHerzog}. This completes the proof of the first statement. The second statement follows by applying the first statement to the ideal $\m_AR\subseteq R= A_0\otimes_AR$ (which is generated by a system of parameters). 
\end{proof}

\begin{proposition}
\label{prop.Normalizedlengthequalsrank}
    With notation as in \autoref{not.setup}, suppose $J\subseteq \Ainfty$ is a finitely generated ideal that contains a power of $\m_A$. Then we have 
$$\lambda_\infty(R^{\Ainfty}_{\perfd}/JR^{\Ainfty}_{\perfd})=(\rank_A R)\cdot \lambda_\infty(\Ainfty/J\Ainfty).$$
\end{proposition}
\begin{proof}
Since $J$ is finitely generated and contains a power of $\m_A$, we may assume without loss of generality that $J\subseteq A_n$ is a $\m_{A_n}$-primary ideal of $A_n$ for some $n\gg0$. Note that both functions $\lambda_\infty(R^{\Ainfty}_{\perfd}\otimes_{A_n} -)$ and $(\rank_A R)\cdot \lambda_\infty(\Ainfty\otimes_{A_n} -)$ are additive on short exact sequences of finite length $A_n$-modules. This is obvious for $(\rank_A R)\cdot \lambda_\infty(\Ainfty\otimes_{A_n} -)$ since $A_\infty$ is flat over $A_n$. For $\lambda_\infty(R^{\Ainfty}_{\perfd}\otimes_{A_n} -)$, this follows from the fact that $R^{\Ainfty}_{\perfd}$ is $g$-almost flat over $A_n$ by \autoref{thm.RperfdAlmostFaithfullyFlatOverA} (in that theorem almost flatness was stated over $A$ but the same argument works over $A_n$), and that $g$-almost zero modules have zero normalized length by \autoref{prop.AlmostNormalizedLength}. Thus, by taking a finite filtration of $A_n/J$ by $A_n/\m_{A_n}$, it is enough to prove the proposition when $J=\m_{A_n}$, in which case it follows from \autoref{prop.Normalizedlengthequalsmultiplicity} applied to $\m_{A_n}$ and that \[\frac{1}{p^{nd}}e(\m_{A_n}, A_n\otimes_AR)=\frac{1}{p^{nd}}\rank_A(R)=\rank_A(R)\cdot \lambda_\infty(\Ainfty/\m_{A_n}A_\infty).\qedhere\]
\end{proof}

Before we state the next consequence of \autoref{prop.Normalizedlengthequalsmultiplicity}, we recall that in an arbitrary commutative ring $R$, an ideal $J\subseteq I$ is called a \emph{reduction} of $I$ if there exists an integer $n$ so that $I^{n+1}=JI^n$ (see \cite[Definition 1.2.1]{HunekeSwansonIntegralClosure}). When $I$ is finitely generated (e.g., when $R$ is Noetherian), this is equivalent to saying that $I$ is integral over $J$, see \cite[Corollary 1.2.5]{HunekeSwansonIntegralClosure}. In particular, in a Noetherian local ring $(R,\m)$, if $I$ is $\m$-primary and $J$ is a reduction of $I$, then $e(I, R)=e(J, R)$. A reduction $J\subseteq I$ is called \emph{minimal} if no ideal strictly contained in $J$ is a reduction of $I$. If $(R,\m)$ is a Noetherian local ring with infinite residue field and $I$ is an $\m$-primary ideal, then every minimal reduction of $I$ is generated by a system of parameters, see \cite[Proposition 8.3.7 and Corollary 8.3.9]{HunekeSwansonIntegralClosure}. We need the following lemma to handle the case of finite residue field.

\begin{lemma}
\label{lem.finiteresidue}
Let $(R,\m)$ be a complete local domain of mixed characteristic $(0,p>0)$ with a finite residue field $k=R/\m$ and let $I\subseteq R$ be an $\m$-primary ideal. Then there exists a finite field extension $k\to k'$ so that $R':= R\otimes_{W(k)}W(k')$ is a complete local domain and $IR'$ has a minimal reduction generated by a system of parameters.
\end{lemma}
\begin{proof}
Let $R\to S$ be the normalization of $R$. Note that $(S,\m_S)$ is a normal local domain. Let $k_S=S/\m_S$ be the residue field of $S$. We first claim that if $k\to k'$ is any finite field extension such that $[k':k]$ is relative prime to $[k_S:k]$, then $R'$ is a domain. To see this, we consider the commutative diagram
\[\xymatrix{
R \ar@{^{(}->}[d]\ar[r] & R'= R\otimes_{W(k)}W(k')\ar@{^{(}->}[d] \\
S \ar[r] & S':= S\otimes_{W(k)}W(k')
}.
\]
We know that $S'$ is normal since $W(k)\to W(k')$ is finite \'{e}tale (as $k$ is a perfect field). But after modulo $\m_S$, we have $S'/\m_S S'\cong k_S\otimes_kk'$ which is a field since $[k':k]$ is relative prime to $[k_S:k]$. It follows that $\m_SS'$ is the unique maximal ideal of $S'$ and thus $S'$ is normal and local, in particular a domain. Thus $R'$ is a domain as it injects into $S'$.  

Next we note that by \cite[Proposition 8.2.4]{HunekeSwansonIntegralClosure} (applied to $n=1$), an $\m$-primary ideal $J\subseteq I$ is a reduction of $I$ if and only if the fiber cone 
$$\mathcal{F}_I(R):=R/\m \oplus I/\m I \oplus I^2/\m I^2 \cdots $$
is module-finite over its subalgebra generated by $(J+\m I)/\m I$. In particular, a system of parameters $z_1,\dots,z_d$ generates a minimal reduction of $I$ if and only if their images in $I/\m I$ form a homogeneous system of parameters in the standard graded ring $\mathcal{F}_I(R)$. Now we have $R'$ is a complete local ring with maximal ideal $\m'=\m R'$ and that $\mathcal{F}_I(R')\cong \mathcal{F}_I(R) \otimes_k k'$. Let $\widetilde{k}$ be the direct limit of all finite field extensions $k'$ of $k$ inside an algebraic closure of $k$ so that $[k':k]$ is relative prime to $[k_S:k]$. It is clear that $\widetilde{k}$ is an infinite field. Now since $\mathcal{F}_I(R) \otimes_k \widetilde{k}$ is a standard graded ring over the infinite field $\widetilde{k}$, there exists a homogeneous system of parameters $z_1,\dots,z_d$ of degree 1 in $\mathcal{F}_I(R)\otimes_k \widetilde{k}$. Thus we can take $k'$ sufficiently large and finite over $k$ so that $z_1,\dots,z_d$ are degree 1 elements in $\mathcal{F}_I(R) \otimes_k k'$ (and they necessarily form a system of parameters). It follows that $R'$ is a complete local domain (by the first paragraph of the proof) and that $IR'$ admits a minimal reduction generated by a system of parameters (simply choose any lift of $z_1,\dots, z_d$ in $R'$).
\end{proof}

\begin{corollary}
\label{cor.perfdHKparameterideal}
With notation as in \autoref{not.setup}, for every $\m$-primary ideal $I\subseteq R$ we have $e^{\underline{x}}_{\perfd}(I) \leq e(I)$, with equality when $I$ is generated by a system of parameters of $R$.
\end{corollary}
\begin{proof}
The second conclusion follows directly from \autoref{prop.Normalizedlengthequalsmultiplicity} (applied to $n=0$ so that $A_0\otimes_AR=R$). We now prove the first statement. Note that when the residue field $k$ of $R$ is infinite, then $I$ has a minimal reduction generated by a system of parameters $\underline{y}=y_1,\dots,y_d$. In this case we have 
\[e^{\underline{x}}_{\perfd}(I, R) \leq e^{\underline{x}}_{\perfd}(\underline{y}, R) = e(\underline{y}, R) =e(I, R) \]
where we used the second conclusion for the first equality above. 

Now if the residue field $k$ of $R$ is finite. Then by \autoref{lem.finiteresidue}, there exists a finite field extension $k\to k'$ so that $R':= R\otimes_{W(k)}W(k')$ is a complete local domain and $IR'$ has a minimal reduction generated by a system of parameters $\underline{y}=y_1,\dots,y_d$. We then have \[e^{\underline{x}}_{\perfd}(IR', R') \leq e^{\underline{x}}_{\perfd}(\underline{y}, R') = e(\underline{y}, R') =e(IR', R') =e(I, R) \]
where the first equality follows from the second conclusion (applied to $R'$), and the third equality follows from faithfully flatness of $R\to R'$ and $\m R'=\m_{R'}$ (i.e., taking a filtration of $R/I^n$ by $R/\m$ and tensoring with $R'$ produce a filtration of $R'/I^nR'$ by $R'/\m_{R'}$, thus we see that $\length_R(R/I^n)=\length_{R'}(R'/I^nR')$ for all $n$ and so by the definition of Hilbert-Samuel multiplicity, $e(IR', R') =e(I, R)$). Finally, by \autoref{cor.FiniteEtaleHaveTheSameSignature} from \autoref{sec.TransformationRule}, we know that $e^{\underline{x}}_{\perfd}(IR', R')=e^{\underline{x}}_{\perfd}(I, R)$, which completes the proof.
\end{proof}

We recall that in positive characteristic, it is well-known that we have $$\frac{1}{d!}e(I)\leq e_{\HK}(I)\leq e(I)$$ for every $\m$-primary ideal $I$ (the first inequality is simply because $I^{[p^e]}\subseteq I^{p^e}$). \autoref{cor.perfdHKparameterideal} is a mixed characteristic analog of the second inequality. However, we do not know whether the first inequality holds in mixed characteristic:

\begin{question}
With notation as in \autoref{not.setup}, then do we have $$\frac{1}{d!}e(I)\leq e^{\underline{x}}_{\perfd}(I)$$
for every $\m$-primary ideal $I$? 
\end{question}

We next prove that perfectoid Hilbert-Kunz multiplicity is bounded below by $1$ and that perfectoid signature is bounded above by $1$.

\begin{proposition}
\label{prop.eFHKgreaterthan1}
With notation as in \autoref{not.setup}, we have
    $$e^{\underline{x}}_{\perfd}(R) = \lambda_\infty(R^{\Ainfty}_{\perfd}/\m_R R^{\Ainfty}_{\perfd})\ge 1.$$
\end{proposition}
\begin{proof} We first let
    \[ 
        (\m_A R)^{\lim} := \bigcup_{i \geq 0} (p, x_2, \dots, x_d)^{i+1} :_R (px_2 \cdots x_d)^{i}
    \] 
    be the limit closure of $(p,x_2,\dots,x_d)$ in $R$ (for example, see \cite[Definition 8]{MaQuySmirnovColengthMultiplicity}). By the direct summand theorem, we know that $(\m_AR)^{\lim}$ is an $\m_R$-primary ideal, and we may take a filtration of ideals
    \begin{equation}
    \label{eqn.FiltrationofLimClosure}
        (\m_A R)^{\lim}=I_0\subset I_1\subset \dots\subset I_t=R
    \end{equation}
    where $t=\length_R(R/(\m_A R)^{\lim})\le e(\m_A R,R)=\rank_A R=r$, where we used \cite[Theorem 9]{MaQuySmirnovColengthMultiplicity} for the inequality and \cite[Corollary 4.7.9]{BrunsHerzog} (and the fact that $A$ is regular) for the equality $e(\m_A R,R)=\rank_A R$. 
    By expanding this filtration to $R^{\Ainfty}_{\perfd}$, we get a filtration of $R^{\Ainfty}_{\perfd}$ such that $R^{\Ainfty}_{\perfd}/\m_R R^{\Ainfty}_{\perfd}$ surjects onto $ I_{j+1}R^{\Ainfty}_{\perfd}/I_{j}R^{\Ainfty}_{\perfd}$ for each $j$. Hence we have
    \[
        \lambda_\infty(R^{\Ainfty}_{\perfd}/(\m_A R)^{\lim}R^{\Ainfty}_{\perfd})= \sum \lambda_\infty(I_{j+1}R^{\Ainfty}_{\perfd}/I_{j}R^{\Ainfty}_{\perfd})\le t\cdot \lambda_\infty(R^{\Ainfty}_{\perfd}/\m_R R^{\Ainfty}_{\perfd}).
    \]
    
    By the definition of $(\m_A R)^{\lim}$ and the Noetherianness of $R$, we know that $(\m_A R)^{\lim}= (p^n,x_2^n,\dots,x_d^n):_R(px_2\cdots x_d)^{n-1}$ for all $n \gg 0$. Fix such an $n\gg0$, it follows that
    \begin{equation}
    \label{eqn:gAlmostLimClosure}
        (\m_A R)^{\lim}R^{\Ainfty}_{\perfd}\subseteq (p^n,x_2^n,\dots,x_d^n):_{R^{\Ainfty}_{\perfd}}(px_2\cdots x_d)^{n-1}.
    \end{equation}
    Now since $p, x_2,\dots,x_d$ is a regular system of parameters (in particular a regular sequence) on the regular ring $A$, we have 
    $(p^n,x_2^n,\dots,x_d^n):_A(px_2\cdots x_d)^{n-1}=\m_A$, i.e., an injection 
    $$0\to A/\m_A\xrightarrow{\cdot (px_2\cdots x_d)^{n-1}}A/(p^n,x_2^n,\dots,x_d^n).$$
    Since $R^{\Ainfty}_{\perfd}$ is $g$-almost flat over $A$ by \autoref{thm.RperfdAlmostFaithfullyFlatOverA}, we obtain that
    $$R^{\Ainfty}_{\perfd}/\m_A R^{\Ainfty}_{\perfd}\xrightarrow{\cdot (px_2\cdots x_d)^{n-1}}R^{\Ainfty}_{\perfd}/(p^n,x_2^n,\dots,x_d^n) R^{\Ainfty}_{\perfd}.$$
    is $g$-almost injective and thus $(p^n,x_2^n,\dots,x_d^n):_{R^{\Ainfty}_{\perfd}}(px_2\cdots x_d)^{n-1}$ is $g$-almost isomorphic to $\m_AR^{\Ainfty}_{\perfd}$. Putting this together with \autoref{eqn:gAlmostLimClosure} we have that $(\m_A R)^{\lim}R^{\Ainfty}_{\perfd}\subseteq \m_AR^{\Ainfty}_{\perfd}$ up to $g$-almost.
    But since we clearly have $(\m_A R)^{\lim}R^{\Ainfty}_{\perfd}\supseteq \m_AR^{\Ainfty}_{\perfd}$, it follows that $R^{\Ainfty}_{\perfd}/(\m_A R)^{\lim}R^{\Ainfty}_{\perfd}$ is $g$-almost isomorphic to $R^{\Ainfty}_{\perfd}/\m_A R^{\Ainfty}_{\perfd}$. It then follows that
    \[ 
        \lambda_\infty(R^{\Ainfty}_{\perfd}/\m_A R^{\Ainfty}_{\perfd})\le t\cdot \lambda_\infty(R^{\Ainfty}_{\perfd}/\m_R R^{\Ainfty}_{\perfd})\le r\cdot \lambda_\infty(R^{\Ainfty}_{\perfd}/\m_R R^{\Ainfty}_{\perfd}).
    \]
    By \autoref{prop.Normalizedlengthequalsmultiplicity}, we have $\lambda_\infty(R^{\Ainfty}_{\perfd}/\m_A R^{\Ainfty}_{\perfd}) = r$ and so $\lambda_\infty(R^{\Ainfty}_{\perfd}/\m_R R^{\Ainfty}_{\perfd})\ge 1$ as desired.
\end{proof}

\begin{lemma}
\label{lem.Filtration}
  With notation as in \autoref{not.setup}, let $I$ be an $\m$-primary ideal of $R$ and $I\subseteq J$, then we have 
    \[
        \begin{array}{rcl} 
            \length_R(J/I)\cdot s^{\underline{x}}_{\perfd}(R) & = & \length_R(J/I)\cdot\lambda_\infty(R^{\Ainfty}_{\perfd}/I_{\infty}) \\
            & \leq & \lambda_\infty(JR^{\Ainfty}_{\perfd} /IR^{\Ainfty}_{\perfd}) \\
            & = & e^{\underline{x}}_{\perfd}(I) - e^{\underline{x}}_{\perfd}(J).
        \end{array}
    \]
\end{lemma}
\begin{proof}
The first equality above follows from the definition of perfectoid signature. The other equality above follows from the additivity of the normalized length $\lambda_\infty(JR^{\Ainfty}_{\perfd} /IR^{\Ainfty}_{\perfd}) = \lambda_\infty(R^{\Ainfty}_{\perfd} /IR^{\Ainfty}_{\perfd}) - \lambda_\infty(R^{\Ainfty}_{\perfd} /JR^{\Ainfty}_{\perfd})$ and the definition of perfectoid Hilbert-Kunz multiplicity. For the inequality, we consider a filtration $$I=I_0\subset I_1\subset \dots\subset I_n= J$$ where $n=\length_R(J/I)$ and $I_{j+1}/I_j\simeq R/\m_R\simeq k$. Suppose that $I_{j+1}=(y_{j},I_{j})$. By expanding the filtration to $R^{\Ainfty}_{\perfd}$, we get a filtration of $$IR^{\Ainfty}_{\perfd}=I_0R^{\Ainfty}_{\perfd}\subset I_1R^{\Ainfty}_{\perfd}\subset \cdots\subset I_nR^{\Ainfty}_{\perfd}=JR^{\Ainfty}_{\perfd}.$$ 
We claim that for each $j$ there is a well-defined surjection
    $$I_{j+1}R^{\Ainfty}_{\perfd} \big/ I_jR^{\Ainfty}_{\perfd} \xrightarrow{\psi} R^{\Ainfty}_{\perfd}/I_{\infty}$$
   where $\psi(y_j)=1$. To see this, it is enough to show that if $zy_j\in I_jR^{\Ainfty}_{\perfd}$, then $z\in I_\infty$, but this is clear since otherwise the map $R\to R^{\Ainfty}_{\perfd}$ sending $1$ to $z$ is pure and it follows that $y_j\in I_j$ which is a contradiction. Thus we have 
    \[
        \lambda_\infty(R^{\Ainfty}_{\perfd}/I_{\infty})\le \lambda_\infty(R^{\Ainfty}_{\perfd}/I_jR^{\Ainfty}_{\perfd})-\lambda_\infty(R^{\Ainfty}_{\perfd}/I_{j+1}R^{\Ainfty}_{\perfd})
    \]
    for each $j$. Summing over all $j$ and using the additivity of normalized length, we obtain that 
    \[n\cdot \lambda_\infty(R^{\Ainfty}_{\perfd}/I_{\infty})\le \lambda_\infty(JR^{\Ainfty}_{\perfd}/IR^{\Ainfty}_{\perfd}).\qedhere\]
\end{proof}

\begin{proposition}
    \label{prop.sFboundby1}
    With notation as in \autoref{not.setup}, we have 
    $$s^{\underline{x}}_{\perfd}(R) = \lambda_\infty(R^{\Ainfty}_{\perfd}/I_{\infty})\le 1.$$
\end{proposition}
\begin{proof}
    First note that by Nakayama's lemma, $n:=\length_R(R/\m_AR)\ge \rank_A R=r$. By \autoref{lem.Filtration} applied to $I=\m_AR$ and $J=R$, we have 
 \[r\cdot \lambda_\infty(R^{\Ainfty}_{\perfd}/I_{\infty})\le n\cdot \lambda_\infty(R^{\Ainfty}_{\perfd}/I_{\infty})\le \lambda_\infty(R^{\Ainfty}_{\perfd}/\m_AR^{\Ainfty}_{\perfd}).\]
Dividing by $r$ and using \autoref{prop.Normalizedlengthequalsmultiplicity} completes the proof.
\end{proof}

We next prove that the perfectoid signature can be viewed as the infimum of relative drops of the perfectoid Hilbert-Kunz multiplicity, which is analogous to the characteristic $p>0$ picture \cite{PolstraTuckerCombinedApproach}.  

\begin{proposition}
\label{prop.PerfdSignatureRelativePerfdHK}
   With notation as in \autoref{not.setup}, we have 
    $$s^{\underline{x}}_{\perfd}(R)=\inf_{I\subsetneq J} \frac{e^{\underline{x}}_{\perfd}(I) - e^{\underline{x}}_{\perfd}(J)}{\length_R(J/I)} = \inf_{I\subsetneq J} e^{\underline{x}}_{\perfd}(I) - e^{\underline{x}}_{\perfd}(J).$$
Moreover, if $R$ is Gorenstein, then for every system of parameters $\underline{y}$ of $R$ we have
$$s^{\underline{x}}_{\perfd}(R)= e^{\underline{x}}_{\perfd}(\underline{y}, R) - e^{\underline{x}}_{\perfd}((\underline{y}, u), R)$$
where $u$ is a socle representative of $R/(\underline{y})$. 
\end{proposition}
\begin{proof}
    First note that we have the following inequalities
    $$s^{\underline{x}}_{\perfd}(R)\le \inf_{I\subsetneq J} \frac{e^{\underline{x}}_{\perfd}(I) - e^{\underline{x}}_{\perfd}(J)}{\length_R(J/I)} \le  \inf_{I\subsetneq J} e^{\underline{x}}_{\perfd}(I) - e^{\underline{x}}_{\perfd}(J).$$
    The second one is obvious, and the first one follows from \autoref{lem.Filtration}.
    Hence it suffices to show that 
    $$s^{\underline{x}}_{\perfd}(R) \ge \inf_{I\subsetneq J} e^{\underline{x}}_{\perfd}(I) - e^{\underline{x}}_{\perfd}(J).$$
    By \autoref{lem.CharacterizationIinftyviaAppGorenstein} we have
    $$I^R_{\infty}=\bigcup_t (I_tR^{\Ainfty}_{\perfd}:_{R^{\Ainfty}_{\perfd}} u_t).$$
    Now by \autoref{prop.NormalizedLengthProperties}\autoref{prop.NormalizedLengthProperties.c} (see \autoref{rmk.ApproxGorensteinChain}) we have
    $$s^{\underline{x}}_{\perfd}(R)=\inf_t \lambda_\infty(R^{\Ainfty}_{\perfd}/(I_tR^{\Ainfty}_{\perfd}:_{R^{\Ainfty}_{\perfd}} u_t)).$$
    Hence by the following short exact sequence
    \[ 
        \begin{array}{rl}
        0\to R^{\Ainfty}_{\perfd}/(I_t R^{\Ainfty}_{\perfd}:_{R^{\Ainfty}_{\perfd}} u_t) \to & R^{\Ainfty}_{\perfd}/I_t R^{\Ainfty}_{\perfd}\\
        \to & R^{\Ainfty}_{\perfd}/(I_t+(u_t))R^{\Ainfty}_{\perfd}\to 0
        \end{array}
    \]
    and by \autoref{prop.NormalizedLengthProperties}\autoref{prop.NormalizedLengthProperties.a} we have
    $$s^{\underline{x}}_{\perfd}(R)=\inf_t e^{\underline{x}}_{\perfd}(I_t) - e^{\underline{x}}_{\perfd}(I_t+(u_t)) \ge \inf_{I\subsetneq J} e^{\underline{x}}_{\perfd}(I) - e^{\underline{x}}_{\perfd}(J).$$
This completes the proof of the first statement. 

For the second statement, by \autoref{lem.CharacterizationIinftyviaAppGorenstein} we know that 
\begin{align*}
    s^{\underline{x}}_{\perfd}(R)=\lambda_\infty(R^{\Ainfty}_{\perfd}/I_{\infty}) & =\lambda_\infty\left(R^{\Ainfty}_{\perfd}/((\underline{y}){R^{\Ainfty}_{\perfd}} : u)\right) \\
    & = \lambda_\infty\left(R^{\Ainfty}_{\perfd}/(\underline{y}) R^{\Ainfty}_{\perfd}\right) - \lambda_\infty\left(R^{\Ainfty}_{\perfd}/(\underline{y}, u) R^{\Ainfty}_{\perfd}\right)\\
    &  = e^{\underline{x}}_{\perfd}(\underline{y}, R) - e^{\underline{x}}_{\perfd}((\underline{y}, u), R). \qedhere
\end{align*}
\end{proof}

\subsection{Characterization of regular local rings}
\label{subsec.CharacterizationOfRegularLocalRings}
Our main goal of this subsection is to establish the following characterization of regular local rings.
\begin{theorem}
\label{thm.CharacterizationOfRegularLocalRings}
With notation as in \autoref{not.setup}, the following are equivalent:
\begin{enumerate}
    \item $R$ is regular.
    \item $R^{\Ainfty}_{\perfd}$ is $g$-almost flat and $g$-almost faithful over $R$ (in the sense of \autoref{thm.RperfdAlmostFaithfullyFlatOverA}).
    \item $R^{\Ainfty}_{\perfd}$ is $g$-almost flat over $R$.
    \item $e^{\underline{x}}_{\perfd}(R)=1$ for some (or equivalently, all) choices of $\underline{x}$.
    \item $s^{\underline{x}}_{\perfd}(R)=1$ for some (or equivalently, all) choices of $\underline{x}$.
\end{enumerate}
\end{theorem}
We divided the proof into several parts: we first show that $(a)\Rightarrow(b)\Rightarrow(c)\Rightarrow(a)$; then we show $(a)\Leftrightarrow(d)$; finally we show $(a)\Leftrightarrow(e)$. Before proving \autoref{thm.CharacterizationOfRegularLocalRings}, we need a lemma which is a variant from \cite[Theorem 5.2]{BhattIyengarMaRegularRingsPerfectoid}.

\begin{lemma}
\label{lem.valmostflatimpliesregular}
If $\lambda_\infty(\Tor_i^R(R^{\Ainfty}_{\perfd},k))=0$ for all $i\neq 0$, then $R$ is regular.
\end{lemma}
\begin{proof}

    We follow the argument in \cite[Theorem 5.2]{BhattIyengarMaRegularRingsPerfectoid} to show that $R$ is regular. We set $R_n:=\Kos(p^{1/p^n},x_2^{1/p^n},...,x_d^{1/p^n}; R^{\Ainfty}_{\perfd})$. By induction on $d$ one can show that  
\begin{equation}
\label{eqn.NormalizedLengthofTorR_n}
    \lambda_{\infty}(\Tor^R_{t}(R_n,k))=\lambda_{\infty}(H_{t}(R_n\otimes_R^Lk))=0
\end{equation}
    for all $t\geq d+1$. In fact, when $d=0$ this is saying that $\lambda_\infty(\Tor_{t}^R(R^{\Ainfty}_{\perfd},k))=0$ for $t\geq 1$ which is exactly our assumptions. Now realizing $\Kos(p^{1/p^n},x_2^{1/p^n},...,x_{i+1}^{1/p^n}; R^{\Ainfty}_{\perfd})$ as the cone of the map $$\Kos(p^{1/p^n},x_2^{1/p^n},...,x_i^{1/p^n}; R^{\Ainfty}_{\perfd})\xrightarrow{\cdot x_{i+1}^{1/p^n}} \Kos(p^{1/p^n},x_2^{1/p^n},...,x_i^{1/p^n}; R^{\Ainfty}_{\perfd})$$ 
    and chasing the long exact sequence of $\Tor$ proves \autoref{eqn.NormalizedLengthofTorR_n}.
    
    If $R$ is not regular, then $\Tor_{d+1}^R(k,k)\neq 0$. By \cite[Theorem 5.2 Claim 5]{BhattIyengarMaRegularRingsPerfectoid},\footnote{Note that, in the statement of \cite[Theorem 5.2]{BhattIyengarMaRegularRingsPerfectoid}, there is a valuative condition but that condition is not used in the proof of Claim 5.} we have $\Tor^R_{d+1}(R_n,k)\cong H_0(R_n) \otimes_k \Tor_{d+1}^R(k,k)$ for $n\gg 0$, in particular it contains $H_0(R_n)$ as a direct summand (as $\Tor_{d+1}^R(k,k)$ is a nonzero $k$-vector space). But note that $$H_0(R_n)=R^{\Ainfty}_{\perfd}/(p^{1/p^n},x_2^{1/p^n},...,x_d^{1/p^n})$$ has positive normalized length (in fact, by \autoref{prop.Normalizedlengthequalsmultiplicity}, it has normalized length $r/p^{nd}$). It follows that $\lambda_{\infty}(\Tor^R_{d+1}(R_n,k))>0$ which contradicts \autoref{eqn.NormalizedLengthofTorR_n}.
\end{proof}

Now we prove $(a)\Leftrightarrow(b)\Leftrightarrow(c)$ which is a variant of \cite[Corollary 5.7]{BhattIyengarMaRegularRingsPerfectoid} or a mixed characteristic analog of Kunz's theorem.
\begin{proposition}
\label{prop.MixedKunz}
With notation as in \autoref{not.setup}, $R$ is regular if and only if $R^{\Ainfty}_{\perfd}$ is $g$-almost flat and $g$-almost faithful over $R$, if and only if $R^{\Ainfty}_{\perfd}$ is $g$-almost flat over $R$.
\end{proposition}

\begin{proof}
We first assume $R$ is regular and we show $g$-almost flatness. We will use descending induction on $i$ to show that $\Tor_i^R(R/P, R^{\Ainfty}_{\perfd})$ is $g$-almost zero for all $i\geq 1$ and all $P\in\Spec(R)$. This implies $\Tor_i^R(M, R^{\Ainfty}_{\perfd})$ is $g$-almost zero for all finitely generated $R$-module $M$ by taking a prime filtration, and thus $\Tor_i^R(M, R^{\Ainfty}_{\perfd})$ is annihilated by $(g)_{\perfd}$ for all $R$-modules $M$ by taking a direct limit (i.e., $R^{\Ainfty}_{\perfd}$ is $g$-almost flat over $R$). Now when $i>d$, we know $\Tor_i^R(M, R^{\Ainfty}_{\perfd})$ vanish since $R$ is regular. We next fix $P\in\Spec(R)$ and let $Q=P\cap A$, then it is easy to see that $P$ is a minimal prime of $QR$. In particular, $P$ is an associated prime of the $R$-module $R/QR$ and thus there exists an exact sequence $0\to R/P\to R/QR \to C\to 0$. This induces:
$$\cdots \to \Tor_{i+1}^R(C, R^{\Ainfty}_{\perfd})\to \Tor_i^R(R/P, R^{\Ainfty}_{\perfd})\to \Tor_i^R(R/QR, R^{\Ainfty}_{\perfd}) \to \cdots. $$
The first term is $g$-almost zero by induction while the third term is $g$-almost zero by \autoref{thm.RperfdAlmostFaithfullyFlatOverA} since $\Tor_i^R(R/QR, R^{\Ainfty}_{\perfd})\cong \Tor_i^A(A/Q, R^{\Ainfty}_{\perfd})$, as $R$ is flat over $A$. 

As for the $g$-almost faithfulness, note that by \autoref{prop.eFHKgreaterthan1}, $\lambda_\infty(R^{\Ainfty}_{\perfd}/\m_RR^{\Ainfty}_{\perfd})\neq 0$ and hence $R^{\Ainfty}_{\perfd}/\m_RR^{\Ainfty}_{\perfd}$ is not $g$-almost zero by \autoref{prop.AlmostNormalizedLength}. Now it follows from the same proof as in part (b) of \autoref{thm.RperfdAlmostFaithfullyFlatOverA} that $R^{\Ainfty}_{\perfd}$ is $g$-almost faithful over $R$.

It remains to show that if $R^{\Ainfty}_{\perfd}$ is $g$-almost flat over $R$, then $R$ is regular. The condition implies that for all $i\ge 1$, $\Tor_i^R(R^{\Ainfty}_{\perfd},k)$ is $g$-almost zero and hence $\lambda_\infty(\Tor_i^R(R^{\Ainfty}_{\perfd},k))=0$ for all $i\neq 0$ by \autoref{prop.AlmostNormalizedLength}. So the result follows from \autoref{lem.valmostflatimpliesregular}.
\end{proof}

Using \autoref{lem.valmostflatimpliesregular} and \autoref{prop.MixedKunz} we can prove $(a)\Leftrightarrow(d)$. 

\begin{proposition}
    \label{prop.eFHKequals1regular}
    With notation as in \autoref{not.setup}, we have
    $$e^{\underline{x}}_{\perfd}(R)=1\text{   if and only if   }R\text{ is regular.}$$
\end{proposition}
\begin{proof}
    We first show the `if' part. Since $R$ is regular, it is Cohen-Macaulay. Hence we have a filtration:
    $$\m_A=I_0\subset I_1\subset \dots\subset I_r=R$$ 
    such that $I_{j+1}/I_j\simeq R/\m_R\simeq k$ and $r=\length_R(R/\m_AR)=\rank_A R$. Since $R^{\Ainfty}_{\perfd}$ is $g$-almost flat over $R$ by \autoref{prop.MixedKunz}, we have 
    $$I_{j+1}R^{\Ainfty}_{\perfd}/I_jR^{\Ainfty}_{\perfd} \overset{a}{\simeq} R^{\Ainfty}_{\perfd}/\m_RR^{\Ainfty}_{\perfd}.$$
By \autoref{prop.Normalizedlengthequalsmultiplicity}, we have $r=\lambda_\infty(R^{\Ainfty}_{\perfd}/\m_A R^{\Ainfty}_{\perfd})= r\cdot \lambda_\infty(R^{\Ainfty}_{\perfd}/\m_R R^{\Ainfty}_{\perfd})$. Thus $e^{\underline{x}}_{\perfd}(R)=\lambda_\infty(R^{\Ainfty}_{\perfd}/\m_R R^{\Ainfty}_{\perfd})=1$.

    The strategy for the `only if' part is to measure the flatness of $R\to R^{\Ainfty}_{\perfd}$. We will first show that $R$ is Cohen-Macaulay, and then use a filtration argument to deduce that $\Tor_i^R(R^{\Ainfty}_{\perfd},k)$ has zero normalized length for all $i>0$, and then apply \autoref{lem.valmostflatimpliesregular}.

    Now assume $e^{\underline{x}}_{\perfd}(R)=1$. Let $\m_A^{\lim}$ be the limit closure of $(p,x_2,...,x_d)$ and take the filtration \autoref{eqn.FiltrationofLimClosure} as in \autoref{prop.eFHKgreaterthan1}. Then the same argument as in \autoref{prop.eFHKgreaterthan1} shows that $$\lambda_\infty(R^{\Ainfty}_{\perfd}/\m_A^{\lim}R^{\Ainfty}_{\perfd})\le t\cdot  e^{\underline{x}}_{\perfd}(R)=t$$
    where $t=\length_R(R/\m_A^{\lim})\le e(\m_A,R)=\rank_AR$. Recall that $R^{\Ainfty}_{\perfd}/\m_A^{\lim}R^{\Ainfty}_{\perfd}$ is $g$-almost isomorphic to $R^{\Ainfty}_{\perfd}/\m_A R^{\Ainfty}_{\perfd}$ (see \autoref{eqn:gAlmostLimClosure}). Hence
    $$\lambda_\infty(R^{\Ainfty}_{\perfd}/\m_A^{\lim}R^{\Ainfty}_{\perfd})= \lambda_\infty(R^{\Ainfty}_{\perfd}/\m_A R^{\Ainfty}_{\perfd})=\rank_A R$$
    where the second equality follows from \autoref{prop.Normalizedlengthequalsmultiplicity}.
    It follows that $$\length_R(R/\m_A^{\lim})=t=\rank_A R.$$ Therefore by \cite[Theorem 9]{MaQuySmirnovColengthMultiplicity} (which originates from \cite[Theorem 3.1]{CuongNhanPseudoCM}), $R$ is Cohen-Macaulay. 
    
    Since $R$ is Cohen-Macaulay, we may take the filtration
    $$\m_AR=I_0\subset I_1\subset \cdots \subset I_{r-1}=\m_R$$
    where $r=\rank_A R$. Since $\lambda_\infty(R^{\Ainfty}_{\perfd}/\m_AR^{\Ainfty}_{\perfd})=r$ and $\lambda_\infty(R^{\Ainfty}_{\perfd}/\m_RR^{\Ainfty}_{\perfd})=1$, we must have 
    $\lambda_\infty(R^{\Ainfty}_{\perfd}/I_jR^{\Ainfty}_{\perfd})=r-j$ for each $j$. Consider the short exact sequence
    $$0\to k\simeq I_{j+1}/I_{j} \to R/I_j \to R/I_{j+1}\to 0$$
    By tensoring with $R^{\Ainfty}_{\perfd}$, we get the associated long exact sequence 
    \begin{align*}
        \Tor_1^R(R^{\Ainfty}_{\perfd},R/I_{j})\to \Tor_1^R(R^{\Ainfty}_{\perfd},R/I_{j+1})\to & \\
        R^{\Ainfty}_{\perfd}/\m_RR^{\Ainfty}_{\perfd} \to R^{\Ainfty}_{\perfd}/I_jR^{\Ainfty}_{\perfd} \to & R^{\Ainfty}_{\perfd} /I_{j+1}R^{\Ainfty}_{\perfd}\to 0.
    \end{align*}
    Since $R$ is Cohen-Macaulay, $\m_A R =(p, x_2,\dots,x_d)R$ is generated by a regular sequence and thus  $\Tor_1^R(R^{\Ainfty}_{\perfd},R/\m_A R)=H_1(p,x_2,\dots,x_d; R^{\Ainfty}_{\perfd})$ is $g$-almost zero 
   by \autoref{cor.RperfdAlmostCMKoszulandLocal}. Thus for $j=0$, we have $\lambda_\infty(\Tor_1^R(R^{\Ainfty}_{\perfd},R/I_{0}))=0$ by \autoref{prop.AlmostNormalizedLength}. Now assume by induction that $\lambda_\infty(\Tor_1^R(R^{\Ainfty}_{\perfd},R/I_{j}))=0$. Then by taking the normalized length of the above long exact sequence and combining with $\lambda_\infty(R^{\Ainfty}_{\perfd}/I_jR^{\Ainfty}_{\perfd})=r-j$ for each $j$, we obtain that  $\lambda_\infty(\Tor_1^R(R^{\Ainfty}_{\perfd},R/I_{j+1}))=0$. Therefore by a straightforward induction, we have that $\lambda_\infty(\Tor_1^R(R^{\Ainfty}_{\perfd},k))=0$. It follows (by taking filtration) that $\lambda_\infty(\Tor_1^R(R^{\Ainfty}_{\perfd},N))=0$ for any finite length $R$-module $N$.

    We next use induction to show that $\lambda_\infty(\Tor_i^R(R^{\Ainfty}_{\perfd},k))=0$ for all $i>0$, we have proved this for $i=1$ in the previous paragraph. Let $N=\m_R/\m_AR$. This is an $R$-module of finite length, and by tensoring the short exact sequence $0\to N\to R/\m_AR\to k\to 0$ with $R^{\Ainfty}_{\perfd}$, we obtain that
    $$ \Tor^R_{i+1}(R^{\Ainfty}_{\perfd},R/\m_AR) \to \Tor^R_{i+1}(R^{\Ainfty}_{\perfd},k) \to \Tor^R_{i}(R^{\Ainfty}_{\perfd},N) \to \Tor^R_{i}(R^{\Ainfty}_{\perfd},R/\m_AR). $$
Since $\Tor_i^R(R^{\Ainfty}_{\perfd},R/\m_A R)=H_i(p,x_2,\dots,x_d; R^{\Ainfty}_{\perfd})$ is $g$-almost zero for all $i\neq 0$ by \autoref{cor.RperfdAlmostCMKoszulandLocal}, $\Tor^R_{i+1}(R^{\Ainfty}_{\perfd},k)$ is $g$-almost isomorphic to $\Tor^R_{i}(R^{\Ainfty}_{\perfd},N)$. Thus by induction, we have $\lambda_\infty(\Tor^R_{i}(R^{\Ainfty}_{\perfd},k))=0$ for all $i>0$ (here we also repeatedly used the fact that if $\lambda_\infty(\Tor^R_{i}(R^{\Ainfty}_{\perfd},k))=0$ then $\lambda_\infty(\Tor^R_{i}(R^{\Ainfty}_{\perfd},N))=0$ for any finite length $R$-module $N$ by taking filtration). Now the result follows from \autoref{lem.valmostflatimpliesregular}.
\end{proof}

Finally, we derive $(a)\Leftrightarrow(e)$ from the results above.

\begin{corollary}
    \label{cor.PerfdSignature1IsRegular}
With notation as in \autoref{not.setup}, we have
$$ s^{\underline{x}}_{\perfd}(R)=1\text{   if and only if   }R\text{ is regular.}$$
\end{corollary}
\begin{proof}
Suppose $R$ is regular. Let $\m_R=(\underline{z})=(z_1,\dots,z_d)$ and by \autoref{lem.CharacterizationIinftyviaAppGorenstein} applied to $\underline{y}=\underline{z}$ and $u=1$, we obtain that $I^R_\infty / \m_RR^{\Ainfty}_{\perfd}$ is $g$-almost zero (one can also prove this by \autoref{prop.MixedKunz} and an almost variant of \autoref{lem.IinfinityForFlat}). Thus we have 
$$s^{\underline{x}}_{\perfd}(R)=\lambda_\infty(R^{\Ainfty}_{\perfd}/I^R_\infty)=\lambda_\infty(R^{\Ainfty}_{\perfd}/\m_RR^{\Ainfty}_{\perfd})=e^{\underline{x}}_{\perfd}(R)=1$$
where the last equality follows from \autoref{prop.eFHKequals1regular}.

Now assume $s^{\underline{x}}_{\perfd}(R)=1$. Let $n=\length(R/\m_A R)\geq \rank_AR=r$. We have 
\begin{align*}
    \lambda_\infty\left({R^{\Ainfty}_{\perfd}}/{\m_A R^{\Ainfty}_{\perfd}}\right) & = \lambda_\infty\left({R^{\Ainfty}_{\perfd}}/{\m_R R^{\Ainfty}_{\perfd}}\right) + \lambda_\infty\left({\m_RR^{\Ainfty}_{\perfd}}/{\m_A R^{\Ainfty}_{\perfd}}\right) \\
    & \geq e^{\underline{x}}_{\perfd}(R) +(n-1)\cdot s^{\underline{x}}_{\perfd}(R) \\
    & \geq 1+ (n-1) =n \geq r. 
\end{align*}
where the first inequality follows from \autoref{lem.Filtration} applied to $I=\m_AR$ and $J=\m_R$, and the second inequality follows from \autoref{prop.eFHKgreaterthan1}. Therefore by \autoref{prop.Normalizedlengthequalsmultiplicity}, we must have equality throughout the above. In particular, we must have $e^{\underline{x}}_{\perfd}(R)=1$. Thus the result follows from \autoref{prop.eFHKequals1regular}.
\end{proof}

\begin{remark}
As mentioned in \autoref{not.setup}, after fixing a compatible system of $p$-power roots of $\underline{x}$ in $R^+$, we have a canonical map $R^{\Ainfty}_{\perfd}\to\widehat{R^+}$. The image of this map is also a perfectoid ring (and it is a domain since $\widehat{R^+}$ is so, see \cite{HeitmannR+domain}). However, we caution the readers that, to define the ``correct" perfectoid Hilbert-Kunz multiplicity or perfectoid signature, it is important to work with the perfectoid ring $R^{\Ainfty}_{\perfd}$ in our setup rather than its image in $R^+$. For instance, consider $A=\bZ_p\to R=\bZ_p[p^{1/p}]$, then $\Ainfty= \widehat{\bZ_p[p^{1/p^\infty}]}$, and after fixing a compatible system of $p$-power roots $\{p^{1/p^e}\}_{e=2}^\infty$ of $p^{1/p}$ in $R^+$, we see that the image of $(\Ainfty\otimes_AR)_{\perfd}$ inside $\widehat{R^+}$ is precisely $\widehat{\bZ_p[p^{1/p^\infty}]}$, which is isomorphic to $\Ainfty$. In this case we have
$$\lambda_{\infty}(\widehat{\bZ_p[p^{1/p^\infty}]}/\m_R \cdot \widehat{\bZ_p[p^{1/p^\infty}]})=\lambda_{\infty}(\widehat{\bZ_p[p^{1/p^\infty}]}/p^{1/p} \cdot\widehat{\bZ_p[p^{1/p^\infty}]})=1/p,$$
which is different from the perfectoid Hilbert-Kunz multiplicity of $\m_R$, which by \autoref{def.PerfdeHKandPerfdSignature} and \autoref{prop.Normalizedlengthequalsmultiplicity} is equal to
$$\lambda_{\infty}(R^{\Ainfty}_{\perfd}/\m_RR^{\Ainfty}_{\perfd})=\lambda_{\infty}\left(\frac{\widehat{\bZ_p[p^{1/p^\infty}]}\otimes_{\bZ_p}\bZ_p[p^{1/p}]}{\m_R\cdot (\widehat{\bZ_p[p^{1/p^\infty}]} \otimes_{\bZ_p}\bZ_p[p^{1/p}])}\right)=1.$$
\end{remark}
\section{BCM-regularity via perfectoid signature}
\label{sec.BCM-regularrityPerfdSignature}
In this section, we prove that positivity of perfectoid signature captures BCM-regularity, see \cite{HunekeLeuschkeTwoTheoremsAboutMaximal,AberbachLeuschke} for related results in characteristic $p > 0$.
Throughout this section, we continue to use \autoref{not.setup}.

We briefly summarize the ideas of this section.  First we develop the theory of BCM-regularity in \autoref{subsec.BCMRegularRings} giving a number of equivalent characterizations, which we will use in this section and later in the paper.

In \autoref{subsec.BCMregularityPerfdsignature}, we show BCM-regularity implies positivity of the perfectoid signature.  We note that if the signature is zero, then $R^{A_\infty}_\perfd/I_\infty$ is $v$-almost zero by \autoref{lem.valmostzero}. Under the $\bQ$-Gorensten assumption, we show that the latter implies that the image of the socle representative of $E$ in $H_\m^d(R^+)$ is $v$-almost zero, which contradicts BCM-regularity by applying \autoref{cor.GabberTrick}.
For the reverse direction, we use \autoref{prop.AlmostNormalizedLength} to show that positivity of perfectoid signature implies $R\to R^{A_\infty}_{\perfd}$ sending $1$ to $g_1$ is pure for some $g_1\in (g)_\perfd$, the latter essentially gives an equivalent characterization of BCM-regularity under some hypotheses that we will establish in  \autoref{prop.PerfedPureAlongGImpliesBCMAlongG}. 

We also discuss some variants of BCM-regularity without the $\bQ$-Gorenstein assumption in \autoref{subsec.StrongBCMRegularity} (including some examples), and in \autoref{subsec.BCMrational} we introduced mixed characteristic variants of $F$-rational signature and relate it with BCM-rational singularities, compare with \cite{HochsterYaoFratSignature,SmirnovTuckerRelativeFSignature} in characteristic $p > 0$.

\subsection{BCM-regular rings}
\label{subsec.BCMRegularRings}
In this subsection we recall the definition of and collect some preliminaries on BCM-regularity. We will equate the purity of certain maps from $R$ to $R^{A_\infty}_\perfd$ and $\widehat{R^+}$ respectively in \autoref{prop.PerfedPureAlongGImpliesBCMAlongG}, and this relies on a key result \autoref{lem.ComparisonAlmostKernel} which in turn requires the full use of the almost purity theorem as in \autoref{thm.BhattScholzeAlmostPurity}. 

\begin{definition}
\label{def.weaklyBCMregularRing}
    Suppose that $(R, \m)$ is a Noetherian complete local domain with residue characteristic $p > 0$.  We say that $R$ is \emph{weakly BCM-regular} if for every perfectoid big Cohen-Macaulay $R^+$-algebra $B$, the map $R \to B$ is pure.  If $R$ is $\bQ$-Gorenstein, we say $R$ is \emph{BCM-regular} if it is weakly BCM-regular.
\end{definition}

\begin{remark}
We note that weakly BCM regular rings are the same as weakly $F$-regular rings in characteristic $p > 0$ by \cite[Theorem 11.1]{HochsterSolidClosure} (see also \cite[Theorem on page 250]{HochsterFoundations}).  We will discuss some potential notions of strong BCM-regularity in \autoref{subsec.StrongBCMRegularity} below.
\end{remark}

We next recall BCM-regularity for divisor pairs.

\begin{definition}
    \label{def.BCMRegularityOfPairs}
    Suppose $(R, \m, k)$ is a Noetherian complete normal local  domain and $\Delta \geq 0$ is a $\bQ$-divisor on $\Spec (R)$.    Define 
    \[ 
        R^+(K_R + \Delta) = \underset{h : \Spec (S) \to \Spec (R)}{\colim} S(h^*(K_S + \Delta))
    \]  
    where $R \subseteq S \subseteq R^+$ is a finite extension of $R$ to a normal domain $S$ inside a fixed absolute integral closure $R^+$.  Notice that for sufficiently large $S$ we have that $h^*(K_S + \Delta)$ has integer coefficients and if $K_R + \Delta$ is $\bQ$-Cartier, then for sufficiently large $S$ we have that $h^*(K_S + \Delta)$ is Cartier (in fact linearly equivalent to zero since $S$ is local).
    
    We say that the pair $(R, \Delta)$ is \emph{weakly BCM-regular} if the map
    \[
        E_R(k) = H^d_{\m}(R(K_R)) \to H^d_{\m}(R^+(K_R + \Delta) \otimes_{R^+} B).
    \]
    is injective for every perfectoid big Cohen-Macaulay $R^+$-algebra $B$.  If $K_R + \Delta$ is $\bQ$-Cartier, we call the pair $(R, \Delta)$ \emph{BCM-regular}.  Note that then we may factor the above map as 
    \[
        E_R(k) = H^d_{\m}(R(K_R)) \to H^d_{\m}(R^+(K_R + \Delta)) \cong H^d_{\m}(R^+) \to H^d_{\m}(B) 
    \]
    since $R^+(K_R + \Delta) \cong R^+$.
\end{definition}

In the case that $K_R + \Delta$ is $\bQ$-Cartier, it is not difficult to see that this definition coincides with the one given in \cite[Definition 6.9]{MaSchwedeSingularitiesMixedCharBCM}\footnote{Technically speaking, in \cite[Definition 6.9]{MaSchwedeSingularitiesMixedCharBCM}, we call BCM-regular ``{\it perfectoid} BCM-regular" to emphasize that we only use those big Cohen-Macaulay algebras $B$ that are perfectoid.  We suppress the ``perfectoid'' modifier in this work however as all big Cohen-Macaulay algebras we consider will be perfectoid.} (for example, see \cite[Theorem 6.12 and Proposition 6.14]{MaSchwedeSingularitiesMixedCharBCM} or \cite[Lemma 7.1]{MaSchwedeTuckerWaldronWitaszekAdjoint}).  That definition was arranged to immediately yield certain transformation rules for test ideals under finite maps.  Our formulation here, on the other hand, more closely mimics the definitions of \cite{HaraWatanabeFRegFPure}, see also \cite{TakagiInterpretationOfMultiplierIdeals}.

\begin{remark}
As we already mentioned in \autoref{subsec:BigCohenMacaulayAlgebras}, if $B$ is a perfectoid big Cohen-Macaulay $R$-algebra, then $\widehat{B}^\m$, the $\m$-adic completion of $B$ is perfectoid and balanced big Cohen-Macaulay. Therefore in \autoref{def.weaklyBCMregularRing} and \autoref{def.BCMRegularityOfPairs}, we only need to consider those $B$ that are also balanced. For this reason, throughout the rest of this section we will tacitly assume that all big Cohen-Macaulay algebras are balanced. 
\end{remark}

The first goal of this section is to interpret BCM-regularity in terms of the map $R \to R^{\Ainfty}_{\perfd}$.  We need several preliminary lemmas.

\begin{lemma}
    \label{lem.DominationViaGalois}
   With notation as in \autoref{not.setup} and fix $M$ to be an arbitrary $R$-module. Further, assume that $S$ is the integral closure of $R$ in the Galois closure of $K(R)/K(A)$.  
    For any $R^+$-algebra $B$ and for each $\sigma \in G = \Gal(K(S)/K(A))$, we write $\sigma_* B$ to be the $R$-algebra induced by the map
    \[
        R \to S \xrightarrow{\sigma} S \to B.
    \]
    We may also view this as an $R^+$-algebra by lifting $\sigma$ to $R^+$.  
    Then, there exists a perfectoid big Cohen-Macaulay $R^+$-algebra $B$ such that for every $e > 0$ and every $\sigma\in G$, we have that 
    \[
        \eta \in 0^{B, g^{{1/p^e}}}_M := \ker\Big( M \xrightarrow{z \mapsto g^{1/p^e} \otimes z} B\otimes_RM \Big)
    \]
    if and only if  $\eta \in 0^{\sigma_* B, g^{1/p^e}}_M$, and that if $\eta \in 0^{B', g^{1/p^e}}_M$ for some perfectoid big Cohen-Macaulay $R^+$-algebra $B'$, then $\eta \in 0^{B, g^{1/p^e}}_M$.
\end{lemma}
\begin{proof}
The proof is similar to \cite[Proposition 5.7]{MaSchwedeSingularitiesMixedCharBCM} and it essentially follows from \cite[Theorem 4.9]{MaSchwedeSingularitiesMixedCharBCM}. We first set $$0^{\scr{B}, g^{1/p^e}}_M:= \{\eta\in M \; | \; \eta \in 0^{B', g^{1/p^e}}_M \text{ for some perfectoid big Cohen-Macaulay $R^+$-algebra $B'$}\}.$$
Note that by \cite[Theorem 4.9]{MaSchwedeSingularitiesMixedCharBCM}, $0^{\scr{B}, g^{1/p^e}}_M$ is a submodule of $M$. By the axiom of choice, for each $\eta\in 0^{\scr{B}, g^{1/p^e}}_M$, we choose $B_\eta$ such that $\eta\in 0^{B_\eta, g^{1/p^e}}_M$. Now we apply \cite[Theorem 4.9]{MaSchwedeSingularitiesMixedCharBCM} to the set $\{B_\eta\}_\eta$ to obtain a perfectoid big Cohen-Macaulay $R^+$-algebra $B_e$ that dominates all $B_\eta$. It follows by construction that $0^{\scr{B}, g^{1/p^e}}_M=0^{B_e, g^{1/p^e}}_M$. We can further apply \cite[Theorem 4.9]{MaSchwedeSingularitiesMixedCharBCM} to the set $\{B_e\}_{e\in\mathbb{N}}$ to find a $\widetilde{B}$ that dominates all $B_e$. Then for this $\widetilde{B}$, we know that for every $e>0$, if $\eta \in 0^{B', g^{1/p^e}}_M$ for some $B'$, then $\eta \in 0^{\widetilde{B}, g^{1/p^e}}_M$. 

Finally, we apply \cite[Theorem 4.9]{MaSchwedeSingularitiesMixedCharBCM} to the set $\{\sigma_*\widetilde{B}\}_{\sigma\in G}$ to find a $B$ that dominates all $\sigma_*\widetilde{B}$. Now if $\eta\in 0^{\sigma_*B, g^{{1/p^e}}}_M$, then $\eta\in 0^{\sigma_*\widetilde{B}, g^{{1/p^e}}}_M$ by construction of $\widetilde{B}$, but then $\eta\in 0^{B, g^{{1/p^e}}}_M$ by construction of $B$ (note that we are implicitly using that $\sigma(g^{1/p^e})=ug^{1/p^e}$ for some $p^e$-th root of unity $u$ since $g\in A$, and thus $0^{B, g^{1/p^e}}_M=0^{B, \sigma(g^{1/p^e})}_M$ for all $\sigma$). A similar argument shows that if $\eta\in 0^{B, g^{{1/p^e}}}_M$, then $\eta\in 0^{\sigma_*B, g^{{1/p^e}}}_M$. Therefore $B$ satisfies all the desired properties. 
\end{proof}

\begin{lemma}
\label{lem.ComparisonAlmostKernel}
With notation as in \autoref{not.setup}, let $M$ be a $p^{\infty}$-torsion $R$-module. Consider the following submodules of $M$:
\begin{align*}
    M_1 := \{\eta \in M \; | \; & (g)_{\perfd} \otimes \eta =0 \text{ in } R^{\Ainfty}_{\perfd} \otimes_R M\}.\\
    M_2 :=  \{\eta \in M  \; | \; &   (g^{1/p^\infty}) \otimes \eta =0 \text{ in } R^+ \otimes_R M\}. \\
    M_3 := \{\eta \in M \; | \; & (g^{1/p^\infty}) \otimes \eta =0 \text{ in } B \otimes_R M \text{ for some perfectoid big Cohen-Macaulay}\\ &  \text{$R^+$-algebra $B$}\}.\\
    M_4 := \{\eta \in M \; | \; & 1 \otimes \eta =0 \text{ in } B \otimes_R M \text{ for some perfectoid big Cohen-Macaulay} \\ & \text{$R^+$-algebra $B$} \} 
\end{align*}
Then we have $M_1=M_2=M_3\supseteq M_4$. Moreover, if $M$ is finitely generated, then we have $M_3=M_4$.
\end{lemma}
\begin{proof}
We first note that $M_3$ and $M_4$ are indeed submodules of $M$ by \cite[Theorem 4.9]{MaSchwedeSingularitiesMixedCharBCM}. Moreover, it is clear that $M_2\subseteq M_3$ and $M_3\supseteq M_4$. Since we $R^{\Ainfty}_{\perfd} \to \widehat{R^+}$ is a map of perfectoid rings, by \autoref{lem.gPerfd=AllpPowerRootsofg}, we have $(g)_{\perfd}R^+$ and $(g^{1/p^\infty})R^+$ agree up to $p$-adic closure. But since $M$ is $p^{\infty}$-torsion, we have $R^+\otimes_R M\cong \widehat{R^+}\otimes_RM$, and as any $\eta\in M$ is annihilated by a power of $p$, the image of $(g)_{\perfd} \otimes \eta$ agrees with the image of $(g^{1/p^\infty}) \otimes \eta$ in $R^+\otimes_R M$. Therefore we have $M_1\subseteq M_2$. It remains to show $M_3\subseteq M_1$ and $M_3\subseteq M_4$ when $M$ is finitely generated. 

\noindent{\it Proof of $M_3\subseteq M_1$}: Let $S$ be the normalization of the integral closure of $R$ in the Galois closure of $K(R)/K(A)$ so that $S$ is generically Galois over $A$ with Galois group $G$. Note that we have $R[g^{-1}]\to S[g^{-1}]$ is finite \'etale. Thus by \autoref{prop.BhattScholzeAlmostPuritypcompletefaithfullyflat}, we know that $R^{\Ainfty}_{\perfd}\to S^{\Ainfty}_{\perfd}$ is $p$-completely $g$-almost faithfully flat. 

Suppose $\eta\in M_3$, by enlarging $B$ (via \cite[Theorem 4.9]{MaSchwedeSingularitiesMixedCharBCM}) we may assume that $B$ satisfies the conclusion of \autoref{lem.DominationViaGalois}.
    Consider the following base change:  
    \[
        \xymatrix{
            \Ainfty \ar[d] \ar[r] & B \ar[d]  \\
            S^{\Ainfty}_{\perfd} \ar[r] & S^{\Ainfty}_{\perfd} \widehat{\otimes}_{\Ainfty} B \ar@{=}[d]^{a}\\
            & \prod_{\sigma \in G} \sigma_* B 
        }
    \]
    where the first row is $p$-completely faithfully flat since $A_n$ is regular for all $n$ and $B$ is big Cohen-Macaulay. Thus the second row is $p$-completely faithfully flat by base change. The almost isomorphism above follows from 
    \[ 
        \begin{array}{rcl}
        S^{\Ainfty}_{\perfd} \widehat{\otimes}_{\Ainfty} B \cong S^{\Ainfty}_{\perfd} \widehat{\otimes}_{\Ainfty} S^{\Ainfty}_{\perfd}\widehat{\otimes}_{S^{\Ainfty}_{\perfd}} B & \overset{a}{\cong} & \left( \prod_{\sigma \in G} \sigma_*S^{\Ainfty}_{\perfd}\right) \widehat{\otimes}_{S^{\Ainfty}_{\perfd}} B 
\cong \prod_{\sigma \in G} \sigma_* B,
        \end{array}
    \]
where the $g$-almost isomorphism above is a consequence of the Galois case of \autoref{thm.BhattScholzeAlmostPurity} and the final isomorphism comes from fact that the left $R$-action on $\sigma_*S^{\Ainfty}_{\perfd} \widehat{\otimes}_{S^{\Ainfty}_{\perfd}} B$ identifies that module with $\sigma_*^{-1} B$.
    It follows that $S^{\Ainfty}_{\perfd} \to \prod_{\sigma \in G} \sigma_* B$ is $p$-completely $g$-almost faithfully flat. Next we note that
$$(R^{\Ainfty}_{\perfd} \otimes_RS)_{\perfd}\cong S^{\Ainfty}_{\perfd}.$$
Therefore by \autoref{prop.BhattScholzeAlmostPuritypcompletefaithfullyflat} (applied to the perfectoid ring $R^{\Ainfty}_{\perfd}$ and its finitely presented finite algebra $R^{\Ainfty}_{\perfd} \otimes_RS$, which has constant rank outside $V(g)$ since it is based changed from $R\to S$), we know that 
$R^{\Ainfty}_{\perfd}\to S^{\Ainfty}_{\perfd}$ is $p$-completely $g$-almost faithfully flat. Hence the composition $R^{\Ainfty}_{\perfd} \to \prod_{\sigma \in G} \sigma_* B$ is $p$-completely $g$-almost faithfully flat, in particular it is $p$-completely $g$-almost pure (see the last part of \autoref{prop.BhattScholzeAlmostPuritypcompletefaithfullyflat}). 

By \autoref{lem.DominationViaGalois}, we know that $(g^{1/p^\infty})\otimes\eta =0 $ in  $(\prod_{\sigma \in G} \sigma_* B) \otimes M$. Since $M$ is $p^{\infty}$-torsion, by $p$-complete $g$-almost purity, we see that $$R^{\Ainfty}_{\perfd} \otimes  M \to  (\prod_{\sigma \in G} \sigma_* B) \otimes M$$ is $g$-almost injective. In particular, $(g)_{\perfd} \otimes \eta =0$ in $R^{\Ainfty}_{\perfd} \otimes_R M$ as desired.

\noindent {\it Proof of $M_3\subseteq M_4$ when $M$ is finitely generated}: The proof is very similar to the argument in \autoref{thm.GabbersTrickVVersion} and \autoref{cor.GabberTrick}. Suppose $\eta\in M_3$, i.e., $g^{1/p^\infty}\otimes \eta=0$ in $B\otimes M$ for some $B$. We set $$B':= \widehat{W^{-1}\prod B}$$ where $W$ is the multiplicative set $(g^{\epsilon_1}, g^{\epsilon_2},\dots,g^{\epsilon_n},\dots)$ such that $\epsilon_j\in\bQ_{>0}$ and that $\epsilon_j\to 0$. Then $B'$ is a perfectoid big Cohen-Macaulay $R^+$-algebra (by \autoref{thm.GabbersTrickVVersion}) and we have 
$$B'\otimes M \cong \left(W^{-1}\prod B\right) \otimes M \cong W^{-1} \left((\prod B)\otimes M\right) \cong W^{-1}\left(\prod (B\otimes M)\right)$$
where the first isomorphism follows as $M$ is $p^\infty$-torsion, the second isomorphism follows as localization commutes with tensor product, and the third isomorphism follows since $M$ is finitely generated (and hence finitely presented since $R$ is Noetherian). Since $g^{\epsilon}\otimes \eta=0$ in $B\otimes M$ for all $\epsilon\in \mathbb{Q}_{>0}$ by assumption, it is straightforward to check that 
$$\prod (1\otimes \eta)=0 \text{  in } W^{-1}\left(\prod (B\otimes M)\right).$$
Unwinding the above isomorphisms, this is precisely saying that $1\otimes \eta=0$ in $B'\otimes M$, i.e., we have $\eta\in M_4$ as we wanted.
\end{proof}

As a consequence of \autoref{lem.ComparisonAlmostKernel}, we can prove: 

\begin{proposition}
    \label{prop.PerfedPureAlongGImpliesBCMAlongG}
    With notation as in \autoref{not.setup}, the following are equivalent:
    \begin{enumerate}
        \item There exists $g_1 \in (g)_{\perfd}$ so that 
    $
        R \xrightarrow{1 \mapsto g_1} R^{\Ainfty}_{\perfd}
    $
    is pure.
        \item  There exists $e > 0$ so that 
    $
        R \xrightarrow{1 \mapsto g^{1/p^e}} R^+
    $
    is pure.
    \item  There exists $e > 0$ so that for every perfectoid big Cohen-Macaulay $R^+$-algebra $B$, the map 
    $
        R \xrightarrow{1 \mapsto g^{1/p^e}} B
    $
    is pure.  
    \end{enumerate}
In particular, when the above equivalent conditions hold, $R$ is weakly BCM-regular. 
\end{proposition}
\begin{proof}
Let $E=E_R(k)$ be the injective hull of the residue field of $R$.  We first observe that $(a)$ holds if and only if 
\begin{equation}
\label{eqn:equationinPerfdPureAlongG}
\{ \eta\in E\; |\; (g)_{\perfd}\otimes \eta=0 \text{ in } R^{\Ainfty}_{\perfd} \otimes_R E\}=0.    
\end{equation}
For assuming $(a)$ holds, then $g_1\otimes \eta\neq 0$ for all $\eta \neq 0$ and so \autoref{eqn:equationinPerfdPureAlongG} follows. Conversely, if \autoref{eqn:equationinPerfdPureAlongG} holds then there exists $g_1\in (g)_{\perfd}$ such that $g_1\otimes u\neq 0$ where $u\in E$ is a socle representative. This implies the map $E\to R^{\Ainfty}_{\perfd} \otimes_RE$ sending $\eta\to g_1\otimes \eta$ is injective and hence $(a)$ holds. 

By the same argument, we see that $(b)$ holds if and only if 
$$\{ \eta\in E \;| \; (g^{1/p^\infty})\otimes \eta=0 \text{ in } R^+ \otimes_R E\}=0,$$
and $(c)$ holds if and only if 
$$\{ \eta\in E \; | \; (g^{1/p^\infty})\otimes \eta=0 \text{ in } B \otimes_R E \text{ for some perfectoid big Cohen-Macaulay $R^+$-algebra $B$}\}=0.$$
Thus it follows from \autoref{lem.ComparisonAlmostKernel} that $(a)\Leftrightarrow(b)\Leftrightarrow(c)$.
\end{proof}

In fact, if $R$ is $\bQ$-Gorenstein, then the equivalent conditions in \autoref{prop.PerfedPureAlongGImpliesBCMAlongG} are also equivalent to $R$ being (weakly) BCM-regular, this will be proved in a slightly stronger form and in the pair setting in \autoref{subsec.BCMregularityPerfdsignature}.

\subsection{BCM regularity and perfectoid signature}
\label{subsec.BCMregularityPerfdsignature}
In this subsection, we relate BCM-regularity and perfectoid signature. We will show that under the $\bQ$-Gorenstein assumption, $R$ is BCM-regular precisely when the perfectoid signature is positive.

\begin{proposition}[BCM-regular implies $s_{\perfd} > 0$]
    \label{prop.dH-BCMRegularImpliesPerfdPositive}
    With notation as in \autoref{not.setup}, and suppose $R$ is normal and $\Delta \geq 0$ is a $\bQ$-divisor on $R$ such that $K_R + \Delta$ is $\bQ$-Cartier.  If $(R, \Delta)$ is BCM-regular, then 
    \[
        s^{\underline{x}}_{\perfd}(R) > 0.
    \]
\end{proposition}
\begin{proof}
    It suffices to show that if $s^{\underline{x}}_{\perfd}(R)=0$ then $(R, \Delta)$ is not BCM-regular. By \autoref{lem.valmostzero}, $R^{\Ainfty}_{\perfd}/I_{\infty}$ is $v$-almost zero, that is, there exists a sequence of elements $\{c_i\} \subseteq \Ainfty$ such that $v(c_i)\to 0$ and that the map $R\to R^{\Ainfty}_{\perfd}$ sending $1$ to $c_i$ is not pure.  
    
    Suppose $\eta \in H^d_\m(R(K_R)) = E$ is a socle representative. Then $c_i\otimes \eta =0$ inside $R^{\Ainfty}_{\perfd} \otimes_R E$. Since $R\to \widehat{R^+}$ factors through $R^{\Ainfty}_{\perfd}$, we have that $c_i\otimes \eta =0$ inside $\widehat{R^+}\otimes_R E \cong R^+ \otimes_R E$. It follows that $c_i\otimes \eta=0$ inside
    \begin{align*}
         E \otimes_R R^+(\pi^* \Delta) & \cong  H^d_{\m}(R(K_R)) \otimes_R R^+(\pi^*\Delta) \\ & \cong   H^d_{\m}\big(R(K_R) \otimes_R R^+(\pi^*\Delta)\big) \\
         & \cong  H^d_{\m}\big(R^+(\pi^*(K_R + \Delta))\big),
    \end{align*}
    where $\pi$: $\Spec(R^+)\to\Spec(R)$ is the natural map.
    To see the last isomorphism above, note that while the canonical map $\varphi$: $R(K_R) \otimes_R R^+(\pi^*\Delta) \to R^+(\pi^*(K_R + \Delta))$ need not be an isomorphism, it is an isomorphism over the regular locus of $\Spec(R)$, whose complement has codimension 2 (i.e., after localizing at any height one prime of $R$, $K_R$ is a principal divisor and it is clear that $\varphi$ becomes an isomorphism). In particular, the kernel and cokernel of $\varphi$ have dimension at most $d-2$ so their local cohomology vanish in cohomological degree $d-1$. This implies that the map $H^d_{\m} \big(R(K_R) \otimes_R R^+(\pi^*\Delta)\big) \to H^d_{\m} \big(R^+(\pi^*(K_R + \Delta))\big)$ is an isomorphism.

    Finally, since $K_R + \Delta$ is $\bQ$-Cartier, we have $R^+(\pi^*(K_R + \Delta)) \cong R^+$.  Thus the image of $\eta$ in 
    $H^d_{\m}(R^+(\pi^* (K_R + \Delta))) \cong H^d_{\m}(R^+)$ satisfies $c_i\eta=0$.    Now it follows from  \autoref{cor.GabberTrick} that there exists a perfectoid big Cohen-Macaulay $R^+$-algebra $B$ such that the image of $\eta$ is zero under the map 
    \[
        H^d_{\m}(R^+(\pi^* (K_R + \Delta))) \cong H^d_{\m}(R^+) \to H^d_\m(B).
    \]   
    Therefore, $(R, \Delta)$ is not BCM-regular.  
\end{proof}

\begin{proposition}[$s_{\perfd} > 0$ implies BCM-regular]
    \label{prop.SignaturePositiveImpliesPerturbationBCMRegular}
    With notation as in \autoref{not.setup}, and suppose $s^{\underline{x}}_{\perfd}(R) > 0$.  Then for every nonzero $g\in R^+$, there exists $e \gg 0$ such that for any perfectoid big Cohen-Macaulay $R^+$-algebra $B$, we have that $R \to B$ sending $1 \mapsto g^{1/p^e}$ is pure.
\end{proposition}
\begin{proof}
    Without loss of generality, replacing $g$ with a multiple if necessary, we may assume that $g\in A$ and that $A[g^{-1}] \subseteq R[g^{-1}]$ is \'etale.  By \autoref{prop.PerfedPureAlongGImpliesBCMAlongG}, it suffices to show that there exists $g_1\in(g)_{\perfd}$ such that $R \to R^{\Ainfty}_{\perfd}$ sending $1$ to $g_1$ is pure.  But if this is not the case, then $(g)_{\perfd}\subseteq I_{\infty}$ and hence $s^{\underline{x}}_{\perfd}(R)=\lambda_{\infty}\big(R^{\Ainfty}_{\perfd}/ I_\infty)  = 0$ by \autoref{prop.AlmostNormalizedLength}, which is a contradiction.
\end{proof}

The following variant of what we have done so far will have an application to projective schemes later.
\begin{proposition}
    \label{prop.AlternateTheoremOnBCMRegularityAndPerfdSignaturePositive}
    With notation as in \autoref{not.setup}, suppose $R$ is normal and $\Delta \geq 0$ is a $\bQ$-divisor on $\Spec(R)$ such that $K_R + \Delta$ is $\bQ$-Cartier, and $S \supseteq R$ is a finite normal local extension such that the pullback $\pi^*(K_R + \Delta)$ to $\Spec(S)$ is Cartier,  where $\pi : \Spec(S) \to \Spec(R)$ is the induced map. Let $0\neq h\in A$ so that $A[1/h]\to S[1/h]$ is finite \'etale. Further suppose the following: 
    \begin{enumerate}
        \item there is an $(A_\infty\otimes_AS)$-algebra $T$ together with a map $T \to \widehat{R^+}$ where $T$ is perfectoid and contains a compatible system of $p$-power roots of $h$, $\{h^{1/p^e} \}$.
        \item the map $H^d_{\m}(R(K_R)) \xrightarrow{1 \mapsto h^{1/p^e}} H^d_{\m}(T \otimes_{S} S(\pi^*(K_R + \Delta)))$ is injective for $e\gg0$. 
    \end{enumerate}
    Then $(R, \Delta)$ is BCM-regular and in particular
    \[
        s_{\perfd}^{\underline{x}}(R) > 0.
    \]
\end{proposition}
\begin{proof}
Let $B$ be a perfectoid big Cohen-Macaulay $R^+$-algebra. Thus $B$ is a $T$-algebra and we have natural maps:
$$H^d_{\m}(R(K_R)) \to  H^d_{\m}(S(\pi^*(K_R + \Delta)))\to  H^d_{\m}(T \otimes_{S} S(\pi^*(K_R + \Delta))) \to  H^d_{\m}(B \otimes_{S} S(\pi^*(K_R + \Delta))).$$
Let $u$ be a socle representative of $H^d_{\m}(R(K_R))$ and let $u'$ be its image in 
$$H^d_{\m}(S(\pi^*(K_R + \Delta)))\cong H_\m^d(S).$$

If $(R,\Delta)$ is not BCM-regular, then for some perfectoid big Cohen-Macaulay $R^+$-algebra $B$, we have $1\otimes u'=0$ in $H^d_{\m}(B \otimes_{S}S(\pi^*(K_R + \Delta)))\cong H_\m^d(B)$ (technically, its image is zero, a slight abuse of notation we suppress). By \autoref{lem.ComparisonAlmostKernel}, we know that $(h)_\perfd \otimes u'=0$ in $H^d_{\m}(S^{\Ainfty}_{\perfd} \otimes_{S}S(\pi^*(K_R + \Delta))) \cong H_\m^d(S^{\Ainfty}_{\perfd})$, which implies that $(h)_\perfd\otimes u'=0$ in 
$$H^d_{\m}(T \otimes_{S} S(\pi^*(K_R + \Delta))) \cong H_\m^d(T)$$ since $S^{\Ainfty}_{\perfd}$ maps to $T$ by the universal property of perfectoidization (as $T$ is an algebra over $A_\infty\otimes_AS$). But since $T$ contains a compatible system of $p$-power roots of $h$ by assumption, we know that $((h)_{\perfd}T)^-=(h^{1/p^\infty})^-$ by \autoref{lem.gPerfd=AllpPowerRootsofg}. It follows that $(h^{1/p^\infty})\otimes u'=0$ in $H^d_{\m}(T \otimes_{S} S(h^*(K_R + \Delta)))$. But then the hypothesis that 
$$H^d_{\m}(R(K_R)) \xrightarrow{1 \mapsto h^{1/p^e}} H^d_{\m}(T \otimes_{S} S(\pi^*(K_R + \Delta)))$$ injects for $e \gg 0$ implies that $u=0$ which is a contradiction.
\end{proof}

\begin{remark}
    In characteristic $p > 0$ our results \autoref{prop.dH-BCMRegularImpliesPerfdPositive} and \autoref{prop.SignaturePositiveImpliesPerturbationBCMRegular} give an alternative proof of the fact that the $F$-signature is positive, $s(R) > 0$, if and only if $R$ is strongly $F$-regular, a result of \cite{AberbachLeuschke} \cf \cite{HunekeLeuschkeTwoTheoremsAboutMaximal,PolstraTuckerCombinedApproach,PolstraATheoremAboutMCM}.  Indeed, if $R$ is strongly $F$-regular, then by \cite{SchwedeSmithLogFanoVsGloballyFRegular}, there exists a $\Delta \geq 0$ such that $(R, \Delta)$ is log $\bQ$-Gorenstein strongly $F$-regular and hence BCM-regular (see \cite{MaSchwedeTuckerWaldronWitaszekAdjoint}).  Hence  the conditions of \autoref{prop.dH-BCMRegularImpliesPerfdPositive} are satisfied.  Conversely, if $s(R) > 0$ then by \autoref{prop.SignaturePositiveImpliesPerturbationBCMRegular} we have proven that $R$ is strongly $F$-regular (note that $R_{\perf} \subseteq R^+$, thus splitting in the latter implies splitting in the former).  Of course, our approaches are inspired by the characteristic $p > 0$ theory.
\end{remark}

We also prove a related result which connects perfectoid big Cohen-Macaulay algebras and perfectoid Hilbert-Kunz multiplicities. This should be viewed as an analog of the characteristic $p>0$ fact that for two $\m$-primary ideals $I\subseteq J$, $e_{\HK}(I)=e_{\HK}(J)$ if and only if $I$ and $J$ have the same tight closure \cite[Theorem 8.1.7]{HochsterHunekeTC1}. 

\begin{proposition}
\label{prop.HKVsBCMClosureVsepfClosure}
With notation as in \autoref{not.setup}, then for every pair of $\m$-primary ideals $I\subseteq J$, the following are equivalent: 
\begin{enumerate}
    \item $e^{\underline{x}}_{\perfd}(I) = e^{\underline{x}}_{\perfd}(J)$.
    \item There exists a perfectoid big Cohen-Macaulay $R^+$-algebra $B$ such that $IB=JB$.
    \item $I^{epf}=J^{epf}$ where $(-)^{epf}$ denotes (full) extended plus closure.
\end{enumerate}
In particular, $R$ is weakly BCM-regular if and only if every $\m$-primary ideal $I$ is extended plus closed (i.e., $I^{epf}=I$).
\end{proposition}

\begin{proof}
Without loss of generality, we may assume that $J=I+(u)$. 

\noindent$(a)\Rightarrow(b)$: If $e^{\underline{x}}_{\perfd}(I) = e^{\underline{x}}_{\perfd}(J)$, then $\lambda_{\infty}\left({(I+u)R^{\Ainfty}_{\perfd}}/{IR^{\Ainfty}_{\perfd}}\right)=0.$ By \autoref{lem.valmostzero}, we know that there exists a sequence $\{c_i\}\subseteq \Ainfty$ such that $v(c_i)\to 0$ and such that $c_iu\in IR^{\Ainfty}_{\perfd}$ for all $i$. Since $R^{\Ainfty}_{\perfd}$ maps to $\widehat{R^+}$, we have $c_iu\in I\widehat{R^+}$. Now by \autoref{thm.GabbersTrickVVersion}, there exists a perfectoid big Cohen-Macaulay $R^+$-algebra $B$ such that $u\in IB$. 

\noindent$(c)\Rightarrow(b)$: By the definition of extended plus closure, if $u\in I^{epf}$, then there exists $0\neq c\in R$ such that $c^{1/n}u\in I\widehat{R^+}$ for all $n$. Since $I^{epf}$ is a finitely generated ideal, we may assume there exists $0\neq c\in R$ so that $c^{1/n}I^{epf}\subseteq I\widehat{R^+}$ for all $n$. Thus again by \autoref{thm.GabbersTrickVVersion}, there exists a perfectoid big Cohen-Macaulay $R^+$-algebra $B$ such that $I^{epf}\in IB$.

\noindent$(b)\Rightarrow(a),(c)$: Suppose $u\in IB$ for some $B$, by \autoref{lem.ComparisonAlmostKernel} applied to $M=R/I$, we have $(g)_{\perfd}u\in IR^{\Ainfty}_{\perfd}$. By \autoref{prop.AlmostNormalizedLength} we have that $\lambda_{\infty}\left({(I+u)R^{\Ainfty}_{\perfd}}/{IR^{\Ainfty}_{\perfd}}\right)=0$ and thus $e^{\underline{x}}_{\perfd}(I) = e^{\underline{x}}_{\perfd}(J)$. Since $R^{\Ainfty}_{\perfd}$ maps to $\widehat{R^+}$, we also have $(g)_{\perfd}u\in I\widehat{R^+}$, and by \autoref{lem.gPerfd=AllpPowerRootsofg}, it follows that $g^{1/p^e}u\in I\widehat{R^+}$ for all $e>0$, which implies $u\in I^{epf}$ by \cite[Lemma 3.1]{HeitmannMaExtendedPlusClosure}.

Now we prove the last statement. Suppose $R$ is weakly BCM-regular. Now if $I^{epf}=J^{epf}$, then $IB=JB$ for some perfectoid big Cohen-Macaulay algebra $B$ via $(b)\Leftrightarrow(c)$, and thus $I=IB\cap R=JB\cap R=J$ since $R\to B$ is pure. Conversely, suppose every $\m$-primary ideal is extended plus closed. If $R$ is not weakly BCM-regular, then there exists $B$ such that $R\to B$ is not pure, i.e., $E_R(k)\to B\otimes_RE_R(k)$ is not injective. But as $R$ is approximately Gorenstein, we know that $E_R(k)=\varinjlim_t R/I_t$ (see \autoref{subsec:PureSplitMap}). Thus there exists $I_t$ such that $R/I_t\to B\otimes_R R/I_t$ is not injective for some $t$. It follows that $J:=I_tB\cap R\supsetneq I_t$. But since $I_tB=JB$, it follows from $(b)\Leftrightarrow(c)$ that $I_t^{epf}=J^{epf}$ and thus $I_t=J$ which is a contradiction. 
\end{proof}

\subsection{Strong BCM-regularity}
    \label{subsec.StrongBCMRegularity}
    So far we have seen weak BCM-regularity and in the (log) $\bQ$-Gorenstein case, BCM-regularity.  In analogy with the equal characteristic $p>0$ setting, there are a number of potential definitions of strong BCM-regularity that one might consider.  

    \begin{definition}
        \label{def.StrongBCM-Regularity}
        Suppose $(R, \m)$ is a Noetherian complete local domain of residual characteristic $p > 0$.  
        \begin{enumerate}
            \item We say that $(R, \m)$ is \emph{perturbation-weakly-BCM-regular} if for every $0 \neq g \in R^+$ there exists $e \gg 0$ such that $R \xrightarrow{1 \mapsto g^{1/p^e}} B$ is pure for every perfectoid big Cohen-Macaulay $R^+$-algebra $B$.\label{def.StrongBCM-Regularity.a}
            \item We say that $(R, \m)$ is \emph{signature-BCM-regular} if $ s^{\underline{x}}_{\perfd}(R) > 0$.\label{def.StrongBCM-Regularity.b}
            \item We say that $(R, \m)$ is \emph{$v$-BCM-regular} if for some extension of $v$ to $\widehat{R^+}$ extended as a $\bQ$-valuation from the $\m$-adic order on $A$, there exists $\epsilon$ such that for all $c\in R^+$ such that $v(c)<\epsilon$, the map $R \xrightarrow{1 \mapsto c} B$ is pure for every perfectoid big Cohen-Macaulay $R^+$-algebras $B$. \label{def.StrongBCM-Regularity.c}  
            \item We say that $(R, \m)$ is \emph{de Fernex-Hacon (dFH)-BCM-regular} (or \emph{BCM-regular type}) if there exists an effective $\bQ$-divisor $\Delta$ on $\Spec (R)$ such that $K_R + \Delta$ is Cartier and $(R, \Delta)$ is BCM-regular.\label{def.StrongBCM-Regularity.d}
        \end{enumerate}
    \end{definition}

    The arguments earlier in this section, have established that 
    \begin{equation}
        \text{\autoref{def.StrongBCM-Regularity.d}} \Rightarrow \text{\autoref{def.StrongBCM-Regularity.c}} \Rightarrow \text{\autoref{def.StrongBCM-Regularity.b}} \Rightarrow \text{\autoref{def.StrongBCM-Regularity.a}} \Rightarrow \text{(weakly BCM-regular)} \Rightarrow \text{(splinter)}.
    \end{equation}
More precisely, $\text{\autoref{def.StrongBCM-Regularity.d}} \Rightarrow \text{\autoref{def.StrongBCM-Regularity.c}}$ is contained in the proof of \autoref{prop.dH-BCMRegularImpliesPerfdPositive}; $\text{\autoref{def.StrongBCM-Regularity.c}} \Rightarrow \text{\autoref{def.StrongBCM-Regularity.b}}$ follows from \autoref{lem.valmostzero}; $\text{\autoref{def.StrongBCM-Regularity.b}} \Rightarrow \text{\autoref{def.StrongBCM-Regularity.a}}$ is exactly the content of \autoref{prop.SignaturePositiveImpliesPerturbationBCMRegular}; and, 
finally, $\text{\autoref{def.StrongBCM-Regularity.a}} \Rightarrow \text{(weakly BCM-regular)} \Rightarrow \text{(splinter)}$ are obvious from the definitions. 
    Furthermore, if $R$ is $\bQ$-Gorenstein then we may take $\Delta = 0$ to see that weakly BCM-regular implies (dFH)-BCM-regular, in particular we have  $\autoref{def.StrongBCM-Regularity.a} - \autoref{def.StrongBCM-Regularity.d}$ all coincide in this case. 
    
    While these potential different notions of strong BCM-regularity will \emph{not} be important for us going forward, we do hope that they all coincide.  Note in characteristic $p > 0$, weakly BCM-regular coincides with weakly $F$-regular, and at least in the $F$-finite case, \autoref{def.StrongBCM-Regularity.a} implies strong $F$-regularity which then implies \autoref{def.StrongBCM-Regularity.d} by \cite{SchwedeSmithLogFanoVsGloballyFRegular}. Hence $\autoref{def.StrongBCM-Regularity.a} - \autoref{def.StrongBCM-Regularity.d}$ all coincide in characteristic $p>0$ (also see \cite{HochsterHunekeTightClosureAndElementsOfSmallOrder} for the implication \autoref{def.StrongBCM-Regularity.a} $\Rightarrow$ \autoref{def.StrongBCM-Regularity.c}). The remaining converses, namely whether weakly BCM-regular implies $\autoref{def.StrongBCM-Regularity.a} - \autoref{def.StrongBCM-Regularity.d}$, and whether splinter implies weakly BCM-regular, become the famous ``weak implies strong'' and ``splinter implies weak'' conjectures in characteristic $p>0$.  Those two conjectures are open even for local domains essentially of finite type over an algebraically closed field.

Below we list some examples of BCM-regular singularities in mixed characteristic, see \cite{MaSchwedeSingularitiesMixedCharBCM, MaSchwedeTuckerWaldronWitaszekAdjoint,BMPSTWW-MMP} for more details.

\begin{example}
\begin{enumerate}
    \item If $R=\mathbb{Z}_p\llbracket x_2,\dots,x_d\rrbracket/(p^n+x_2^n+\cdots + x_d^n)$ where $n<d$. Then for $p\gg 0$, $R$ is BCM-regular \cite[Example 7.8]{MaSchwedeTuckerWaldronWitaszekAdjoint}.
    \item Suppose $(R,\Delta)$ is a two-dimensional klt pair of residue characteristic $p>5$, and $\Delta\geq 0$ has standard coefficients. Then $(R,\Delta)$ is BCM-regular by \cite[Theorem 7.11]{MaSchwedeTuckerWaldronWitaszekAdjoint}, see also \cite[Corollary 7.16]{BMPSTWW-MMP}.
    \item If $R$ is a direct summand of a complete regular local ring $S$ then $R$ is always perturbation-weakly-BCM-regular by $(b)\Leftrightarrow(c)$ in \autoref{prop.PerfedPureAlongGImpliesBCMAlongG}. In particular, if $R$ is also $\bQ$-Gorenstein, then $R$ is BCM-regular. 
\end{enumerate}
\end{example}

Next, we present an example of a class of rings $(R,\frak{m})$ which are (dFH)-BCM-regular. Suppose that $\cM\subseteq \bN^k$ be a finitely generated monoid and $I_\cM\subset W(k)\llbracket\cM\rrbracket$ be the ideal generated by all non-unit monomials of $\cM$. Here, $I_\cM$ is the maximal ideal of $W(k)\llbracket\cM\rrbracket$. Take $p-f\in W(k)\llbracket \cM \rrbracket$ such that the image of $p-f$ in $W(k)\llbracket\cM\rrbracket/(I_\cM+(p^2))$ is $p$. We will prove that if $R$ is of the form $ W(k)\llbracket \cM \rrbracket$ or $W(k)\llbracket \cM \rrbracket/(p-f)$, then $R$ admits an effective divisor $\Delta$ such that $K_R + \Delta$ is $\bQ$-Cartier and $(R,\Delta)$ is BCM-regular. For example, complete log-regular local rings with perfect residue field of characteristic $p>0$ satisfies this property, see \cite[Chapter 3, Theorem 1.11.2]{OgusLogarithmicAlgebraicGeometry}, \cite[
Theorem 2.2]{ShimomotoEtAlPerfectoidTowersAndTheirTilts}. This is slightly stronger than the fact that such an $R$ is a splinter, which is proven in \cite{GabberRameroAlmostRingsAndPerfectoidSpaces}, \cf \cite{ShimomotoEtAlPerfectoidTowersAndTheirTilts}. To begin with, we summarize notations on toric rings, see also \cite{CoxLittleSchenckToricVarieties,FultonToric,HochsterRingsInvariantsofTori,MustataToricNotes}.

\begin{definition}
Let $\cM$ be a monoid. We denote $\cM^*$ to be the set of elements $x$ in $\cM$ such that there exists $y\in \cM$ with $x+y=0$. Furthermore, we write $\cM^{\mathrm{gp}}:=\left\{v - w \mid v,w\in \cM\right\}$.
\end{definition}

\begin{definition}
We say a monoid $\cM$ is

\begin{itemize}

    \item \emph{strongly convex} if $\cM^* = 0$.

    \item \emph{integral} if the cancellation law holds.
    \item \emph{saturated} if $\cM$ is integral and for $ v \in \cM^{\mathrm{gp}}$ there is some $n\in \N$ such that $nv\in \cM$, then $v\in \cM$.


     \item \emph{normal} if $\cM$ is finitely generated and saturated. 
    
\end{itemize}

A strongly convex monoid $\cM$, with an embedding $\cM \subseteq \bN^k$ (that is, with a choice of variables if working multiplicatively), is

\begin{itemize}
   
    \item \emph{full} if $\cM$ is finitely generated and for $v$ and $w\in \cM$ and $u \in \bN^k$, if $u+v = w $, then $u \in \cM$.
\end{itemize}

\end{definition}

Suppose that $\cM$ is a finitely generated, integral monoid and $\cM^{\mathrm{gp}}$ is torsion-free, then $W(k)[\cM]$, equipped with the multiplicative monoid, is an affine semigroup ring, and hence it is a domain (see \cite[Section I.3]{OgusLogarithmicAlgebraicGeometry}, \cite[Chapter 1]{MustataToricNotes}, also \cite[Proposition 1.1.14]{CoxLittleSchenckToricVarieties}). A strongly convex monoid $\cM$ is normal if and only if the corresponding ring $W(k)[\cM]$ is normal \cite[Theorem 1.3.5]{CoxLittleSchenckToricVarieties}. Suppose that $\cM$ is a monoid of monomials in variables $x_1, \dots, x_n$ (in other words, with an embedding $\cM \subseteq \bN^k$), then $\cM$ is normal if and only if $\cM$ is isomorphic to a full monoid of monomials with possibly a different choice of variables, see \cite[Proposition 1]{HochsterRingsInvariantsofTori}. Let $\sigma$ be a rational polyhedral cone with $\sigma^\vee \cap \N^k = \cM$. 
The first lattice point $v_1,\dots, v_t$ of edges of $\sigma$ define divisors on $R=W(k)[\cM]$. Suppose that $D_i$ is the prime divisor corresponding to  $v_i^\bot \cap \sigma^\vee$. For $m\in \cM$, if $D = \Div_R (m)$ is a $\bQ$-Cartier divisor, then $D=\sum \langle m, v_i \rangle D_i$ as the sum of divisors defined by the first lattice point of edges of $\sigma$. In general, one can choose the canonical divisor $K_R = -\sum_{i=1}^t D_i$ under these circumstances \cite[Lemma 2.5]{RobinsonToricBCM}.

    \begin{proposition}[\cf \cite{BlickleMultiplierIdealsAndModulesOnToric,GabberRameroAlmostRingsAndPerfectoidSpaces,RobinsonToricBCM,ShimomotoEtAlPerfectoidTowersAndTheirTilts}]
        Let $\cM$ be a strongly convex normal monoid and $k$ is a perfect field of characteristic $p > 0$. Let $I_\cM$ be the ideal generated by non-unit monomials of $\cM$.  Suppose $R' = W(k)\llbracket \cM \rrbracket$ and $R = W(k)\llbracket \cM \rrbracket/(p-f)$ where the image of $p-f$ in $W(k)\llbracket\cM\rrbracket/(I_{\cM}+(p^2))$ is $p$.  Then there exists $\Delta \geq 0$ on $\Spec R$ such that $K_R + \Delta$ is $\bQ$-Cartier and $(R, \Delta)$ is BCM-regular.  In other words, $R$ is (dFH)-BCM-regular.
    \end{proposition}

    By \cite{RobinsonToricBCM} there exists $\Delta_{R'}$ on $\Spec R'$ such that $(R', \Delta_{R'})$ is BCM-regular.  We take a slightly different approach for the ramified case which can also be modified to recover Robinson's result.

    \begin{proof}
        If $f = 0$ and $R$ has equal characteristic $p > 0$ then this and more is done in \cite{BlickleMultiplierIdealsAndModulesOnToric}.  
                

        Our strategy is similar to that of \cite{ShimomotoEtAlPerfectoidTowersAndTheirTilts}.  
        Let $D=\Div_{\Spec R'}(p-f) = \Spec R$. Let $D_i'$ be the divisor defined by the extremal rays of the cone of $W(k)\llbracket \cM\rrbracket$. 
         Choose a canonical divisor as $K_{R'} = -D-\sum_{i=1}^t D_i'$. Consider an effective Cartier divisor $H'=\sum b_i D_i'$ such that $\Supp(H') = \Supp(K_{R}+D)$. Fix $1\gg\epsilon>0$ a small rational number and define $\Delta_{R'} = -K_{R'}-D -\epsilon H' = \sum (1-\epsilon b_i) D_i'$. Then, by construction, $K_{R'}+\Delta_{R'} +D = -\sum \epsilon b_i D_i'$ is $\bQ$-Cartier and we also have that $\Delta_{R'}$ and $D$ have no common components.

        Now, by the adjunction formula for adjoint $\bQ$-divisors, see  \cite{KollarKovacsSingularitiesBook} or \cite[Subsection 2.1]{MaSchwedeTuckerWaldronWitaszekAdjoint}, we may write $(K_{R'}+\Delta_{R'}+D)|_{D}\sim_\bQ K_{R} + \Delta_{R}$ where $\Delta_{R} := \mathrm{Diff}_{R}(D+\Delta_{R'})$ is effective since $R$ is normal.        
        Observe that $K_{R}+ \Delta_R$ is $\bQ$-linearly equivalent to the restriction of $-\sum\epsilon b_i D_i$ to $R$, thereby it is $\bQ$-Cartier. 
        We compute the different $\Delta_R$ explicitly.  Notice that for each prime divisor $D_i'$ as above with associated prime ideal $I_{D_i'} \subseteq R'$, we have that the domain $R_i' := R'/I_{D_i'}$ is also normal, since every face of a saturated semigroup is saturated, \cf \cite{FultonToric,MustataToricNotes}, and hence $R_i := R_i'/(p-f)$ is a normal integral domain.  Then $I_{D_i'}$ is principal along codimension 1 points of $\Spec R$, so we obtain prime divisors $D_i := D_i' |_{D}$.  It then follows from the adjunction sequence that $K_{R} = -\sum D_i$ and also that $\Delta_R = \mathrm{Diff}_{R}(D+\Delta_{R'}) = (1 - \epsilon b_i) D_i$. 

    We claim that $(R,\Delta_{R})$ is BCM-regular. By, \cite[Lemma 2.10]{ShimomotoEtAlPerfectoidTowersAndTheirTilts} \cf \cite{HochsterRingsInvariantsofTori}, we know there exists $S' = W(k)\llbracket x_1, \dots, x_t \rrbracket$ such that $R' \to S'$ splits (note $R' \to S'$ is \emph{not} finite in general as $S'$ may have larger dimension).  We briefly need to recall the construction of the map $R' \to S'$.  Let $v_i$ be the valuation associated with each monomial prime divisor $D_i$ on $\Spec R$.  Consider a map $W(k)[\cM]\to W(k)[x_1,\dots, x_t]$ defined by sending $m\in\cM$ to $x_1^{v_1(m)}x_2^{v_2(m)}\dots x_t^{v_t(m)}$ which splits as $W(k)[\cM]$-modules by \cite{HochsterRingsInvariantsofTori} (in that reference, a splitting is called a ``Reynolds operator'') or by  \cite[Lemma 2.10]{ShimomotoEtAlPerfectoidTowersAndTheirTilts}. 
    Now, the proof of \cite[Lemma 2.28]{ShimomotoEtAlPerfectoidTowersAndTheirTilts} and \cite[Remark 2.9]{ShimomotoEtAlPerfectoidTowersAndTheirTilts} shows the splitting extends to the completions. Finally we define $R \to S = S'/(p-f)$ to be the map induced by modding out by $p-f$ (which is also split, \cite[Proof of Theorem 2.29]{ShimomotoEtAlPerfectoidTowersAndTheirTilts}).  We write $E_i' = \Div_{\Spec S'}(x_i)$ and $E_i'|_{\Spec S} = E_i = \Div_{\Spec S}(x_i)$ and observe that these are the pullbacks of $D_i'$ and $D_i$ respectively (in the sense that $I_{D_i'} S'$ reflexifies to $I_{E_i'} = S'(-E_i')$, and likewise $I_{D_i} S$ reflexifies to $I_{E_i}$) by our choice of $R' \to S'$.
    
    Next, consider $B$ a perfectoid big Cohen-Macaulay $R^+$-algebra. Since $R$ and $S$ have the same residue field, by \cite[Theorem A.5]{MaSchwedeTuckerWaldronWitaszekAdjoint}, there exists a perfectoid big Cohen-Macaulay $S^+$-algebra $C$ such that the following diagram commutes:
    \[
        \xymatrix{
        R \ar[r]\ar[d] & R^+ \ar[r]\ar[d] & B\ar[d]\\
        S \ar[r] & S^+ \ar[r] & C.
        }
    \]         
    Let $E = \Div_{\Spec S'}(p-f) = \Spec S$, consider $\Delta_S' =\sum (1-\epsilon b_i) E_i'$, and write $\Delta_{S} = \mathrm{Diff}_{S}(E+\Delta_{S'}) = \sum-\epsilon b_iE_i'|_{\Spec S} - K_{S} = \sum(1 - \epsilon b_i)E_i$ as the corresponding divisors on $\Spec S$.  By construction of the map $R \to S$ and our observations above, we notice that $R^+(\Delta_R) \subseteq S^+(\Delta_S)$.  Combining this with the previous diagram, we obtain a commutative diagram
    \[
        \xymatrix{
            R \ar[r]\ar[d] & R^+(\Delta_{R})  \ar[r]\ar[d]& R^+(\Delta_{R})\otimes_{R^+} B\ar[d]\\
            S \ar[r] & S^+(\Delta_{S}) \ar[r] & S^+(\Delta_{S})\otimes_{{S}^+}C.
        }
    \]      
    Now, tensor the above diagram with $ R(K_{R})$ and 
    take local cohomology to obtain
    \[
    \xymatrix{
    H^d_\fram (R(K_{R})) \ar[r]\ar[d] & H^d_\fram (R^+(K_R+\Delta_{R})) \ar[r]\ar[d]&H^d_\fram(R^+(K_{R}+\Delta_{R}) \otimes_{R^+}B)\ar[d] \\
    H^d_\fram (R(K_R)\otimes_R S) \ar[r] &H^d_\fram (R(K_R)\otimes_R S^+(\Delta_{S}) )\ar[r]&H^d_\fram(R(K_R)\otimes_R S^+(\Delta_{S}) \otimes_{S^+}C).
    }
    \]
    Here, recall that the top local cohomology does not detect the difference between $R^+(K_R+\Delta)$ and $R(K_R)\otimes R^+(\Delta_R)$ (see the proof of \autoref{prop.dH-BCMRegularImpliesPerfdPositive}). Since $S$ is regular, $\Delta_{S}$ has SNC support, and $\lfloor\Delta_{S}\rfloor = 0$, we see that $S\to S^+(\Delta_{S}) \otimes_{S^+}C$ splits by \cite[Theorem 4.1]{MaSchwedeTuckerWaldronWitaszekAdjoint} and so the bottom row of the above diagram is injective. 
    Therefore, the composed map from the top left corner to bottom right corner:
    \[
        H^d_{\fram}(R(K_{R})) \to H^d_\fram(R(K_R)\otimes_R S) \to  H^d_\fram(R(K_R)\otimes_R S^+(\Delta_{S}) \otimes_{S^+}C)
    \]
    is injective and hence the top row 
    \[
        H^d_{\fram}(R(K_{R})) \to H^d_\fram(R^+(K_{R}+\Delta_{R}) \otimes_{R^+}B)
    \]
    injects as well.  This is what we wanted to show.
    \end{proof}

    \subsection{BCM-rational rings and perfectoid rational and relative rational signatures}
    \label{subsec.BCMrational}
In this subsection we introduce mixed characteristic variants of $F$-rational signature of Hochster-Yao \cite{HochsterYaoFratSignature} and relative $F$-rational signature of Smirnov-Tucker \cite{SmirnovTuckerRelativeFSignature}. Our main result is to relate these notions with BCM-rational singularities as defined in Ma-Schwede \cite{MaSchwedeSingularitiesMixedCharBCM}. This is in complete analogy with the characteristic $p>0$ picture and is the rational version of \autoref{prop.dH-BCMRegularImpliesPerfdPositive} and \autoref{prop.SignaturePositiveImpliesPerturbationBCMRegular}. 

\begin{definition}
  Suppose that $(R, \m)$ is a Noetherian complete local domain with residue characteristic $p > 0$ and dimension $d$.  We say that $R$ is \emph{BCM-rational}\footnote{Again, technically speaking, this is called ``perfectoid BCM-rational" in \cite{MaSchwedeSingularitiesMixedCharBCM}.  We are only considering perfectoid big Cohen-Macaulay $R^+$-algebras and so we ignore this distinction.} if it is Cohen-Macaulay and if for every perfectoid big Cohen-Macaulay $R^+$-algebra $B$, the induced map $H_\m^d(R)\to H_\m^d(B)$ is injective.
\end{definition}

Now we introduce perfectoid rational and relative rational signatures.
 
\begin{definition}
    We use notation as in \autoref{not.setup}. Then we
    define the \emph{perfectoid rational signature} to be
    $$s^{\underline{x}}_{\ratperfd}(R)=\inf_{\substack{\underline{y} \\ z \not\in \underline{y}} } e^{\underline{x}}_{\perfd}(\underline{y}) - e^{\underline{x}}_{\perfd}((\underline{y},z))$$
    where the infimum varies over all systems of parameters $\underline{y}$ and all elements $z \not\in (\underline{y})$.
    Similarly, we
    define the \emph{perfectoid relative rational signature} to be
    $$s^{\underline{x}}_{\relperfd}(R)=\inf_{\substack{\underline{y} \\ \underline{y}\subsetneq J}} \frac{e^{\underline{x}}_{\perfd}(\underline{y}) - e^{\underline{x}}_{\perfd}(J)}{\length_R(J/(\underline{y}))}$$
    where the infimum varies over all systems of parameters $\underline{y}$ and all strictly larger ideals $(\underline{y}) \subsetneq J$.
\end{definition}

\begin{proposition}
    \label{prop.notcmratsigzero}
    With notation as in \autoref{not.setup}, if $(R,\m)$ is not Cohen-Macaulay, then for any system of parameters $\underline{y}$ there exists some $u \in R \setminus (\underline{y})$ with  $e^{\underline{x}}_{\perfd}(\underline{y},u) = e^{\underline{x}}_{\perfd}(\underline{y})$. In particular, $s^{\underline{x}}_{\ratperfd}(R) = s^{\underline{x}}_{\relperfd}(R) = 0$.
\end{proposition}

\begin{proof}
    If $(R,\m)$ is not Cohen-Macaulay, then $(\underline{y})^{\lim} \neq (\underline{y})$ by \cite[Corollary 2.4]{CuongHoaLoanOnCertainLengthFunctions} (see also \cite[Theorem 9]{MaQuySmirnovColengthMultiplicity}). Thus, there is some $u \in R \setminus (\underline{y})$ and $t \in \mathbb{Z}_{>0}$ with $u (y_1 \cdots y_d)^{t-1} \in (y_1^t, \ldots, y_d^t)$. If $B$ is a perfectoid balanced big Cohen-Macaulay $R^+$-algebra, it follows that $u \in (\underline{y})B$ and so by \autoref{prop.HKVsBCMClosureVsepfClosure} we have $e^{\underline{x}}_{\perfd}(\underline{y}) = e^{\underline{x}}_{\perfd}(\underline{y},u)$ as desired. 
    It follows immediately that $s^{\underline{x}}_{\ratperfd}(R) = 0$, and taking $(\underline{y}) \subsetneq J := (\underline{y},u)$ in the definition we see $s^{\underline{x}}_{\relperfd}(R) = 0$ as well.
\end{proof}

\begin{lemma}
    \label{lem.reducetosocleideals}
    With notation as in \autoref{not.setup}, suppose $\underline{y}$ is a system of parameters, $J\subseteq R$ is an ideal with $\underline{y} \subsetneq J$, and $a \in \m$. Then there exists an ideal $I$ with $(\underline{y}) \subseteq I \subseteq J$ and $a I \subseteq (\underline{y})$ so that
    \begin{equation*}
        \frac{e^{\underline{x}}_{\perfd}(\underline{y}) - e^{\underline{x}}_{\perfd}(J)}{\length_R(J/(\underline{y}))} \geq \frac{e^{\underline{x}}_{\perfd}(\underline{y}) - e^{\underline{x}}_{\perfd}(I)}{\length_R(I/(\underline{y}))}.
    \end{equation*}
\end{lemma}

\begin{proof}
    This essentially follows from the method of proof in \cite[Proposition 3.5]{SmirnovTuckerRelativeFSignature}, and we spell out the argument for completeness. There is an integer $m$ such that $a^{m}J \subseteq (\underline{y})$. If $m = 1$, we may take $I = J$, and otherwise proceed via induction on $m$.

    There are short exact sequences
    \begin{equation*}
        0 \to \frac{((\underline{y}):_J a)}{(\underline{y})} \to \frac{J}{(\underline{y})} \xrightarrow{\cdot a} \frac{(\underline{y},aJ)}{(\underline{y})} \to 0 \; \text{ and}
    \end{equation*}
    \begin{equation*}
        0 \to \frac{((\underline{y})R^{\Ainfty}_{\perfd}:_{JR^{\Ainfty}_{\perfd}} a)}{(\underline{y})R^{\Ainfty}_{\perfd}} \to \frac{JR^{\Ainfty}_{\perfd}}{(\underline{y})R^{\Ainfty}_{\perfd}} \xrightarrow{\cdot a} \frac{(\underline{y},aJ)R^{\Ainfty}_{\perfd}}{(\underline{y})R^{\Ainfty}_{\perfd}} \to 0.
    \end{equation*}
    and using that $((\underline{y}):_J a)R^{\Ainfty}_{\perfd} \subseteq ((\underline{y})R^{\Ainfty}_{\perfd}:_{JR^{\Ainfty}_{\perfd}} a)$, it follows by \autoref{prop.NormalizedLengthProperties}\autoref{prop.NormalizedLengthProperties.a} that we have 
    \begin{equation*}
        e^{\underline{x}}_{\perfd}(\underline{y}) - e^{\underline{x}}_{\perfd}(J) 
        \geq e^{\underline{x}}_{\perfd}(\underline{y}) - e^{\underline{x}}_{\perfd}((\underline{y}):_J a) + e^{\underline{x}}_{\perfd}(\underline{y}) - e^{\underline{x}}_{\perfd}(\underline{y},aJ).
    \end{equation*}
    From the elementary inequality $\frac{a+c}{b+d} \geq \min(\frac{a}{b},\frac{c}{d})$, it follows that
    \begin{equation*}
        \frac{e^{\underline{x}}_{\perfd}(\underline{y}) - e^{\underline{x}}_{\perfd}(J)}{\length_R(J/(\underline{y}))} 
        \geq \min\left( \frac{e^{\underline{x}}_{\perfd}(\underline{y}) - e^{\underline{x}}_{\perfd}((\underline{y}):_J a)}{\length_R(((\underline{y}):_J a)/(\underline{y}))}, \frac{e^{\underline{x}}_{\perfd}(\underline{y}) - e^{\underline{x}}_{\perfd}(\underline{y},aJ)}{\length_R((\underline{y},aJ)/(\underline{y}))} \right).
    \end{equation*}
    If the first quantity achieves the minimum, we have 
    \begin{equation*}
        \frac{e^{\underline{x}}_{\perfd}(\underline{y}) - e^{\underline{x}}_{\perfd}(J)}{\length_R(J/(\underline{y}))} 
        \geq \frac{e^{\underline{x}}_{\perfd}(\underline{y}) - e^{\underline{x}}_{\perfd}((\underline{y}):_J a)}{\length_R(((\underline{y}):_J a)/(\underline{y}))}
    \end{equation*}
    so we see $I = ((\underline{y}):_J a)$ satisfies the desired properties. Otherwise, if the second quantity gives the minimum, we have 
    \begin{equation*}
        \frac{e^{\underline{x}}_{\perfd}(\underline{y}) - e^{\underline{x}}_{\perfd}(J)}{\length_R(J/(\underline{y}))} \geq \frac{e^{\underline{x}}_{\perfd}(\underline{y}) - e^{\underline{x}}_{\perfd}(\underline{y},aJ)}{\length_R((\underline{y},aJ)/(\underline{y}))}.
    \end{equation*}
    Since $a^{m-1}(\underline{y},aJ) \subseteq (\underline{y})$, the induction hypothesis gives an ideal $(\underline{y})\subsetneq I \subseteq (\underline{y},aJ) \subseteq J$ with $aI \subseteq (\underline{y})$ and 
    \begin{equation*}
        \frac{e^{\underline{x}}_{\perfd}(\underline{y}) - e^{\underline{x}}_{\perfd}(J)}{\length_R(J/(\underline{y}))} \geq \frac{e^{\underline{x}}_{\perfd}(\underline{y}) - e^{\underline{x}}_{\perfd}(\underline{y},aJ)}{\length_R((\underline{y},aJ)/(\underline{y}))} \geq \frac{e^{\underline{x}}_{\perfd}(\underline{y}) - e^{\underline{x}}_{\perfd}(I)}{\length_R(I/(\underline{y}))}
    \end{equation*}
    completing the proof.
\end{proof}

The next result shows that, in computing the perfectoid rational or relative rational signature, it suffices to consider socle elements or ideals, respectively. 

\begin{proposition}
    \label{prop.socleenoughforratsig}
    With notation as in \autoref{not.setup}, we have
    \begin{equation*}
        \begin{array}{crclc}
        \;\;\;\;\;& s^{\underline{x}}_{\ratperfd}(R) & = & {\displaystyle\inf_{\substack{\underline{y} \\ z \in ((\underline{y}):_R \m) \setminus (\underline{y})} } e^{\underline{x}}_{\perfd}(\underline{y}) - e^{\underline{x}}_{\perfd}((\underline{y},z))} & \text{ and}\\ & \\
        \;\;\;\;\;& s^{\underline{x}}_{\relperfd}(R) & =& {\displaystyle\inf_{\substack{\underline{y} \\ \underline{y}\subsetneq J \subseteq ((\underline{y}):_R \m)}} \frac{e^{\underline{x}}_{\perfd}(\underline{y}) - e^{\underline{x}}_{\perfd}(J)}{\length_R(J/(\underline{y}))}}.
        \end{array}
    \end{equation*}
\end{proposition}

\begin{proof}
    Suppose $\underline{y}$ is a system of parameters. If $z \not\in (\underline{y})$, as $R/(\underline{y})$ is an essential extension of its socle, we may choose $u \in \left( (\underline{y},z) \cap ((\underline{y}):_R \m) \right) \setminus (\underline{y})$. Since $(\underline{y},u) \subseteq (\underline{y},z)$, we have that 
    \begin{equation*}
        e^{\underline{x}}_{\perfd}(\underline{y}) - e^{\underline{x}}_{\perfd}((\underline{y},z)) \geq e^{\underline{x}}_{\perfd}(\underline{y}) - e^{\underline{x}}_{\perfd}((\underline{y},u))
    \end{equation*}
    and the first equality follows. Now suppose we have an ideal $(\underline{y}) \subsetneq J \subseteq R$. Repeatedly applying \autoref{lem.reducetosocleideals} as necessary for a set of generators of $\m$, we see that there is an ideal $I$ with $(\underline{y}) \subseteq I \subseteq J$ and $\m I \subseteq (\underline{y})$ so that
    \begin{equation*}
        \frac{e^{\underline{x}}_{\perfd}(\underline{y}) - e^{\underline{x}}_{\perfd}(J)}{\length_R(J/(\underline{y}))} \geq \frac{e^{\underline{x}}_{\perfd}(\underline{y}) - e^{\underline{x}}_{\perfd}(I)}{\length_R(I/(\underline{y}))}.
    \end{equation*}
    giving the second equality.
\end{proof}

\begin{corollary} With notation as in \autoref{not.setup}, we have
    \label{cor.sratinequalities}
    \begin{enumerate}
        \item \label{cor.gensratinequalities}
        $s^{\underline{x}}_{\ratperfd}(R) \geq s^{\underline{x}}_{\relperfd}(R) \geq s^{\underline{x}}_{\perfd}(R)$, with equality throughout if $R$ is Gorenstein.
        \item \label{cor.cmsrelsratinequality}
        If $R$ is Cohen-Macaulay with type $t$, then $ s^{\underline{x}}_{\relperfd}(R) \geq (1/t)\cdot s^{\underline{x}}_{\ratperfd}(R)$.
    \end{enumerate}
\end{corollary}

\begin{proof}
    The first inequality in \autoref{cor.gensratinequalities} follows immediately from the characterization of $s^{\underline{x}}_{\ratperfd}(R)$ in \autoref{prop.socleenoughforratsig} and the  definition of $s^{\underline{x}}_{\relperfd}(R)$. The second inequality in \autoref{cor.gensratinequalities} and the equalities when $R$ is Gorenstein are direct corollaries of \autoref{prop.PerfdSignatureRelativePerfdHK}.
    
    For \autoref{cor.cmsrelsratinequality}, assuming $R$ is Cohen-Macaulay and $\underline{y}$ is a system of parameters, we know that the socle $((\underline{y}):_R \m) / (\underline{y})$ of $R/(\underline{y})$ is a $k$-vector space of dimension $t$. Thus, supposing we have $\underline{y}\subsetneq J \subseteq ((\underline{y}):_R \m)$ and taking some $z \in J \setminus (\underline{y})$, we see 
    \begin{equation*}
        \frac{e^{\underline{x}}_{\perfd}(\underline{y}) - e^{\underline{x}}_{\perfd}(J)}{\length_R(J/(\underline{y}))} \geq \frac{e^{\underline{x}}_{\perfd}(\underline{y}) - e^{\underline{x}}_{\perfd}(\underline{y},z)}{\length_R(J/(\underline{y}))} \geq \frac{s^{\underline{x}}_{\ratperfd}(R)}{\length_R(J/(\underline{y}))} \geq \frac{s^{\underline{x}}_{\ratperfd}(R)}{t}
    \end{equation*}
    and the result now follows from the characterization of $s^{\underline{x}}_{\relperfd}(R)$ in \autoref{prop.socleenoughforratsig}.
\end{proof}

We now characterize regularity via perfectoid relative rational signature.

\begin{proposition}
    With notation as in \autoref{not.setup}, $s^{\underline{x}}_{\relperfd}(R) \leq 1$ with equality if and only if $R$ is regular.
\end{proposition}

\begin{proof}
    If $R$ is not Cohen-Macaulay, the inequality follows from \autoref{prop.notcmratsigzero}. Otherwise suppose $R$ is Cohen-Macaulay and $s^{\underline{x}}_{\relperfd}(R) \geq 1$.  For a system of parameters $\underline{y}$ with $(\underline{y}) \neq \m$, it follows that we must have
    \begin{equation*}
        1 \leq s^{\underline{x}}_{\relperfd}(R) \leq \frac{e^{\underline{x}}_{\perfd}(\underline{y}) - e^{\underline{x}}_{\perfd}(\m)}{\length_R(\m/(\underline{y}))} = \frac{\length_R\left(R/(\underline{y})\right) - e^{\underline{x}}_{\perfd}(\m)}{\length_R\left(R/(\underline{y})\right) - 1}
    \end{equation*}
    where we have used that $e^{\underline{x}}_{\perfd}(\underline{y}) = e(\underline{y})$ by \autoref{cor.perfdHKparameterideal} and that $e(\underline{y}) = \length_R\left(R/(\underline{y})\right)$ since $R$ is Cohen-Macaulay (see \cite[Corollary 4.7.11]{BrunsHerzog}). Thus, it follows that $e^{\underline{x}}_{\perfd}(\m) \leq 1$ by  \autoref{prop.eFHKgreaterthan1}, and thus $e^{\underline{x}}_{\perfd}(\m) = 1$ and $s^{\underline{x}}_{\relperfd}(R) = 1$. This happens if and only if $R$ is regular by \autoref{thm.CharacterizationOfRegularLocalRings}.
\end{proof}

We next relate the perfectoid (relative) rational signature with normalized lengths of images of certain maps to local cohomology. 

\begin{lemma}
    \label{lem.relativehkoverparamvialocalcohom}
    With notation as in \autoref{not.setup}, if $\underline{y}$ is a system of parameters and $J\subseteq R$ is an ideal with $\underline{y} \subsetneq J$, then
    \begin{equation*}
        e^{\underline{x}}_{\perfd}(\underline{y}) - e^{\underline{x}}_{\perfd}(J) = \lambda_\infty\left( \im(R^{\Ainfty}_{\perfd} \otimes_R J/(\underline{y}) \xrightarrow{\alpha} H^d_\m(R^{\Ainfty}_{\perfd})) \right).
    \end{equation*}
    where the map $\alpha$ is given by the composition
    \begin{equation*}
        R^{\Ainfty}_{\perfd} \otimes_R J/(\underline{y}) \xrightarrow{\beta}  R^{\Ainfty}_{\perfd} \otimes_R R/(\underline{y}) \xrightarrow{\gamma} \varinjlim_t  R^{\Ainfty}_{\perfd} \otimes_R R/(\underline{y}^t) = H^d_\m(R^{\Ainfty}_{\perfd}).
    \end{equation*}
\end{lemma}

\begin{proof}
    Note that $\im \beta = JR^{\Ainfty}_{\perfd} / (\underline{y})R^{\Ainfty}_{\perfd}$, and so $e^{\underline{x}}_{\perfd}(\underline{y}) - e^{\underline{x}}_{\perfd}(J) = \lambda_\infty(\im \beta)$ by definition. Since $\gamma$ is $g$-almost injective by \autoref{rmk.AlmostInjectivityLocalCohomology}, we have that $\im \alpha$ and $\im \beta$ are $g$-almost isomorphic. Thus, the assertion now follows as $ \lambda_\infty(\im \beta) = \lambda_\infty(\im \alpha) $ by \autoref{prop.AlmostNormalizedLength}.
\end{proof}

\begin{proposition}
    \label{prop.onesopandlocalcohominterpforratsig}
    With notation as in \autoref{not.setup}, we have
    \begin{enumerate}
        \item 
        \label{prop.onesopforratsig}
        For any fixed system of parameters $\underline{y} \subseteq R$ we have that 
    \begin{equation*}
        \begin{array}{crclc}
            \;\;\;\;\;&s^{\underline{x}}_{\ratperfd}(R) & =& {\displaystyle \inf_{ z \in (\underline{y}:_R \m) \setminus (\underline{y})}  e^{\underline{x}}_{\perfd}(\underline{y}) - e^{\underline{x}}_{\perfd}((\underline{y},z))} & \text{ and} \\
            \\
            \;\;\;\;\;&s^{\underline{x}}_{\relperfd}(R) & = &{\displaystyle\inf_{\underline{y}\subsetneq J \subseteq ((\underline{y}:_R \m))} \frac{e^{\underline{x}}_{\perfd}(\underline{y}) - e^{\underline{x}}_{\perfd}(J)}{\length_R(J/(\underline{y}))}}.
        \end{array}
    \end{equation*}
    \item 
    \label{prop.localcohominterpretforratsig}
    Moreover, if $(R,\m,k)$ is Cohen-Macaulay, then 
    \begin{equation*}
        \begin{array}{rcl}
        s^{\underline{x}}_{\ratperfd}(R) & = &{\displaystyle \inf \left\{ \lambda_\infty \left( \im (R^{\Ainfty}_{\perfd} \otimes_R \psi ) \right) \mid 0 \to k \xrightarrow{\psi} H^d_{\m} (R) \mbox{ is exact} \right\} \; \text{ and}}\\\\
        s^{\underline{x}}_{\relperfd}(R)  & = & \inf \left\{ \frac{1}{i} \lambda_\infty \left( \im (R^{\Ainfty}_{\perfd} \otimes_R \psi ) \right) \mid i \in \Z_{>0}, \, 0 \to k^{\oplus i} \xrightarrow{\psi} H^d_{\m} (R) \mbox{ is exact} \right\}.
        \end{array}
    \end{equation*}
\end{enumerate}
\end{proposition}

\begin{proof}
    If $R$ is not Cohen-Macaulay, \autoref{prop.onesopforratsig} follows immediately from \autoref{prop.notcmratsigzero}. Thus, hereafter we assume $R$ is Cohen-Macaulay and will proceed to show \autoref{prop.onesopforratsig} and \autoref{prop.localcohominterpretforratsig} simultaneously.

    For any system of parameters $\underline{y}$,  we have an injection $R / (\underline{y}) \hookrightarrow H^d_\m(R)$ given by viewing $H^d_\m(R) = \varinjlim_t R/(\underline{y}^t)$ and using that $R$ is Cohen-Macaulay. For any $z \in ((\underline{y}):_R \m) \setminus (\underline{y})$, this gives an injection $\psi \colon k \simeq (\underline{y},z)/(\underline{y}) \hookrightarrow H^d_\m(R)$ so that $\lambda_\infty\left( \im(R^{\Ainfty}_{\perfd} \otimes_R \psi) \right) = e^{\underline{x}}_{\perfd}(\underline{y}) - e^{\underline{x}}_{\perfd}((\underline{y},z))$ by \autoref{lem.relativehkoverparamvialocalcohom}. Moreover, as any injection $k \hookrightarrow H^d_\m(R)$ is given by specifying an element of the socle of $H^d_\m(R)$ which is identified with the socle $((\underline{y}):_R \m)/(\underline{y})$ of $R/(\underline{y})$ under the given inclusion $R / (\underline{y}) \hookrightarrow H^d_\m(R)$, all injections $\psi \colon k \hookrightarrow H^d_\m(R)$ arise from a $z \in ((\underline{y}):_R \m) \setminus (\underline{y})$ in this manner. Thus it follows that
    \begin{equation*}
        \inf_{ z \in (\underline{y}:_R \m) \setminus (\underline{y})}  e^{\underline{x}}_{\perfd}(\underline{y}) - e^{\underline{x}}_{\perfd}((\underline{y},z)) = \inf \left\{ \lambda_\infty \left( \im (R^{\Ainfty}_{\perfd} \otimes_R \psi ) \right) \mid 0 \to k \xrightarrow{\psi} H^d_{\m} (R) \mbox{ is exact} \right\}
    \end{equation*}
    for any fixed system of parameters $\underline{y}$. As the expression on the right hand side above is independent of the choice of $\underline{y}$, the characterizations of $s^{\underline{x}}_{\ratperfd}(R)$ in \autoref{prop.onesopforratsig} and \autoref{prop.localcohominterpretforratsig} now follow from \autoref{prop.socleenoughforratsig}.

    The stated characterizations of $s^{\underline{x}}_{\relperfd}(R)$ follow in a similar fashion. For any system of parameters $\underline{y}$ and any $\underline{y}\subsetneq J \subseteq ((\underline{y}:_R \m))$, we have an induced injection $\psi \colon k^{\oplus i} \simeq J/(\underline{y}) \hookrightarrow H^d_\m(R)$ where $i = \length_R(J/(\underline{y}))$ so that $\lambda_\infty\left( \im(R^{\Ainfty}_{\perfd} \otimes_R \psi) \right) = e^{\underline{x}}_{\perfd}(\underline{y}) - e^{\underline{x}}_{\perfd}(J)$ by \autoref{lem.relativehkoverparamvialocalcohom}. As any injection $k^{\oplus i} \hookrightarrow H^d_\m(R)$ is given by specifying an $i$-dimensional $k$-vector subspace of the socle of $H^d_\m(R)$ which is identified with the socle $((\underline{y}):_R \m)/(\underline{y})$ of $R/(\underline{y})$ under the given inclusion $R / (\underline{y}) \hookrightarrow H^d_\m(R)$, all injections $\psi \colon k^{\oplus i} \hookrightarrow H^d_\m(R)$ arise from a $(\underline{y})\subsetneq J \subseteq ((\underline{y}):_R \m)$ in this manner. Thus it follows that 
        \begin{align*}
             \inf_{\underline{y}\subsetneq J \subseteq ((\underline{y}:_R \m))} & \frac{e^{\underline{x}}_{\perfd}(\underline{y}) - e^{\underline{x}}_{\perfd}(J)}{\length_R(J/(\underline{y}))} \\ &= \inf \left\{ \frac{1}{i} \lambda_\infty \left( \im (R^{\Ainfty}_{\perfd} \otimes_R \psi ) \right) \mid i \in \Z_{>0}, \, 0 \to k^{\oplus i} \xrightarrow{\psi} H^d_{\m} (R) \mbox{ is exact} \right\}
        \end{align*}
    for any fixed system of parameters $\underline{y}$. Again, as the expression on the second line above is independent of the choice of $\underline{y}$, the characterizations of $s^{\underline{x}}_{\relperfd}(R)$ in \autoref{prop.onesopforratsig} and \autoref{prop.localcohominterpretforratsig} now follow from \autoref{prop.socleenoughforratsig}.
\end{proof}



\begin{theorem}
    \label{thm.bcmratandeltsofsmallorder}
    Suppose $v$ is an $\bR$-valuation on $\widehat{R^+}$ that is positive on $\m R^+$. If $R$ is BCM rational, then there exists an $\epsilon > 0$ satisfying the following condition: for all $c \in \widehat{R^+}$ with $v(c) < \epsilon$, the composition
    \begin{equation*}
        H^d_\m(R) \to H^d_\m(R^+) \xrightarrow{\cdot c} H^d_\m(R^+)
    \end{equation*}
    is injective, or equivalently, 
    multiplication by $c$ is injective on the image of $H^d_\m(R)$ in $H^d_\m(R^+)$. 
\end{theorem}

\begin{proof}
    Note first that $H^d_\m(R) \to H^d_\m(R^+)$ injects as $R$ is BCM-rational.  We proceed by way of contradiction. If the statement fails, then there exists a sequence $\{c_j\}_{j\geq 0}$ of elements of $\widehat{R^+}$ such that $v(c_j)\to 0$ as $j\to\infty$ where the kernel of multiplication by $c_j$ on $H^d_\m(R^+)$ has non-trivial intersection with the image of $H^d_\m(R)$. As $H^d_\m(R)$ is an essential extension of its socle $V := \Ann_{H^d_\m(R)}\m$, we have that
    \begin{equation*}
        K_j := V \cap \ker\left( H^d_\m(R) \to H^d_\m(R^+) \xrightarrow{\cdot c_j} H^d_\m(R^+) \right)
    \end{equation*}
    is nonzero for all $j$. Each $K_j$ is a subspace of the finite dimensional $k$-vector space $V$, where $\dim_k(V) = t$ is the Cohen-Macaulay type of $R$. Let $V_j = \sum_{i \geq j} k\cdot K_i$, the $k$-span of the subspaces $K_j, K_{j+1}, K_{j+2}, \ldots$ inside of $V$. As $\dim_k(V_j) \leq t$, each $V_j$ can in fact be realized as the $k$-span of no more than $t$ of these subspaces. Let $i_j(1), \dots, i_j(t) \in \mathbb{Z}_{\geq j}$ be such that $V_j = k \cdot K_{i_j(1)} + \cdots + k \cdot K_{i_j(t)}$, and set $d_j = c_{i_j(1)} c_{i_j(2)}\cdots c_{i_j(t)}$. The new sequence $\{d_j\}_{j\geq 0}$ of elements of $\widehat{R^+}$ still satisfies $v(d_j)\to 0$ as $j\to\infty$ and the kernel of multiplication by $d_j$ on $H^d_\m(R^+)$ contains $V_j$. The descending chain of subspaces $V_0 \supseteq V_1 \supseteq V_2 \supseteq \cdots$ necessarily stabilizes, so for some $N \gg 0$ we have $V_N \subseteq V_j$ for all $j$. Choosing $0 \neq \eta \in V_N$, we have that the image of $\eta$ in $H^d_\m(R^+)$ is annihilated by $d_j$ for all $j$. Using \autoref{cor.GabberTrick}, we see that there exists a perfectoid balanced big Cohen-Macaulay $R^+$-algebra $B$ so that the image of $\eta$ in $H^d_\m(B)$ is zero. In particular, $H^d_\m(R) \to H^d_\m(B)$ is not injective, contradicting that $R$ is BCM-rational.
\end{proof}

Finally, we show that perfectoid rational signature and perfectoid relative rational signature characterize BCM-rational singularities. 

\begin{theorem}
    \label{thm.sratpositiveandbcmregular}
    With notation as in \autoref{not.setup}, the following are equivalent:
    \begin{enumerate}
        \item $R$ is BCM-rational; \label{thm.part.bcmrat}
        \item $s^{\underline{x}}_{\ratperfd}(R) > 0$; \label{thm.part.sratpositive}
        \item $s^{\underline{x}}_{\relperfd}(R) > 0$. \label{thm.part.srelpositive}
    \end{enumerate}
\end{theorem}

\begin{proof}
    As BCM-rational rings are Cohen-Macaulay and in light of \autoref{prop.notcmratsigzero}, without loss of generality we may assume $R$ is Cohen-Macaulay. The equivalence of \autoref{thm.part.sratpositive} and \autoref{thm.part.srelpositive} is then immediate from \autoref{cor.sratinequalities}\autoref{cor.cmsrelsratinequality}.

    To see that \autoref{thm.part.sratpositive} implies \autoref{thm.part.bcmrat}, suppose now that $s^{\underline{x}}_{\ratperfd}(R) > 0$. Let $B$ be any perfectoid balanced big Cohen-Macaulay $R^+$-algebra. Let $\underline{y}$ be a system of parameters. We first show that the induced map $R / (\underline{y}) \to B / (\underline{y}) B$ is injective. It suffices to show injectivity on the socle, so consider $z \in ((\underline{y}):_R \m) \setminus (\underline{y})$. We have $e^{\underline{x}}_{\perfd}(\underline{y}) - e^{\underline{x}}_{\perfd}((\underline{y},z)) \geq   s^{\underline{x}}_{\ratperfd}(R) > 0$, so by \autoref{prop.HKVsBCMClosureVsepfClosure} it follows that $z \not\in (\underline{y})B$ and so $R / (\underline{y}) \to B / (\underline{y}) B$ is injective. Since $B$ is big Cohen-Macaulay, the map $B / (\underline{y})B \to H^d_\m(B) = \varinjlim_t B/(\underline{y}^t)B$ is injective as well. This gives that $R/(\underline{y}) \to H^d_\m(B)$ is injective. Since $R/(\underline{y}) \hookrightarrow H^d_\m(R)$ is an essential extension as $R$ is Cohen-Macaulay, we see that $H^d_\m(R) \to H^d_\m(B)$ is also necessarily injective. Thus, it follows $R$ is BCM-rational.

    Finally, to see that \autoref{thm.part.bcmrat} implies \autoref{thm.part.sratpositive}, assume $R$ is BCM-rational. Let $v$ be an $\mathbb{R}$-valuation on $\widehat{R^+}$ that restricts to the $\mathbb{Z}[1/p]$-valued extension of the $\m_A$-adic valuation on $\Ainfty$ (see \autoref{def.standardvaluationonAinftyzero} and \autoref{rmk.valuationextendstoRplus}). For any $c \in R^{\Ainfty}_{\perfd}$, we write $v(c)$ to mean the order of the image of $c$ in $\widehat{R^+}$ with respect to $v$.    
    Let $0 \to k \xrightharpoondown[1 \mapsto \eta]{\psi} H^d_\m(R)$ be an exact sequence, so that $\im(R^{\Ainfty}_{\perfd} \otimes_R \psi)$ is generated by $1 \otimes \eta$ in $R^{\Ainfty}_{\perfd} \otimes_R H^d_\m(R) = H^d_\m(R^{\Ainfty}_{\perfd})$. Let $\epsilon > 0$ be as in \autoref{thm.bcmratandeltsofsmallorder}; since we have factorizations $R \to R^{\Ainfty}_{\perfd} \to \widehat{R^+}$ and $H^d_\m(R) \to H^d_\m(R^{\Ainfty}_{\perfd}) \to H^d_\m(R^+)$, it follows that multiplication by any element $c \in R^{\Ainfty}_{\perfd}$ on the image of $H^d_\m(R)$ in $H^d_\m(R^{\Ainfty}_{\perfd})$ is injective provided $v(c) < \epsilon$. In particular, we have
    \begin{equation*}
        \Ann_{R^{\Ainfty}_{\perfd}} \left( 1 \otimes \eta \right) \subseteq \{ c \in R^{\Ainfty}_{\perfd} \mid v(c) \geq \epsilon \}.
    \end{equation*}
    Thus, if $\frac{1}{p^{e_0}} < \epsilon$, it follows that
    \begin{align*}
        \lambda_\infty\left( \im(R^{\Ainfty}_{\perfd} \otimes_R \psi) \right) & =  \lambda_\infty \left( R^{\Ainfty}_{\perfd} / \Ann_{R^{\Ainfty}_{\perfd}} \left( 1 \otimes \eta \right) \right) \\
        & \geq \lambda_\infty\left( R^{\Ainfty}_{\perfd} /  \{ c \in R^{\Ainfty}_{\perfd} \mid v(c) \geq \epsilon \}  \right) \\ 
        & \geq \lambda_\infty\left( \Ainfty / \left\{ c \in \Ainfty \mid v(c) \geq \epsilon \right\} \right) \\
        & \geq \lambda_\infty\left( \Ainfty / \left\{ c \in \Ainfty \mid v(c) \geq 1/p^{e_0} \right\} \right) = 1/p^{e_0}
    \end{align*} 
    where the final equality is an easy computation. Taking the infimum over all such $\psi$ and using \autoref{prop.onesopandlocalcohominterpforratsig}\autoref{prop.localcohominterpretforratsig}, this gives $s^{\underline{x}}_{\ratperfd}(R) \geq 1/p^{e_0} > 0$ as desired.
\end{proof}

\section{Transformation rules under finite maps}
\label{sec.TransformationRule}

In this section, we establish the transformation rule for the perfectoid signature under finite quasi-étale split maps, as stated in \autoref{thm.TransformationRule}. The key intuition comes from almost purity (see \autoref{thm.BhattScholzeAlmostPurity}), which suggests that after certain perfectoidizations, finite quasi-étale split maps become almost finite projective with a well-defined generic rank. However, dealing with almost mathematics in mixed characteristic introduces subtle challenges. To overcome these, we will work modulo $(p^{1/p^\infty})$ and handle the characteristic $p>0$ case, and then use almost finite projectivity to carefully lift things back to mixed characteristic. 
As an application, we compute the perfectoid signature and perfectoid Hilbert-Kunz signature for quotient singularities in dimension 2, as shown in \autoref{cor.DuValSingularities}.

We continue to use \autoref{not.setup}. So $R$ is a complete Noetherian local domain of mixed characteristic $(0,p)$ with perfect residue field $k=R/\m$ and $A=W(k)[[x_2,\dots,x_d]] \subseteq R$ is module-finite. Recall that $f: Y\to X\coloneqq\Spec(R)$ is quasi-\'etale if it is quasi-finite and there is some codimension $\ge 2$ subset $Z\subset Y$ such that $Y-Z\to X$ is finite \'etale. In this section, we will prove that perfectoid signature satisfies the transformation rule with respect to finite quasi-\'etale maps in the following sense:
\begin{theorem}
\label{thm.TransformationRule}
With notation as in \autoref{not.setup}, further assume $R$ is normal and that $f:R\to S$ is a finite quasi-\'etale split extension of complete normal local domains.  Suppose also that the generic rank of the extension is $r = [K(S) : K(R)]$ and that the residue field extension degree is $t = [S/\m_S : R/\m_R]$.  Then 
    \[
        t \cdot s^{\underline{x}}_{\perfd}(S)=r\cdot s^{\underline{x}}_{\perfd}(R).
    \]
\end{theorem}
Compare with \cite{HunekeLeuschkeTwoTheoremsAboutMaximal,VonKorffToricVarieties,CarvajalRojasSchwedeTuckerFundamentalGroup,CarvajalRojasFiniteTorsors} in characteristic $p > 0$, with \cite{XuZhuangUniquenessOfTheMinimizerOfTheNormalizedVolume} in characteristic zero, and also with \cite{BrennerJeffriesNunezBetancourtQuantifyingSingularities,SmirnovJeffriesTransformationRuleForNaturalMult} in equal characteristic more generally.

\begin{remark}
    \label{rem.TransformationHandlesResidueFieldExtension}
    Unless $k = R/\m_R \to S/\m_S = k'$ is an isomorphism, we do not compute $s^{\underline{x}}_{\perfd}(S)$ as the normalized length of $S_{\perfd}^{\Ainfty}/I_{\infty}^S$ respect to $\Ainfty$.  Instead, we first form $A' = W(k')\llbracket x_2, \dots, x_d\rrbracket \supseteq A$, a subring of $S$, and compute the normalized length of $S^{\Ainfty'}_{\perfd}/I_{\infty}^S$ respect to $A'_{\infty}$.  This difference can be handled as follows.

    First of all, we claim that $S^{\Ainfty}_{\perfd} \cong S^{\Ainfty'}_{\perfd}$. To see this, note that since perfectoidization does not depend on the choice of base perfectoid ring by \cite[Proposition 8.5]{BhattScholzepPrismaticCohomology}, we have  
    \begin{equation}
    \label{eq.RperfddoesNotDependonBase}
        S^{\Ainfty}_{\perfd} = (S {\otimes}_A \Ainfty)_{\perfd} = ( (S {\otimes}_{A'} A') {\otimes}_A \Ainfty)_{\perfd} = (S {\otimes}_{A'} \Ainfty')_{\perfd} = S^{\Ainfty'}_{\perfd},
    \end{equation}
    where the third equality follows from the fact that $A \to A'$ is finite flat and so $A'\otimes_AA_\infty= A'_{\infty}$ (note also that, all terms inside $(-)_\perfd$ above are classically $p$-complete since they are derived $p$-complete and $p$-torsion free). In view of this, notice also that the construction of the ideal $I_{\infty}$ is independent of our choice of $A$ or $A'$.
    
    Finally, by utilizing \autoref{prop.NormalizedLengthProperties} \autoref{prop.NormalizedLengthProperties.b} and \autoref{prop.NormalizedLengthProperties.c}, we observe that  
    \begin{equation}
        \label{eq.remResidueFieldExtensions}
        \lambda_{\infty}^{\Ainfty}(S^{\Ainfty}_{\perfd} / I_{\infty}) = t \cdot \lambda_{\infty}^{\Ainfty'}(S^{\Ainfty'}_{\perfd} / I_{\infty}) = t \cdot s_{\perfd}^{\underline{x}}(S).
    \end{equation}
    The point is that $M = S^{\Ainfty}_{\perfd} / I_{\infty}$ can be written first as a union of finitely generated $\m_{A'}$-torsion $\Ainfty'$-modules, and hence finitely generated $\Ainfty$-modules, and so it suffices to prove that $\lambda_{\infty}^{\Ainfty}(N) = t \cdot \lambda_{\infty}^{\Ainfty'}(N)$ for $N$ finitely generated.  But now $N$ can be written as a colimit (with surjective transition maps) of finitely presented  $\Ainfty'$-modules which are also finitely presented as $\Ainfty$-modules, all of which are $\m_A$-torsion and so have finite length.  Finally, the statement on scaling of length is clear for finite length modules and so \autoref{eq.remResidueFieldExtensions} holds.
\end{remark}

A very simple special case of this transformation rule is worth highlighting.  If $R \subseteq S$ is a finite \'etale extension of complete local domains, then $S$ is a free $R$-module whose generic rank is equal to its rank which is equal to the residue field extension degree. So \autoref{thm.TransformationRule} says $R$ and $S$ have the same perfectoid signature. In this case we also have a corresponding result on perfectoid Hilbert-Kunz multiplicity. We give a direct proof below. 

\begin{corollary}
    \label{cor.FiniteEtaleHaveTheSameSignature}
    With notation as in \autoref{not.setup}, suppose additionally that $R \to S$ is a finite \'etale extension of complete local domains.  Then for any $\m$-primary ideal $I$ of $R$ we have $e^{\underline{x}}_{\perfd}(I, R)= e^{\underline{x}}_{\perfd}(IS, S)$, and we also have $s_{\perfd}^{\underline{x}}(R) = s_{\perfd}^{\underline{x}}(S)$. 
    
    In particular, if $R'=R\otimes_{W(k)}W(k')$ where $k\to k'$ is a finite extension of perfect fields such that $R'$ is a complete local domain.\footnote{We made this domain assumption on $R'$ because we only defined perfectoid Hilbert-Kunz multiplicity for complete local domains with perfect residue fields. But in fact, as long as $R$ is a complete reduced ring with perfect residue field such that $p$ is part of a system of parameters, then one can define $R^{\Ainfty}_\perfd$, $e^{\underline{x}}_{\perfd}$ and $s^{\underline{x}}_{\perfd}$ essentially in the same way as in \autoref{not.setup} and \autoref{def.PerfdeHKandPerfdSignature}, and most of the results in \autoref{section:mixedchardef} remain true. We plan to explore the more general definition (as well as removing the perfect residue field assumption) in a future work.} Then for every $\m$-primary ideal $I$ of $R$ we have $e^{\underline{x}}_{\perfd}(I, R)= e^{\underline{x}}_{\perfd}(IR', R').$
\end{corollary}
\begin{proof}
Let $A':=W(k')[[x_2,\dots,x_d]]$. By \autoref{eq.RperfddoesNotDependonBase} we know that $S^{\Ainfty'}_{\perfd}=S^{\Ainfty}_{\perfd}$. Moreover, since $R\to S$ is finite \'etale, we know that $S^{\Ainfty}_{\perfd}\cong R^{\Ainfty}_{\perfd} \otimes_RS$ by the almost purity theorem (see \autoref{thm.BhattScholzeAlmostPurity}). In particular, $S^{\Ainfty}_{\perfd}$ is finite free of rank $\rank_RS$ over $R^{\Ainfty}_{\perfd}$. It follows that 
$$\rank_R(S) \cdot \lambda_\infty(S^{A'_\infty}_{\perfd}/IS^{A'_\infty}_{\perfd}) = \lambda_\infty(S^{\Ainfty}_{\perfd}/IS^{\Ainfty}_{\perfd})= \rank_R(S) \cdot \lambda_\infty(R^{\Ainfty}_{\perfd}/IR^{\Ainfty}_{\perfd}).$$
where the first equality follows from the same reasoning as in \autoref{eq.remResidueFieldExtensions}. But this is precisely saying that $e^{\underline{x}}_{\perfd}(I, R)= e^{\underline{x}}_{\perfd}(IS, S)$. 

The conclusion that $s_{\perfd}^{\underline{x}}(R) = s_{\perfd}^{\underline{x}}(S)$ follows by noting that $I_\infty^RS=I_\infty^S$, this follows from the same argument as in \autoref{lem.I_inftyCompare} below by using that $R^{\Ainfty}_{\perfd}\to S^{\Ainfty}_{\perfd}$ is (honest) finite \'etale and we omit the details since as we observed above, the perfectoid signature statement is a special case of our more general result \autoref{thm.TransformationRule}.

Finally, the last conclusion follows because $R\to R'$ is finite \'etale by construction (note that $k$ and $k'$ are perfect fields).
\end{proof}

As a consequence, we obtain the following result.

\begin{corollary}
\label{cor.DegreeTwoHypersurface}
   With notation as in \autoref{not.setup}, suppose $R$ is $S_2$ and $e(R)=2$.  Then 
    \[
        e^{\underline{x}}_{\perfd}(R) + s^{\underline{x}}_{\perfd}(R) = 2.
    \]    
\end{corollary}
\begin{proof}
We first reduce to the case that $\m$ has a minimal reduction generated by $d$ elements (i.e., a system of parameters). We only need this step when the residue field of $R$ is finite since otherwise this is automatic (see \cite[Proposition 8.3.7 and Corollary 8.3.9]{HunekeSwansonIntegralClosure}). If the residue field $k$ of $R$ is finite, then by \autoref{lem.finiteresidue}, there exists a finite field extension $k\to k'$ so that $R':= R\otimes_{W(k)}W(k')$ is a complete local domain with maximal ideal $\m'=\m R'$ having a minimal reduction generated by a system of parameters. Since $R'$ is a finite \'{e}tale extension of $R$, it is clear that $R'$ still satisfies $S_2$ and that $e(R')=e(R)=2$ (in fact, $e(IR', R')=e(I, R)$ for any $\m$-primary ideal, see the last paragraph of the proof of \autoref{cor.perfdHKparameterideal}). Moreover, by \autoref{cor.FiniteEtaleHaveTheSameSignature} both the perfectoid Hilbert-Kunz multiplicity and the perfectoid signature of $R$ and $R'$ agree. Therefore we may replace $R$ by $R'$ to assume $\m$ has a minimal reduction generated by $d$ elements.

By the monomial conjecture (which is a theorem thanks to \cite[Page 72 (2)]{AndreDirectsummandconjecture}) and \cite{ikedaunpublished}, our hypotheses imply that $R$ is Cohen-Macaulay. Let $\underline{z}=z_1,\dots, z_d$ be a minimal reduction of $\m$. It follows that $$2=e(R)=e(\underline{z}, R) =\length(R/(\underline{z)}) =1+\length(\m/(\underline{z}))\geq 1+\edim(R)-\dim(R) \geq 2$$
where the third equality follows from \cite[Corollary 4.7.11]{BrunsHerzog} as $\underline{z}$ is a system of parameters and $R$ is Cohen-Macaulay. 
Thus we have $\edim(R)=\dim(R)+1$ and that $\m=(\underline{z}, u)$ for some $u$, which must be a socle representative of $R/(\underline{z})$. Note that this implies $R$ is a hypersurface: write $R=S/I$ such that $S$ is a regular local ring with $\edim(R)=\edim(S)$, then $I$ is a height one ideal of $S$ that is unmixed, so $I$ must be principal. Therefore we have  
$$s^{\underline{x}}_{\perfd}(R)= e^{\underline{x}}_{\perfd}((\underline{z}), R) - e^{\underline{x}}_{\perfd}(\m, R)=e(\underline{z}, R)-e^{\underline{x}}_{\perfd}(R) =2-e^{\underline{x}}_{\perfd}(R)$$
where the first equality follows from \autoref{prop.PerfdSignatureRelativePerfdHK} and the second equality follows from \autoref{prop.Normalizedlengthequalsmultiplicity}.
\end{proof}


\subsection{Almost projective and almost finite free modules}
In this subsection we collect some definitions and basic results from almost mathematics. We setup almost mathematics with respect to $(V,\m_V)$ in the sense of \cite[2.1.1 and 2.5.14]{GabberRameroAlmostringtheory}. That is, $V$ is a ring (commutative and containing $1$), $\m_V$ is an ideal so that $\m_V=\m_V^2$ and $\m_V\otimes_V\m_V$ is a flat $V$-module. A particular example is that $V$ is a perfect ring of characteristic $p>0$ and $\m_V=(f^{1/p^\infty})$ for some $f\in V$. Another example (that actually contains the previous example as a special case) is that $(V,\m_V) = (S/p^n, (g)_\perfd S/p^n)$ where $S$ is a perfectoid ring and $g\in S$, see \cite[paragraph between Remark 10.7 and Proposition 10.8]{BhattScholzepPrismaticCohomology}.

Assume that ${V'}$ is a $V$-algebra. Let $P$ be a ${V'}$-module. Recall that $P$ is \emph{almost projective}\footnote{Rigorously speaking, one should say $P^a$ is almost projective as in \cite{GabberRameroAlmostringtheory}, we ignore this distinction though since we work with ``almostifications'' of the honest modules we are interested.} if the functor $\text{alHom}_{V'}(P^a,-)$ is exact, or equivalently $\Ext_{V'}^{i}(P,N)$ is almost zero for any $i>1$ and any ${V'}$-module $N$. We also recall the definition of finite rank in the sense of \cite[Definition 4.3.9]{GabberRameroAlmostringtheory} here. We say an almost finitely generated projective module $P$ is of \emph{finite rank} if there exisits an integer $i\ge 0$ such that $\Lambda_{V'}^iP$ is almost zero; we say $P$ is of \emph{constant rank $r\in \mathbb{N}$}, if $\Lambda_{V'}^{r+1}P$ is almost zero and $\Lambda_{V'}^rP$ is an almost invertible $V'$-module, i.e., $(\Lambda_{V'}^rP) \otimes_{V'}\Hom(\Lambda_{V'}^rP,V')$ is almost isomorphic to $V'$.

Maps from almost projective modules lift with respect to surjective maps up to multiplication by elements in $\m_V$ in the following sense: 
\begin{lemma}
\label{lem.Almostprojectivelifts}
    Suppose that $P,M,N$ are ${V'}$-modules where $P$ is almost projective. Let $\phi:M\to N$ be a surjective ${V'}$-linear map. Then for any $V'$-linear $\psi:P\to N$ and $\epsilon\in \m_V$, there exists $V'$-linear map $\psi_\epsilon: P\to M$ such that $\phi\circ\psi_\epsilon = \epsilon\cdot\psi$.
\end{lemma}
\begin{proof}
    Consider the long exact sequence obtain by applying $\Hom_{V'}(P,-)$ to the short exact sequence $0\to \ker(\phi)\to M\to N\to 0$:
    $$\cdots \xrightarrow{}\Hom_{V'}(P,M)\to \Hom_{V'}(P,N)\to \Ext^1_{V'}(P,\ker(\phi))\to \cdots.$$
    Since the image of $\epsilon \cdot \psi$ in $\Ext^1_{V'}(P,\ker(\phi))$ is $0$ (as $\epsilon$ annihilates $\Ext^1_{V'}(P,\ker(\phi))$), there is some $\psi_\epsilon\in \Hom(P,M)$ that maps to $\epsilon \cdot \psi$. It is easy to see that $\psi_\epsilon$ satisfies the conclusion of the lemma.
\end{proof}

We also record the fact that in the almost finitely generated case, almost projective modules behaves as almost direct summands of free modules.

\begin{lemma}[{\cite[Lemma 2.4.15]{GabberRameroAlmostringtheory}}]
    \label{lem.AlmostFGAlmostProjectiveVsAlmostDirectSummandOfFree}
    Let $P$ be an almost finitely generated ${V'}$-module, then $P$ is almost projective if and only if for any $\epsilon \in \m_V$, there exists some $n(\epsilon)\in \mathbb{N}$ and ${V'}$-linear maps $\phi_\epsilon: P\to  {V'}^{\oplus n(\epsilon)}$ and $\psi_\epsilon:{V'}^{\oplus n(\epsilon)}\to  P$ such that  $\psi_\epsilon\circ\phi_\epsilon=\epsilon\cdot \mathbf{1}_P$. 
\end{lemma}

Following \autoref{lem.AlmostFGAlmostProjectiveVsAlmostDirectSummandOfFree}, we now introduce almost finite free modules. 

\begin{definition}
    We say that an almost finitely generated module $M$ is \emph{almost finite free} if it is almost projective and, with notation as in \autoref{lem.AlmostFGAlmostProjectiveVsAlmostDirectSummandOfFree}, we have that  $\phi_\epsilon\circ\psi_\epsilon=\epsilon\cdot \mathbf{1}_{{V'}^{\oplus n(\epsilon)}}$. Moreover, we say that $M$ is \emph{uniformly almost finite free of rank $r$} if we can take $n(\epsilon)=r$ for any $\epsilon\in \m_V$.
\end{definition}

Our observation (which we believe should be known to experts) here is that, under mild assumptions, the almost purity theorem in characteristic $p>0$ guarantees uniformly almost finite freeness. Recall that for any perfect ring $S$ and any $f\in S$, one can set up almost mathematics by taking $(V,\m_V)=(S,(f^{1/p^\infty}))$. 

\begin{lemma}
\label{lem.Almostfreeincharp}
    Let $\varphi$: $S\to T$ be an integral map of perfect rings of characteristic $p>0$. If there exists some nonzero element $f\in S$ such that $T[1/f]$ is finite \'etale and of constant rank $r$, then $T$ is uniformly $f$-almost projective of constant rank $r$ over $S$.
    
    Moreover, if $T[1/f]$ is finite free over $S[1/f]$ of rank $r$, then $T$ is uniformly $f$-almost free of rank $r$ over $S$.
\end{lemma}
\begin{proof}
We essentially need to run the proof of the almost purity theorem in characteristic $p>0$ (see \cite[Proposition 5.23]{ScholzePerfectoidspaces}).  First note that the $f$-power torsion modules $0 :_S (f^{\infty}) \subset S$ and $0 :_T (f^{\infty})  \subset T$ are $f$-almost zero. Furthermore, $S$ is $f$-almost isomorphic to its integral closure in $S[1/f]$, and similarly for $T$. Hence we may assume that $S$ and $T$ are $f$-torsion free and integrally closed in $S[1/f]$ and $T[1/f]$ respectively. 

Let $e=\sum_{i=1}^k a_i\otimes b_i \in (T\otimes_S T)[1/f]$ be the diagonal idempotent. It follows that we can explicitly express $T[1/f]$ as a finite projective $S[1/f]$-module via maps:
    $$T[1/f]\xrightarrow{\phi}(S[1/f])^{\oplus k}\xrightarrow{\psi}T[1/f]$$
    where $\phi(g)=(\Tr(ga_1), \dots, \Tr(ga_k))$ and $\psi((g_1, \dots, g_k))=\sum g_ib_i$. One can verify that the composition is the identity map.

We next note that there exists some $N,M\gg 0$ such that $f^Na_i\in T$ and $f^Mb_i\in T$ for all $i$.
    Thus $f^{M+N}e=\sum_{i=1}^k (f^Ma_i)\otimes (f^N b_i) \in T\otimes_S T$. As Frobenius is bijective, we have $f^{(M+N)/p^n}e=\sum_{i=1}^k (f^{M/p^n}a_i^{1/p^n})\otimes (f^{N/p^n} b_i^{1/p^n}) \in T\otimes_ST$ for all $n$. Hence for any $\epsilon \in (f^{1/p^\infty})$, we can write $\epsilon e=\sum_{i=1}^k c_i\otimes d_i$ for $c_i,d_i\in T$. Since $S$ is integrally closed in $S[1/f]$, we know that $\Tr(T)\subseteq S$. This can be found in \cite[Proof of Claim 3.5.33]{GabberRameroAlmostringtheory} and we briefly sketch it here: It is enough to show $\Tr(T)\subseteq S$ upon replacing $S$ by a faithfully flat extension. Since $T[1/f]$ is finite \'etale over $S[1/f]$ of constant rank $r$, we may replace $S$ by a faithfully flat extension (in fact a finite product of localizations of $S$, so $S$ is still integrally closed in $S[1/f]$) so that $T[1/f]$ is finite free of rank $r$ over $S[1/f]$. Now since $T$ is integral over $S$, the trace of each element of $T$, viewed as an element in $S[1/f]$, is integral over $S$ (see \cite[Chapter 5, Proposition 17]{Bourbaki1998}) and hence belongs to $S$ by integral closedness of $S$ in $S[1/f]$. Therefore, it follows that for any $\epsilon \in (f^{1/p^\infty})$, we can define
    $$T\xrightarrow{\phi}S^{\oplus k}\xrightarrow{\psi}T$$
    by $\phi(g)=(\Tr(gc_1), \dots, \Tr(gc_k))$ and $\psi((g_1, \dots, g_k))=\sum g_id_i$. The composition is multiplication by $\epsilon$, as this is true after inverting $f$ and $S$ and $T$ are $f$-torsion free.

    We next verify that $T$ is of constant rank $r$ over $S$ in the almost category. By the above argument we know $\Lambda_S^{k+1}T$ is almost zero, hence by \cite[Proposition 4.3.27]{GabberRameroAlmostringtheory} we have $S\simeq \prod_{i=0}^{k}S_i$ where $T\otimes_S S_i$ is an $S_i$-module of constant rank $i$ in the almost category. Let $T_i=T\otimes_S S_i$. Then we have $\Lambda_{S_i}^{i+1}T_i$ is almost zero and $\Lambda_{S_i}^{i}T_i$ is an almost invertible $S_i$-module. After inverting $f$, the above all become honest zeros and honest isomorphisms. However, since $T[1/f]$ is of constant rank $r$ over $S[1/f]$, we have $\Lambda_{S[1/f]}^{r}T[1/f]$ is an invertible $S[1/f]$-module and $\Lambda_{S[1/f]}^{r+1}T[1/f]\simeq 0$. Hence after base change to $S_i$ we see that $S_i$ is almost zero when $i\neq r$, and thus $T\overset{a}{\simeq} T_r$ has constant rank $r$ in the almost category.
    
    Moreover, assume $T[1/f]$ is finite free over $S[1/f]$ of rank $r$. Let $\{e_i\}_{i=1}^r$ be a free basis of $T[1/f]$ over $S[1/f]$. Since $T[1/f]$ is finite \'etale and finite free over $S[1/f]$, the trace map $\Tr$ induces a perfect pairing: 
    $${\bf t}:T[1/f]\otimes T[1/f]\to S[1/f]$$ 
    $${\bf t}(t_1\otimes t_2)=\Tr(t_1t_2).$$ 
    Hence if one identifies $T[1/f]$ with $T[1/f]^*\coloneqq \Hom(T[1/f],S[1/f])$ by the perfect pairing ${\bf t}$, then $e:=\sum_{i=1}^re_i\otimes e_i^*$ is the canonical idempotent for $T[1/f]\otimes T[1/f]$. Then one can verify that, for each $\epsilon\in (f^{1/p^\infty})$, not only $\psi\circ \phi$ but also $\phi\circ \psi$ is an isomorphism after inverting $f$. Now the same argument as in the third paragraph in the proof shows that both $\psi\circ \phi$ and $\phi\circ \psi$ are multiplication by $\epsilon$ on the integral level.
\end{proof}

Using the tilting correspondence, we next prove a weaker variant of \autoref{lem.Almostfreeincharp} in mixed characteristic. In our application we will only use the next corollary when $S$ is a perfectoid valuation ring (in which case the result is certainly well-known), but we state the result in a general setup that we can prove.

We recall that if $S$ is a perfectoid ring, then there exists $\varpi\in S$ such that $\varpi^p=up$ where $u$ is a unit in $S$ by definition (see \autoref{subsec:PefectoidAlgebras}). in fact, we can further assume that $\varpi$ admits a compatible system of $p$-power roots $\{\varpi^{1/p^e}\}$ in $S$ (see \cite[Remark 3.8 and Lemma 3.9]{BhattMorrowScholzeIHES}). Then $(V,\m_V)=(S, (\varpi^{1/p^\infty}))(=(S, (p)_\perfd))$ forms an almost mathematics setup. 

\begin{corollary}
\label{cor.AlmostfreeMixedChar}
Let $S\to T$ be a map of $p$-torsion free perfectoid rings such that $T[1/p]$ is finite \'etale over $S[1/p]$ of constant rank $r$. Then there exists some nonzero $s\in S^\flat$ such that $T/p$ is uniformly $ps^\sharp$-almost finite free of rank $r$ over $S/p$.  

Moreover, if $S$ is a perfectoid valuation ring, then $T/p$ is uniformly $p$-almost finite free of rank $r$ over $S/p$.  
\end{corollary}
\begin{proof}
By the almost purity theorem \cite[Theorem 7.9]{ScholzePerfectoidspaces}, we know that $S\to T$ is $\varpi$-almost finite \'etale. By the tilting equivalence \cite[Theorem 5.25]{ScholzePerfectoidspaces}, we know that $S^\flat\to T^\flat$ is $\varpi^\flat$-almost finite \'etale of rank $r$, where $\varpi^\flat:=(\varpi, \varpi^{1/p},\dots)\in S^\flat$. It follows that $S^\flat[1/\varpi^\flat] \to T^\flat[1/\varpi^\flat]$ is finite \'etale (in particular finitely presented and finite) of constant rank $r$. Thus there exists $s\in S^\flat$ such that $S^\flat[1/s\varpi^\flat] \to T^\flat[1/s\varpi^\flat]$ is finite \'etale and finite free of rank $r$. By \autoref{lem.Almostfreeincharp}, $T^\flat$ is uniformly $s\varpi^\flat$-almost finite free of rank $r$. Thus after modulo $(\varpi^\flat)^p$, we obtain that $T/p\cong T/\varpi^p\cong T^\flat/(\varpi^\flat)^p$ is uniformly $(s\varpi^\flat)^\sharp$-almost finite free of rank $r$ over $S/p\cong S/\varpi^p\cong S^\flat/(\varpi^\flat)^p$.

For the last statement, note that if $S$ is a perfectoid valuation ring, then so is $S^\flat$ and thus $s$ divides $(\varpi^\flat)^n$ for some $n$. It follows that we can take $s=(\varpi^\flat)^{n}$ and hence $(s\varpi^\flat)^\sharp$-almost is the same as $\varpi$-almost (which is the same as $p$-almost). 
\end{proof}

\subsection{Transformation rule: the quasi-\'etale case}
In this subsection, we prove the transformation rule \autoref{thm.TransformationRule}. Our strategy is roughly as follows: 
we work modulo $(p^{1/p^\infty})$, in which case \autoref{lem.Almostfreeincharp} applies, and then we use \autoref{thm.BhattScholzeAlmostPurity} to lift back to mixed characteristic and perform estimates on normalized length. We start with a few lemmas.

\begin{lemma}
\label{lem.SOPisHilbertSamuel}
Let $A=W(k)[[x_2,\dots,x_d]]$. Suppose $p,g,y_3,\dots,y_d$ is a regular sequence on $A$ and $\epsilon_n\in \Ainfty$ is such that $\bar{\epsilon}_n=\bar{g}^{1/p^n}$ in $\Ainfty/(p^{1/p^\infty})$, then we have
\[
    \lambda_\infty(\Ainfty/(p,\epsilon_n,x_3,\dots,x_d)) = \frac{1}{p^n}\cdot \length_A\left(A/(p, g, y_3,\dots,y_d)\right).
\]
\end{lemma}
In the statement above, we write $\bar{g}^{1/p^n}$ to help us remember that $g$ does not necessarily have a $p^n$th root in $A_{\infty}$, but only in the quotient.

\begin{proof}
Since $\bar{\epsilon}_n=\bar{g}^{1/p^n}$ in $\Ainfty/(p^{1/p^\infty})$, we can choose $m\gg n$ such that $\bar{\epsilon}_n=\bar{g}^{1/p^n}$ holds in $\Ainfty/p^{1/p^m}$. It is easy to see that $p^{1/p^m},g,y_3,\dots,y_d$ and hence $p^{1/p^m},\epsilon_n,y_3,\dots,y_d$ is a regular sequence on $A_m$ and $\Ainfty$. Thus we have
    \begin{align*}
    \lambda_\infty(\Ainfty/(p,\epsilon_n,y_3,\dots,y_d))&=p^m\cdot\lambda_\infty(\Ainfty/(p^{1/p^m},\epsilon_n,y_3,\dots,y_d))\\
    &=p^m\cdot(1/p^{md})\cdot \length_{A_m/(p^{1/p^m})}(A_m/(p^{1/p^m},\bar{g}^{1/p^n},y_3,\dots,y_d))\\
  &=p^m\cdot(1/p^{md})\cdot(1/p^{n})\cdot \length_{A_m/(p^{1/p^m})}(A_m/(p^{1/p^m},g,y_3,\dots,y_d)) \\
    &=p^m\cdot(1/p^{n})\cdot \lambda_\infty(\Ainfty/(p^{1/p^m},g,y_3,\dots,y_d))\\
     &=(1/p^{n})\cdot \lambda_\infty(\Ainfty/(p,g,y_3,\dots,y_d))\\
      &=(1/p^{n})\cdot \length_A\left(A/(p, g, y_3,\dots,y_d)\right).
    \end{align*}
Here the first, third, and fifth equality follows from the regular sequence property, while the second, fourth, and sixth equality follows from the definition of normalized length. 
\end{proof}

\begin{lemma}
    \label{lem.I_inftyCompare}
    With notation as in \autoref{not.setup}, further assume $R$ is normal and that $f:R\to S$ is a finite extension of complete normal local domains. Suppose there exists some $\Phi\in \Hom_R(S,R)$ such that 
    \begin{enumerate}
        \item $\Phi$ is a free generator of $\Hom_R(S,R)$ as an $S$-module;
        \item $\Phi$ is surjective;
        \item $\Phi(\m_S)\subset \m_R$.
    \end{enumerate}
If $R[1/h]\to S[1/h]$ is \'{e}tale for some $0\neq h\in R$, then there is a canonical map 
    \[ 
        S^{\Ainfty}_{\perfd}/I^R_{\infty} S^{\Ainfty}_{\perfd}  \to S^{\Ainfty}_{\perfd}/I^S_\infty
    \] 
    which is an $h$-almost isomorphism. 

In particular, the above holds in the situation of \autoref{thm.TransformationRule} by taking $\Phi$ to be the trace map $\Tr: S \to R$ and taking $h=g$ (where $g$ is as in \autoref{not.setup}).
\end{lemma}
\begin{proof}
By \autoref{lem.CharacterizationIinftyviaAppGorenstein} we have
\[ 
    I^R_{\infty}=\bigcup_t (I_tR^{\Ainfty}_{\perfd}:_{R^{\Ainfty}_{\perfd}} u_t).
\]
Since $(S^{\Ainfty}_{\perfd} \otimes_RS)_{\perfd} = S^{\Ainfty}_{\perfd}$, by  \autoref{prop.BhattScholzeAlmostPuritypcompletefaithfullyflat} applied to the perfectoid ring $R^{\Ainfty}_{\perfd}$, we see that  $S^{\Ainfty}_{\perfd}$ is $p$-completely $h$-almost flat over $R^{\Ainfty}_{\perfd}$. Tensoring the short exact sequence 
   \[ 
        \begin{array}{rl}
        0\to R^{\Ainfty}_{\perfd}/(I_t R^{\Ainfty}_{\perfd}:_{R^{\Ainfty}_{\perfd}} u_t) \xrightarrow{\cdot u_t} & R^{\Ainfty}_{\perfd}/I_t R^{\Ainfty}_{\perfd}\\
        \to & R^{\Ainfty}_{\perfd}/(I_t+(u_t))R^{\Ainfty}_{\perfd}\to 0
        \end{array}
    \]
with $S^{\Ainfty}_{\perfd}$, we then have that 
   \[ 
       \ker\left(S^{\Ainfty}_{\perfd}/(I_t R^{\Ainfty}_{\perfd}:_{R^{\Ainfty}_{\perfd}} u_t)S^{\Ainfty}_{\perfd} \xrightarrow{\cdot u_t} S^{\Ainfty}_{\perfd}/I_t S^{\Ainfty}_{\perfd}\right)
    \]
is $h$-almost zero. It follows that $(I_t R^{\Ainfty}_{\perfd}:_{R^{\Ainfty}_{\perfd}} u_t)S^{\Ainfty}_{\perfd} \overset{a}{\simeq} (I_tS^{\Ainfty}_{\perfd}:_{S^{\Ainfty}_{\perfd}} u_t)$. Thus, after taking union over all $t$ we have
    \begin{align*}
I^R_{\infty} S^{\Ainfty}_{\perfd} &\overset{a}{\simeq} \bigcup_t (I_tS^{\Ainfty}_{\perfd}:_{S^{\Ainfty}_{\perfd}} u_t)\\
&=\{x\in S^{\Ainfty}_{\perfd}\mid R\to S^{\Ainfty}_{\perfd}\text{ such that }1\mapsto x \text{ is not pure}\}\eqqcolon J^{S/R}_\infty
\end{align*}
where the equality follows from \autoref{lem.SplittingApproxGorenstein}. 

Now we claim that $J^{S/R}_\infty=I^S_\infty$. Since $R\to S$ splits by (a) and (b), we obviously have $J^{S/R}_\infty \subseteq I^S_\infty$. For the converse inclusion, note that any map $S^{\Ainfty}_{\perfd} \to R$, which sends $x \mapsto 1$, factors through $\Phi \in \Hom_R(S,R) \cong S$, see \cite[Appendix F.17(a)]{KunzKahlerDifferentials}. Then by (b) and (c), we see that $S \to S^{\Ainfty}_{\perfd}$ sending $1\mapsto x$ splits, giving $J^{S/R}_\infty \supseteq I^S_\infty$ as desired.

The above discussion shows that $I^R_{\infty} S^{\Ainfty}_{\perfd} \overset{a}{\simeq} I^S_\infty$. To induce the map in the statement of the lemma, it suffices to show that $R^{\Ainfty}_{\perfd} \to S^{\Ainfty}_{\perfd}$ sends $I^R_{\infty}$ into $I^S_{\infty}=J^{S/R}_\infty$. But this is clear: if $z \in I^R_{\infty}$ then $R \xrightarrow{1 \mapsto z} R^{\Ainfty}_{\perfd}$ does not split, it then follows that $R \xrightarrow{1 \mapsto z} S^{\Ainfty}_{\perfd}$ does not split as well, i.e., $z \in J^{S/R}_{\infty}$ as desired.

Finally, in the situation of \autoref{thm.TransformationRule}, we have $R\to S$ is \'etale in codimension one. In particular, $A[1/g]\to R[1/g]\to S[1/g]$ is \'etale in codimension one and thus \'etale by Zariski-Nagata's purity theorem \cite[\href{https://stacks.math.columbia.edu/tag/0BMB}{Tag 0BMB}]{stacks-project}. Since $A[1/g]\to R[1/g]$ is \'etale, we have that $R[1/g]\to S[1/g]$ is \'etale so we can take $h=g$. 
Furthermore, $\Hom_R(S,R)\simeq \Tr\cdot S$ which gives (a) by taking $\Phi = \Tr$ (see \autoref{lem.TraceForNormalDomains}).  Part (b) follows from the fact that a splitting, which is necessarily surjective, must be a pre-multiple of $\Phi = \Tr$, and this forces $\Phi$ to be surjective.  Finally, (c) follows again by \autoref{lem.TraceForNormalDomains}.
\end{proof}

\begin{proof}[Proof of \autoref{thm.TransformationRule}]
In view of \autoref{eq.remResidueFieldExtensions}, it suffices for us to prove that $\lambda_\infty(S^{\Ainfty}_{\perfd}/I^S_\infty) =  r\cdot \lambda_\infty(R^{\Ainfty}_{\perfd}/I^R_{\infty})$.

As $R$ is normal, its singular locus has codimension $\ge 2$, hence we can find $h\in A$ such that $p,h$ is a regular sequence on $A$ and $R[1/h]$ is regular.\footnote{Note that we are not assuming $A[1/h]\to R[1/h]$ is finite \'etale, i.e., $h$ might be different from $g$ in \autoref{not.setup}.} 
Now, since $R \subseteq S$ is quasi-\'etale, we see that $R[1/h]\to S[1/h]$ is finite \'etale by Zariski-Nagata's purity theorem \cite[\href{https://stacks.math.columbia.edu/tag/0BMB}{Tag 0BMB}]{stacks-project}.  Moreover, by replacing $h$ by a multiple if necessary, we may assume further that $R[1/h]\to S[1/h]$ is finite free while ensuring that $p,h$ remains a regular sequence. To see this, let $\mathfrak{p}=pA$, then $R_\mathfrak{p}\to S_\mathfrak{p}$ is a finite \'etale extension of PIDs hence is free. Therefore, one may pick $h'\in A-\mathfrak{p}$ such that $R[1/h']\to S[1/h']$ is free and replace $h$ by $hh'$.

By \cite[Proposition 8.13]{BhattScholzepPrismaticCohomology}, we have $R^{\Ainfty}_{\perfd}/(p^{1/p^\infty})\simeq (R/p)_\perf$, $S^{\Ainfty}_{\perfd}/(p^{1/p^\infty})\simeq (S/p)_\perf$, and $\Ainfty/(p^{1/p^\infty})\simeq (A/p)_\perf$. Also note that the image of $(h)_\perfd$ in $(A/p)_\perf$ is $(\bar{h}^{1/p^\infty})$ where $\bar{h}$ denotes the image of $h$ in $A/p$ (see \autoref{lem.gPerfd=AllpPowerRootsofg}). Since $R^{\Ainfty}_{\perfd}/p\to S^{\Ainfty}_{\perfd}/p$ is $h$-almost finite \'etale by \autoref{thm.BhattScholzeAlmostPurity}, so is its base change to $R^{\Ainfty}_{\perfd}/(p^{1/p^\infty})$, i.e. $(R/p)_\perf\to (S/p)_\perf$ is $\bar{h}$-almost finite \'etale. Now, by \autoref{thm.BhattScholzeAlmostPurity} and our choice of $h$, we have  
$$S^{\Ainfty}_{\perfd}[1/h]\cong (R^{\Ainfty}_{\perfd} \otimes_R S)[1/h]\cong R^{\Ainfty}_{\perfd}[1/h] \otimes_{R[1/h]} S[1/h]$$ 
is finite free of rank $r$ over $R^{\Ainfty}_{\perfd}[1/h]$, and this is clearly preserved modulo the ideal $(p^{1/p^\infty})$. Thus by \autoref{lem.Almostfreeincharp}, $(R/p)_\perf\to (S/p)_\perf$ is uniformly $\bar{h}$-almost free of rank $r$. This means that for any $\epsilon \in (h)_\perfd $ such that $\epsilon=\bar{h}^{1/p^n}\mod (p^{1/p^\infty})$, we have 
\[ 
    (S/p)_\perf\xrightarrow{\phi_\epsilon}(R/p)_\perf^r\xrightarrow{\psi_\epsilon}(S/p)_\perf
\]
where the two-way compositions are both multiplication by $\bar{h}^{1/p^n}$. Now since $S^{\Ainfty}_{\perfd}/p$ is $h$-almost finite projective over $R^{\Ainfty}_{\perfd}/p$ by \autoref{thm.BhattScholzeAlmostPurity}, it follows from \autoref{lem.Almostprojectivelifts} that the map $\epsilon\cdot \phi_\epsilon \circ can$: $S^{\Ainfty}_{\perfd}/p \to (R/p)_\perf^r$ lifts along $(R^{\Ainfty}_{\perfd}/p)^r\twoheadrightarrow (R/p)_\perf^r$ to $\widetilde{\phi}_\epsilon$: $S^{\Ainfty}_{\perfd}/p \to (R^{\Ainfty}_{\perfd}/p)^r$. Similarly, since $(R^{\Ainfty}_{\perfd}/p)^r$ is projective over $R^{\Ainfty}_{\perfd}/p$, $\psi_\epsilon\circ can$ lifts along $S^{\Ainfty}_{\perfd}/p\twoheadrightarrow (S/p)_\perf$ to $\widetilde{\psi}_\epsilon$. In summary, we have the following diagram:
    \[
    \xymatrix{
    S^{\Ainfty}_{\perfd}/p \ar[r]^-{\widetilde{\phi}_\epsilon}  \ar[d]^-{can}& (R^{\Ainfty}_{\perfd}/p)^r \ar[r]^-{\widetilde{\psi}_\epsilon} \ar[d]^-{can}& S^{\Ainfty}_{\perfd}/p \ar[d]^-{can} \\
    (S/p)_\perf \ar[r]^-{\epsilon \cdot\phi_\epsilon}  & (R/p)_\perf^r \ar[r]^-{\psi_\epsilon}  & (S/p)_\perf 
    }.
    \]

We first show that $\lambda_\infty(S^{\Ainfty}_{\perfd}/I^S_\infty)\le r\cdot\lambda_\infty(R^{\Ainfty}_{\perfd}/I^R_{\infty})$. Consider the module $$C_\epsilon :=\coker(\widetilde{\psi}_\epsilon)=(S^{\Ainfty}_{\perfd}/p)/\im(\widetilde{\psi}_\epsilon).$$ As $S^{\Ainfty}_{\perfd}/p$ is almost finitely generated over $R^{\Ainfty}_{\perfd}/p$, so is $C_\epsilon$. Hence there exists a finitely generated $R^{\Ainfty}_{\perfd}/p$-submodule $D_\epsilon$ of $C_\epsilon$ such that $\epsilon\cdot C_\epsilon\subseteq D_\epsilon$. By the construction of $\widetilde{\psi}_\epsilon$, $\epsilon^2\cdot(C_\epsilon/(p^{1/p^\infty})C_\epsilon)=0$, and thus $\epsilon^2\cdot C_\epsilon\subseteq (p^{1/p^\infty})C_\epsilon$. As $D_\epsilon$ is finitely generated, there exists $m\gg n$ such that $\epsilon^2\cdot D_\epsilon\subseteq p^{1/p^m}C_\epsilon$. It follows that $\epsilon^3\cdot C_\epsilon\subseteq \epsilon^2\cdot D_\epsilon\subseteq p^{1/p^m}C_\epsilon$, in other words, 
\begin{equation}
\label{eqn.EqnInProofTransRule}\tag{\ref*{thm.TransformationRule}.1}
    \epsilon^3 \cdot (S^{\Ainfty}_{\perfd}/p)\subseteq \im(\widetilde{\psi}_\epsilon)+p^{1/p^m}(S^{\Ainfty}_{\perfd}/p) 
\end{equation}
for some $m\gg n$. Note that by \autoref{lem.I_inftyCompare}, we have
$$\lambda_\infty(S^{\Ainfty}_{\perfd}/I^S_\infty) = \lambda_\infty(S^{\Ainfty}_{\perfd}/I^R_\infty S^{\Ainfty}_{\perfd}) \le r\cdot \lambda_\infty(R^{\Ainfty}_{\perfd}/I^R_{\infty})+\lambda_\infty(C_\epsilon/I^R_{\infty} C_\epsilon).$$
Therefore it suffices to show that for any $\delta>0$, there exists some $\epsilon\in (h)_\perfd$ such that $\lambda_\infty(C_\epsilon/I^R_{\infty} C_\epsilon)< \delta$. Since $p, h$ is a regular sequence, we can pick $y_3,\dots,y_d\in A$ such that $p,h,y_3,\dots,y_d$ form a system of parameters of $A$. Now for any fixed $\delta>0$, by \autoref{lem.SOPisHilbertSamuel}, there exists some $\epsilon\in (h)_\perfd$ such that $$\lambda_\infty(\Ainfty/(p,\epsilon^3,y_3,\dots,y_d))<\delta/(\rank_A S).$$ 
Then we have 
    \begin{align*}
    \lambda_\infty(C_\epsilon/I^R_{\infty} C_\epsilon)&\le \lambda_\infty\left((S^{\Ainfty}_{\perfd}/p)/(\m_A(S^{\Ainfty}_{\perfd}/p)+\im(\widetilde{\psi}_\epsilon))\right)\\
    &\le p^m\cdot  \lambda_\infty\left((S^{\Ainfty}_{\perfd}/p)/((p^{1/p^m},y_3,\dots,y_d)(S^{\Ainfty}_{\perfd}/p)+\im(\widetilde{\psi}_\epsilon))\right)\\
    &\le p^m \cdot  \lambda_\infty\left((S^{\Ainfty}_{\perfd}/p)/(p^{1/p^m},\epsilon^3,y_3,\dots,y_d)(S^{\Ainfty}_{\perfd}/p)\right)\\
     &= p^m \cdot  \lambda_\infty\left(S^{\Ainfty}_{\perfd}/(p^{1/p^m},\epsilon^3,y_3,\dots,y_d)S^{\Ainfty}_{\perfd}\right)\\
    &=p^m \cdot (\rank_A S) \cdot \lambda_\infty(\Ainfty/(p^{1/p^m},\epsilon^3,y_3,\dots,y_d))\\
    &=(\rank_A S) \cdot \lambda_\infty(\Ainfty/(p,\epsilon^3,y_3,\dots,y_d))\\
    &<\delta
    \end{align*}
where the inequality in the third line follows from \autoref{eqn.EqnInProofTransRule}, the equality on the fifth line follows from \autoref{prop.Normalizedlengthequalsrank}, and the equality on the sixth line follows from the fact that $p, \epsilon,y_3,\dots,y_d$ is a regular sequence on $A_{\infty}$.

Finally, for the other direction $r\cdot\lambda_\infty(S^{\Ainfty}_{\perfd}/I^R_{\infty})\le \lambda_\infty(S^{\Ainfty}_{\perfd}/I^S_\infty)$, note that $\coker(\widetilde{\phi}_\epsilon)$ is a finitely generated $R^{\Ainfty}_{\perfd}/p$-module. Hence we have $\epsilon^2\cdot \coker(\widetilde{\phi}_\epsilon)\subseteq p^{1/p^m}\coker(\widetilde{\phi}_\epsilon)$, i.e., $$\epsilon^2 \cdot (R^{\Ainfty}_{\perfd}/p)^r\subseteq \im(\widetilde{\phi}_\epsilon)+p^{1/p^m}(R^{\Ainfty}_{\perfd}/p)^r $$ for some $m\gg n$. The rest of the argument is essentially the same and we omit the details.
\end{proof}

\subsection{Transformation rule: a more general case} In this subsection, we prove a technical generalization of the transformation rule in the sense of \cite[Theorem 4.8]{CarvajalRojasFiniteTorsors}, which will later give us the application on divisor class groups of BCM-regular singularities.

\begin{proposition}
    \label{prop.GeneralTransformationRuleForSpecialSOP}
    Let $(R,\m) \subseteq (S,\mathfrak{n})$ be a finite extension of Noetherian complete normal local domains of mixed characteristic $(0, p)$ with perfect residue fields that satisfies conditions $(a)-(c)$ of \autoref{lem.I_inftyCompare}. Let $(A,\m_A)\to (R,\m)$ be a finite extension such that $A$ is a complete unramified regular local ring and $A/\m_A\cong R/\m=k$. Let $0\neq g\in A$ be such that $A[1/g] \subseteq R[1/g]$ and $R[1/g] \subseteq S[1/g]$ are finite \'etale and finite free. Then there exists a regular system of parameters $\underline{x}:= p, x_2,\dots, x_d \in \m_A \setminus \m_A^2$
 ($\underline{x}$ depends on $g$, which itself depends on $R$ and $S$) such that, if we identify $A=W(k)[[x_2,\dots,x_d]]$ and form $\Ainfty$, $R^{\Ainfty}_{\perfd}$, $S^{\Ainfty}_{\perfd}$ as in \autoref{not.setup}, then we have 
    \[ 
        t \cdot s^{\underline{x}}_{\perfd}(S)=r\cdot s^{\underline{x}}_{\perfd}(R)
    \]
    where $r=[K(S): K(R)]$ is the generic rank of the extension and $t = [S/\m_S : R/\m_R]$ is the degree of the residue field extension.
    In fact, we may choose any $x_2, \dots, x_d \in \m_A \setminus \m_A^2$ such that both $p, x_2, \dots, x_d$ and $g, x_2, \dots, x_d$ are systems of parameters of $A$.
\end{proposition}
\begin{proof}
We first note that there exists $x_2,\dots,x_d  \in \m_A \setminus \m_{A}^2$ such that both $p, x_2, \dots, x_d$ and $g, x_2, \dots, x_d$ are systems of parameters of $A$.
Since $pg$ is a nonzero divisor, we may choose $x_2, \dots, x_d \in \m_A \setminus \m_A^2$ such that $pg, x_2, \dots, x_d$ form a system of parameters. Then $p, x_2, \dots, x_d$ and $g, x_2, \dots, x_d$ are both system of parameters.


We next note that our assumption implies that $A[1/g]\to R[1/g]\to S[1/g]$ are finite \'etale maps. By \autoref{thm.BhattScholzeAlmostPurity}, we have that 
$\Ainfty/p \to R^{\Ainfty}_{\perfd}/p \to S^{\Ainfty}_{\perfd}/p$ are $g$-almost finite \'etale maps. Thus, modulo $(x_2.\dots, x_d)_\perfd=(x_2^{1/p^\infty},\dots, x_d^{1/p^\infty})^-\subseteq \Ainfty$, we have
$$ \Ainfty/(p, (x_2,\dots, x_d)_\perfd) \to R^{\Ainfty}_\perfd/(p,(x_2,\dots,x_d)_\perfd) \to S^{\Ainfty}_\perfd/(p,(x_2,\dots,x_d)_\perfd)$$
are $p$-almost finite \'etale maps since $g=p^cu$ for some unit $u\in A$ modulo $(x_2,\dots,x_d)$. Note that here we are implicitly using \autoref{lem.gPerfd=AllpPowerRootsofg}, i.e., {modulo $p$,} $(x_2,\dots,x_d)_\perfd\subseteq R^{\Ainfty}_\perfd$ is the perfectoidization of $(x_2,\dots,x_d)$ inside $R^{\Ainfty}_\perfd$, and similarly for $(x_2,\dots,x_d)_\perfd\subseteq S^{\Ainfty}_\perfd$.
Now by the tilting equivalence (see \cite[Theorem 5.25]{ScholzePerfectoidspaces}), we know that 
\[ 
      \Ainfty/(x_2,\dots, x_d)_\perfd \to R^{\Ainfty}_\perfd/(x_2,\dots,x_d)_\perfd \to S^{\Ainfty}_\perfd/(x_2,\dots,x_d)_\perfd.
\] 
are $p$-almost finite \'etale maps. Moreover, we have that $S^{\Ainfty}_\perfd/(x_2,\dots,x_d)_\perfd[1/p]$ is finite free over $R^{\Ainfty}_\perfd/(x_2,\dots,x_d)_\perfd[1/p]$ (and both of them are finite \'etale over the perfectoid field $\Ainfty/(x_2,\dots, x_d)_\perfd[1/p]$). Therefore,
\[
(R^{\Ainfty}_\perfd/(x_2,\dots,x_d)_\perfd)[1/p] \to (S^{\Ainfty}_\perfd/(x_2,\dots,x_d)_\perfd)[1/p].
\]
is a finite free and finite \'etale extension, and each of them is a finite product of perfectoid fields. By the tilting equivalence again, we know that 
\[ 
  (R^{\Ainfty}_\perfd/(x_2,\dots,x_d)_\perfd)^\flat[1/p^\flat] \to (S^{\Ainfty}_\perfd/(x_2,\dots,x_d)_\perfd)^\flat[1/p^\flat].
\] 
is also a finite \'etale and finite free extension (and each of which is a finite product of perfect fields). By \autoref{lem.Almostfreeincharp}, $(S^{\Ainfty}_\perfd/(x_2,\dots,x_d)_\perfd)^\flat$ is uniformly $p^\flat$-almost finite free of rank $r$ over $(R^{\Ainfty}_\perfd/(x_2,\dots,x_d)_\perfd)^\flat$, and thus $S^{\Ainfty}_\perfd/(p, (x_2,\dots,x_d)_\perfd)$ is uniformly $p$-almost finite free of rank $r$ over $R^{\Ainfty}_\perfd/(p, (x_2,\dots,x_d)_\perfd)$. 
In particular, for any $\epsilon \in (g)_\perfd \subseteq \Ainfty$ such that $\epsilon \equiv \bar{g}^{1/p^n}\equiv  p^{c/p^n}\mod {(x_2,\dots,x_d)_\perfd}$, we have 
    \[    
        S^{\Ainfty}_\perfd/(p, (x_2,\dots,x_d)_\perfd)\xrightarrow{\phi_\epsilon}\left(R^{\Ainfty}_\perfd/(p, (x_2,\dots,x_d)_\perfd)\right)^r\xrightarrow{\psi_\epsilon}S^{\Ainfty}_\perfd(p, (x_2,\dots,x_d)_\perfd)
    \]
    where the two-way compositions are both multiplication by $p^{c/p^n}$. Now, similar to the argument as in the proof of  \autoref{thm.TransformationRule}, since $S^{\Ainfty}_\perfd/p$ is $g$-almost finite projective over $R^{\Ainfty}_\perfd/p$ by \autoref{thm.BhattScholzeAlmostPurity}, we can apply \autoref{lem.Almostprojectivelifts} to obtain the following commutative diagram:
    \[
        {\small\xymatrix{
            S^{\Ainfty}_{\perfd}/p \ar[r]^-{\widetilde{\phi}_\epsilon}  \ar[d]^-{can}& (R^{\Ainfty}_{\perfd}/p)^r \ar[r]^-{\widetilde{\psi}_\epsilon} \ar[d]^-{can}& S^{\Ainfty}_{\perfd}/p \ar[d]^-{can} \\
            S^{\Ainfty}_\perfd/(p,(x_2,\dots,x_d)_\perfd) \ar[r]^-{\epsilon \cdot\phi_\epsilon}  & (R^{\Ainfty}_\perfd/(p,(x_2,\dots,x_d)_\perfd))^r \ar[r]^-{\psi_\epsilon}  & S^{\Ainfty}_\perfd/(p,(x_2,\dots,x_d)_\perfd)
        }}
    \]
        
    The rest of the argument is similar to the proof of \autoref{thm.TransformationRule}. In view of \autoref{eq.remResidueFieldExtensions}, it suffices to prove that $\lambda_\infty(S^{\Ainfty}_{\perfd}/I^S_\infty) =  r\cdot \lambda_\infty(R^{\Ainfty}_{\perfd}/I^R_{\infty})$. Below we will only show $$\lambda_\infty(S^{\Ainfty}_{\perfd}/I^S_\infty)\le r\cdot\lambda_\infty(R^{\Ainfty}_{\perfd}/I^R_{\infty}),$$ as the other direction is similar (and slightly easier). 
    
    We set $C_\epsilon=\coker(\widetilde{\psi}_\epsilon)$. Since $(p,(x_2,\dots,x_d)_\perfd)=(p,x_2^{1/p^\infty},\dots,x_d^{1/p^\infty})$ by \autoref{lem.gPerfd=AllpPowerRootsofg}, by essentially the same argument as in the paragraph above \autoref{eqn.EqnInProofTransRule} in the proof of \autoref{thm.TransformationRule}, we can find $m\gg n$ such that $\epsilon^3\cdot C_\epsilon\subseteq (x_2^{1/p^m},\dots,x_d^{1/p^m})C_\epsilon$, i.e., 
    \begin{equation}
    \label{eqn.EqnInProofofGenTransRule}
        \epsilon^3 \cdot (S^{\Ainfty}_{\perfd}/p)\subseteq \im(\widetilde{\psi}_\epsilon)+(x_2^{1/p^m},\dots,x_d^{1/p^m})(S^{\Ainfty}_{\perfd}/p).
    \end{equation} 
    Increasing $m$ if necessary, we may as well assume that $\epsilon =p^{c/p^n}$ in $\Ainfty/(x_2^{1/p^m},\dots,x_d^{1/p^m})$. 
    By \autoref{lem.I_inftyCompare} (the conditions are satisfied by our assumptions), we have
$$\lambda_\infty(S^{\Ainfty}_{\perfd}/I^S_\infty) = \lambda_\infty(S^{\Ainfty}_{\perfd}/I^R_\infty S^{\Ainfty}_{\perfd}) \le r\cdot \lambda_\infty(R^{\Ainfty}_{\perfd}/I^R_{\infty})+\lambda_\infty(C_\epsilon/I^R_{\infty} C_\epsilon).$$
Therefore it suffices to show that for any $\delta>0$, there exists some $\epsilon\in (g)_\perfd$ such that $\lambda_\infty(C_\epsilon/I^R_{\infty} C_\epsilon)< \delta$. Now we estimate the normalized length:
        \begin{align*}
        \lambda_\infty(C_\epsilon/I^R_{\infty} C_\epsilon)&\le \lambda_\infty\left((S^{\Ainfty}_{\perfd}/p)/(\m_A(S^{\Ainfty}_{\perfd}/p)+\im(\widetilde{\psi}_\epsilon))\right)\\
        &\le p^{m(d-1)}\cdot  \lambda_\infty\left((S^{\Ainfty}_{\perfd}/p)/((x_2^{1/p^m},\dots,x_d^{1/p^m})(S^{\Ainfty}_{\perfd}/p)+\im(\widetilde{\psi}_\epsilon))\right)\\
        &\le p^{m(d-1)} \cdot  \lambda_\infty\left((S^{\Ainfty}_{\perfd}/p)/(\epsilon^3,x_2^{1/p^m},\dots,x_d^{1/p^m})(S^{\Ainfty}_{\perfd}/p)\right)\\
        &\le p^{m(d-1)} \cdot  \lambda_\infty(S^{\Ainfty}_{\perfd}/(\epsilon^3,x_2^{1/p^m},\dots,x_d^{1/p^m})S^{\Ainfty}_{\perfd})\\
        &=p^{m(d-1)} \cdot (\rank_A S) \cdot \lambda_\infty(\Ainfty/(\epsilon^3,x_2^{1/p^m},\dots,x_d^{1/p^m}))\\
        &=p^{m(d-1)} \cdot (\rank_A S) \cdot \lambda_\infty(\Ainfty/(p^{3c/p^n},x_2^{1/p^m},\dots,x_d^{1/p^m}))\\
        &=(3c/p^n)\cdot (\rank_A S) \cdot \lambda_\infty(\Ainfty/(p,x_2,\dots,x_d))\\
        &=(3c/p^n)\cdot (\rank_A S),
        \end{align*}
    where the inequality on the third line follows from \autoref{eqn.EqnInProofofGenTransRule}, the equality on the fifth line follows from \autoref{prop.Normalizedlengthequalsrank}, the equality in the sixth line follows from our choice that $\epsilon =p^{c/p^n}$ in $\Ainfty/(x_2^{1/p^m},\dots,x_d^{1/p^m})$, the equality in the seventh line follows from the fact that $p, x_2,\dots,x_d$ is a regular sequence on $\Ainfty$.
    
    Finally, letting $n\to\infty$, we obtain that $\lambda_\infty(C_\epsilon/I^R_{\infty} C_\epsilon)< \delta$ which completes the proof.
\end{proof}

The transformation rule actually lets us compute the perfectoid signature of singularities with finite regular covers.  That is, if $(R,\m, k) \subseteq (S, \mathfrak{n}, k)$ is quasi-\'etale and $S$ is regular, then 
\[
    s^{\underline{x}}_{\perfd}(R) = {1 \over [K(S): K(R)]}
\]
regardless of what $\underline{x}$ is.  Even more, if $R \subseteq S$ is not quasi-\'etale, then as long as $\underline{x}$ is chosen sufficiently carefully, one can still determine the perfectoid signature.
  This occurs for Du Val singularities when $p > 5$, as we now verify.  Compare with \cite[Example 18]{HunekeLeuschkeTwoTheoremsAboutMaximal}.

\begin{corollary}
    \label{cor.DuValSingularities}
    Suppose that $(R, \m)$ is a two-dimensional Noetherian complete Gorenstein rational singularity of residual characteristic $p > 0$ and with algebraically closed residue field.  Then for every appropriate choice of system of parameters $\underline{x}=p, x_2,\dots, x_d\subseteq R$ in the sense of \autoref{prop.GeneralTransformationRuleForSpecialSOP}, we have the following computations of $s^{\underline{x}}_{\perfd}(R)$ and $e^{\underline{x}}_{\perfd}(R)$.
    \begin{center}
        {\renewcommand{\arraystretch}{1.3}
        \begin{tabular}{c|c|c|c}
            {\bf type} & {\bf residual characteristic} & {$s^{\underline{x}}_{\perfd}(R)$} & {$e^{\underline{x}}_{\perfd}(R)$} \\
            \hline
            $A_n$ & $p > 2$ & ${1 \over n+1}$ & $2-{1 \over n+1}$\\
            $D_n$ & $p > 3$ & ${1 \over 2(2n - 4)} = {1 \over 4(n-2) }$ & $2-{1 \over 4(n-2) }$ \\
            $E_6$ & $p > 5$ & ${1 \over 24}$ & $2- {1 \over 24}$ \\
            $E_7$ & $p > 5$ & ${1 \over 48}$ & $2- {1 \over 48}$ \\
            $E_8$ & $p > 5$ & ${1 \over 120}$ &  $2- {1 \over 120}$  \
        \end{tabular}}
    \end{center}    
    Furthermore, in the case that $R$ is an $A_n$ singularity where $n +1$ is not divisible by $p$, or a $D_n$ singularity where $n-2$ is not divisible by $p \neq 2$, or in the case of $E_6$, $E_7$ or $E_8$, then no special system of parameters $\underline{x}$ needs be chosen and $s^{\underline{x}}_{\perfd}(R)$ and $e^{\underline{x}}_{\perfd}(R)$ are independent of the choice of $\underline{x}$.
\end{corollary}
\begin{proof}
First of all, two-dimensional Gorenstein rational singularities have multiplicity $2$ (see \cite[Lemma 4.1]{CarvajalRojasMaPolstraSchwedeTuckerCoversOfRDP}) and thus the computation of $e^{\underline{x}}_{\perfd}(R)$ follows from the computation of $s^{\underline{x}}_{\perfd}(R)$ and \autoref{cor.DegreeTwoHypersurface}.
    
To compute $s^{\underline{x}}_{\perfd}(R)$, we note that it follows from \cite[Section 4]{CarvajalRojasMaPolstraSchwedeTuckerCoversOfRDP} that these singularities have finite regular covers whose index is the reciprocal of the identified perfectoid-signature.  Furthermore, there is a trace-like map in each case sending the maximal ideal to the maximal ideal. Therefore the result follows from \autoref{thm.TransformationRule}, \autoref{prop.GeneralTransformationRuleForSpecialSOP}, and \autoref{thm.CharacterizationOfRegularLocalRings}.

For the second statement, we need to identify those $n > 0$ for which these regular covers are quasi-\'etale (note a cyclic index-$p$ cover is not quasi-\'etale).  The $A_n$ singularities have regular cyclic covers of index $n+1$ which gives the conclusion.
    For a $D_n$ singularity where $n-2$ is not divisible by $p \neq 2$, then by \cite[Proposition 4.4]{CarvajalRojasMaPolstraSchwedeTuckerCoversOfRDP}, it has a cyclic, index 2, quasi-\'etale, cover which is $A_{2n-5}$.  That further has a finite cyclic cover of index $2n - 5 + 1 = 2n - 4$.  Since $p \neq 2$, this is quasi-\'etale exactly when $p$ does not divide $n-2$. Finally, the regular covers of the $E_6, E_7, E_8$ correspond to degree $24, 48, 120$-extensions respectively, which are quasi-\'etale as long as $p > 5$, see  \cite[Propositions 4.7, 4.9, 4.11]{CarvajalRojasMaPolstraSchwedeTuckerCoversOfRDP}. 
\end{proof}

\section{Applications}
In this section we provide several applications by combining the results of the previous sections.  The transformation rule for perfectoid signature plays a particularly central role. 
\begin{itemize}
    \item{}  In \autoref{thm.FundamentalGroupApplication} we prove an analog of \cite[Theorem 1]{XuEtaleFundamentalGroup} and \cite[Corollary 1.4]{XuZhuangUniquenessOfTheMinimizerOfTheNormalizedVolume} from characteristic 0 and of \cite[Theorem 5.1]{CarvajalRojasSchwedeTuckerFundamentalGroup} from characteristic $p > 0$.  
 Specifically, we provide an explicit bound on the size of the local \'etale fundamental group of a BCM-regular pair.
    \item{} By taking a cone, we obtain a version of this result for large open subsets of certain projective schemes, see  \autoref{thm.FinitenessFunGroupsForBigOpenInLogFano}.  
    \item{} In \autoref{thm.GrebKebekusPeternellVersion} we obtain a version of \cite[Theorem 1.1]{GrebKebekusPeternellEtaleFundamentalGroupsKLT} and \cite[Theorem 1]{StibitzEtaleCoversAndLocalFunGroups}.  We prove that towers of quasi-\'etale generically Galois covers of a BCM-regular scheme are eventually \'etale.
    \item  In \autoref{thm.DivisorClassGroupApplication}, we prove finiteness and bound the size of the torsion part of the divisor class group for BCM-regular singularities, see \cite{PolstraATheoremAboutMCM,MartinTheNumberOfTorsionDivisors,CarvajalRojasFiniteTorsors} for the positive characteristic counterparts.
\end{itemize} 

Before starting these proofs, we need a lemma.

\begin{lemma}[{\cf \cite[Lemma 5.1]{CarvajalRojasMaPolstraSchwedeTuckerCoversOfRDP}, \cite{CarvajalRojasFiniteTorsors}}]
    \label{lem.QuasiEtaleStuffForWeakBCMRegular}
    Suppose $(R, \m) \subseteq (S, \bn)$ is a finite extension of normal complete domains with $R \subseteq S$ \'etale in codimension 1 (or more generally, if $\Hom_R(S, R) \cong S$ as $S$-modules and if an $S$-module generator $T$ of $\Hom_R(S, R)$ satisfies $T(\bn) \subseteq \m$).  If $R$ is weakly BCM-regular, then $R \to S$ is split as a map of $R$-modules and $S$ is also weakly BCM-regular. 
\end{lemma}
\begin{proof}
    As in \autoref{lem.I_inftyCompare}, if $R \to S$ is quasi-\'etale then if $T : S \to R$ is the trace map, we have that $\Hom_R(S, R) = S \cdot T$ and $T(\bn) \subseteq \m$, see \autoref{lem.TraceForNormalDomains}.  Hence we have reduced to the general case.  

    Now, as $R$ is weakly BCM-regular, we have that $R \to S \to B$ splits for any perfectoid big Cohen-Macaulay $R^+ = S^+$-algebra $B$, and so $R \to S$ splits as well.    By Hom-tensor adjunction, $\Hom_R(B, R) \cong \Hom_S(B, \Hom_R(S, R)) \cong \Hom_S(B, S)$ and it follows that any map $\phi : B \to R$ factors as 
    \[
        \phi : B \xrightarrow{\psi} S \xrightarrow{T} R.
    \]
    If $S$ is not weakly BCM-regular, then there exists a perfectoid big Cohen-Macaulay $S^+$-algebra $B$ so that $\psi(B) \subseteq \bn$ for all $\psi\in\Hom_S(B, S)$ and hence $\phi(B) = T(\psi(B)) \subseteq \m$ for all $\phi\in \Hom_R(B, R)$, contradicting the fact that $R$ is weakly BCM-regular.
\end{proof}

\begin{theorem}
    \label{thm.FundamentalGroupApplication}
Let $X=\Spec(R)$ where $(R,\m)$ is a Noetherian complete normal local domain of mixed characteristic $(0,p)$ with algebraically closed residue field. Further suppose that there exists $\Delta \geq 0$ such that $K_R + \Delta$ is $\bQ$-Cartier and such that $(R, \Delta)$ is BCM-regular. 
Then $\pi_1^{\text{\'et}}(X-Z,\bar{q})$ is finite for any closed subscheme $Z$ 
of codimension at least 2. Moreover, the order of $\pi_1^{\text{\'et}}(X-Z,\bar{q})$ is at most $1/s^{\underline{x}}_\perfd(R)$.  Finally, $p$ does not divide $|\pi_1^{\text{\'et}}(X-Z,\bar{q})|$.
\end{theorem}

\begin{remark}
We remark here that since $X$ is normal, $X-Z$ is connected. Hence $\pi_1^{\text{\'et}}(X-Z,\bar{q})$ is independent of the choice of $\bar{q}$ up to an isomorphism. So for convenience, we choose $\bar{q}$ to be $\Spec({\text{Frac}(R)}^{\text{sep}})\to X-Z$. Also since $X$ is normal, we may take $Z=X_{\text{sing}}$, the singular locus of $X$ (another interesting case is $Z=\{\m\}$).
\end{remark}

\begin{proof}
    To compute the \'etale fundamental group $\pi_1^{\text{\'et}}(X-Z,\bar{q})$ we first observe that we can study certain finite covers of $X$.  Let $\mathcal{C}$ be the category of finite morphisms of normal schemes over $X$ which are \'etale except possibly over $Z$. By Zariski's Main Theorem, we have the equivalence of categories $\mathcal{C}\simeq {(X-Z)}_{\text{\'et}}$ via $(Y\to X)\mapsto (Y\times_X(X-Z)\to X-Z)$. The fully faithfulness is clear. For any finite \'etale cover $U$ of $X-Z$, we may apply the Zariski's Main Theorem to factor the quasi-finite map $U\to X$ into $U\hookrightarrow Y\to X$ where the first map is an open immersion and the second map is finite. This gives us the essential surjection. Note that $\pi_1^{\text{\'et}}(X-Z,\bar{q})$ pro-represents the functor $\Hom_{X-Z}(\bar{q},-)$ on ${(X-Z)}_{\text{\'et}}$. Hence we may simply consider the functor $\Hom_{X}(\bar{q},-)$ on $\mathcal{C}$ by the equivalence of the categories.

Now we let 
$$\cdots\to X_i=\Spec(S_i) \to X_{i-1}=\Spec(S_{i-1})\to \cdots \to X_{0}=X$$
be a tower of finite local morphisms of normal schemes such that each one is quasi-\'etale. 
By our choice of $\bar{q}:\Spec({\text{Frac}(R)}^{\text{sep}})\to X-Z$, we assume further that $X_i\to X$ is generically Galois. Hence to prove the finiteness of $\pi_1^{\text{\'et}}(X-Z,\bar{q})$  it suffices to show the tower stabilizes for sufficiently large $i$. Since $(R, \Delta)$ is BCM-regular, we have $s^{\underline{x}}_\perfd(R)>0$ by \autoref{prop.dH-BCMRegularImpliesPerfdPositive}.  

As $R$ is weakly BCM-regular, by \autoref{lem.QuasiEtaleStuffForWeakBCMRegular} we see that $R \to S_1$ splits and $S_1$ is weakly BCM-regular.  Repeating the argument we see that each $S_i \to S_{i+1}$ splits.  

However since perfectoid signature is bounded above by 1 by \autoref{prop.sFboundby1} and have transformation rule for split quasi-\'etale maps by \autoref{thm.TransformationRule}, we must have $\rank_R(S_i)=\rank_R(S_{i-1})$ for large enough $i$, in other words, $\rank_{S_{i-1}}(S_i)=1$ for large enough $i$. 

Note that any maximal element $Y=\Spec(S)\to X$ in the tower above represents the functor $\Hom_{X}(\bar{q},-)$. In fact, if this is not the case, then there is some other $(Y'\to X)\in \mathcal{C}$ such that $Y=\Spec(S)\to X$ does not dominate. Then by going back and forth of the equivalence $\mathcal{C}\simeq {(X-Z)}_{\text{\'et}}$, we can find some $(Y''\to X)\in \mathcal{C}$ dominating them both, giving $Y=Y''$ since the tower stabilizes. Therefore, we have
$$\pi_1^{\text{\'et}}(X-Z,\bar{q})\simeq \Aut_X(Y)\simeq \Gal(\text{Frac}(S)/\text{Frac}(R))$$ 
The order of the latter is less than or equal to $1/s^{\underline{x}}_\perfd(R)$ by the transformation rule.

For the final statement, suppose $(R, \m, k) \subseteq (S, \bn, k)$ is a finite quasi-\'etale extension of normal domains where $p \; | \; [K(S) : K(R)]$.  Now, $\Tr(\bn) \subseteq \m$ (see \autoref{lem.TraceForNormalDomains}) and thus we have an induced map $\overline{\Tr} : R/\bn = k \to k = R/\m$.  Furthermore any element $\overline{y}$ of $k = S/\bn$ is the residue class of some $y \in R$.  So we see $\Tr(y) = [K(S) : K(R)] \cdot y \in \m$ since $p \; | \; [K(S) : K(R)]$ and $p \in \m$. It follows that $\overline{\Tr}$ is the zero map and hence $\Tr$ is not surjective (its image lies in $\m$).  This completes the proof.
\end{proof}

We can translate this to a geometric setting by taking cones.  Recall from \cite{BMPSTWW-MMP}, \cf \cite{TakamatsuYoshikawaMMP}, that a pair $(Y, \Delta)$, where $Y$ projective over a complete local ring $(R, \m)$ and $\Delta$ is an effective $\bQ$-divisor, is called \emph{globally $\bigplus$-regular} if the map $\cO_Y \to \cO_Z(f^* \Delta)$ splits for every finite surjective map $f : Z \to Y$ from an integral normal scheme $Z$ such that $f^* \Delta$ has integer coefficients.

\begin{theorem}
    \label{thm.FinitenessFunGroupsForBigOpenInLogFano}
    Suppose that $Y$ is a normal integral scheme projective over $R = H^0(Y, \cO_Y)$ with equidimensional fibers, where $(R, \m, k)$ is a Noetherian complete local domain such that $k$ is algebraically closed.  
    Further suppose $\Delta \geq 0$ is a $\bQ$-divisor such that $-(K_Y + \Delta)$ is ample.  Suppose that for every Cartier divisor $D > 0$ on $Y$ there exists an $\epsilon > 0$ such that $(Y, \Delta + \epsilon D)$ is globally $\bigplus$-regular.  Then for any open set $U \subseteq Y$ whose complement $Y \setminus U$ has codimension $\geq 2$ (for instance the regular locus of $Y$), we have that $\pi_1^{\text{\'et}}(U)$ is finite.    
\end{theorem}
\begin{proof}
    We first observe that $Y$ is actually flat over any Noether-Cohen normalization $A' = W(k)\llbracket y_1, \dots, y_s\rrbracket \subseteq R$.  Indeed, since $Y$ is globally $\bigplus$-regular, we know that $Y$ is Cohen-Macaulay by \cite[Proposition 6.10]{BMPSTWW-MMP}. Now since $Y\to \Spec(A')$ is projective with equidimensional fibers and $A'$ is regular, we have $Y \to \Spec (A')$ is flat by \cite[\href{https://stacks.math.columbia.edu/tag/00R4}{Tag 00R4}]{stacks-project}.  

    By \cite[Theorem 8.1]{GabberLiuLorenziniHypersurfacesInProjectiveSchemes}, using that $Y \to \Spec (A')$ is projective flat and using that $-K_Y - \Delta$ is ample, there exists some $n \gg 0$ and a surjective finite morphism $\pi : Y \to \bP^{d}_{A'}$ so that $L := \pi^*(\cO_{\bP^{d}_{A'}}(1)) = \cO_Y(-n(K_Y - \Delta))$ and in particular $-n(K_Y - \Delta)$ is Cartier (the choice of $n \gg 0$ appears in their proof).  Let $S' = \bigoplus_{i \geq 0} H^0(Y, L^i)$ denote the section ring with respect to this ample line bundle.  It follows that we have a finite graded extension of graded rings over $A'$:
    \[
        A'' := A'[x_0, \dots, x_d] \subseteq S'.
    \]
    We may choose $g \in  A'[x_0, \dots, x_d]$ homogeneous so that $A''[g^{-1}] \subseteq S'[g^{-1}]$ is \'etale (in other words, $Y \setminus \Div_Y(g) \to \bP^{d}_{A'}\setminus \Div_{\bP^{d}_{A'}}(g)$ is \'etale).

    Consider the (graded) $\bQ$-divisor $\Delta_{S'} \geq 0$ on $S'$ corresponding to $\Delta$.  Since we are taking the section ring with respect to a multiple of $-K_Y - \Delta$, we see that $K_{S'} + \Delta_{S'}$ is $\bQ$-Cartier on $\Spec {S'}$.  
    We complete $S'$ at the homogeneous maximal ideal to obtain $S = \widehat{S'}$ and the corresponding $\bQ$-divisor $\Delta_{S}$.  We still have that $K_{S} + \Delta_{S}$ is $\bQ$-Cartier.  
    Let $(A, \m_A)$ denote the completion of $A''$ at the homogeneous maximal ideal, and notice that $A \to S$ is finite and that $A[g^{-1}] \subseteq {S}[g^{-1}]$ is \'etale.  



    \begin{claim}
        \label{clm.BCM-RegularityOfCone}
        $(S, \Delta_{S})$ is BCM-regular and in particular $s_{\perfd}^{p, y_1, \dots, x_d}(S) > 0$.
    \end{claim}
    \begin{proof}
        Let $S_1' \supseteq S'$ denote a $\bZ[1/m]$-graded normal extension ring of $S'$ such that $h'^*(K_{S'} + \Delta_{S'})$ is Cartier where $h' : \Spec S_1' \to \Spec S'$ is the induced map (this can be induced from a finite cover of $Y$).  Let $S_1$ denote the completion of $S_1'$ and write $h : \Spec S_1 \to \Spec S$ and pick $g_1 \in A''$ homogeneous and divisible by $g$, so that $A''[g_1^{-1}] \to S_1'[g_1^{-1}]$ is \'etale and hence so is $A[g_1^{-1}] \to S_1[g_1^{-1}]$.
        
        We observe that we have a factorization $S \to S_1 \to T := \widehat{S'^{+,\mathrm{GR}}}$ where $S'^{+,\mathrm{GR}}$ is the graded absolute integral closure of $S'$ ($S'$ is a $\bQ$-graded ring) and where $\widehat{(-)}$ denotes $p$-adic completion.  Notice that $T = \widehat{S'^{+,\mathrm{GR}}}$ is perfectoid and contains a compatible system of $p$-power roots of $g_1$.  Our hypothesis that $(Y, \Delta + \epsilon \Div_Y(g_1))$ is globally $\bigplus$-regular implies that $S' \to S_2'(\mu^* (\Delta_{S'} + \epsilon \Div_{S'}(g_1)))$ splits for every finite $\bZ[1/l]$-graded extension $S_2' \supseteq S'$ so that $\mu^* (\Delta_{S'} + \epsilon \Div_{S'}(g_1))$ has integer coefficients.  Hence 
        \[
            H^{s+1+d+1}_{\m_A}\big(S(K_{S})\big) \xrightarrow{1 \mapsto g_1^{1/p^e}} H^{s+1+d+1}_{\m_A}\big(T \otimes_{S_1} {S_1}(h^*(K_{{S}} + \Delta_{{S}}))\big)
        \]
        injects if we choose $e > 0$ so that $1/p^e < \epsilon$.  
        Therefore we can apply \autoref{prop.AlternateTheoremOnBCMRegularityAndPerfdSignaturePositive} and our claim follows.
    \end{proof}
    
    The proof is now nearly complete.  We claim that any quasi-\'etale cover of $Y$ induces quasi-\'etale cover of $S$ of the same degree.  Indeed suppose that $Z \to Y$ is a finite surjective map of degree $m$ between normal integral schemes.  Set $S_Z'$ to be the section ring of $Z$ with respect to the pullback of $L$.  This gives us a finite map $S' \to S_Z'$ that is quasi-\'etale of degree $m$ outside of the irrelevant ideal.  However, as it is harmless to assume that $Y \to \Spec (R)$ has positive degree fibers (as otherwise $Y = \Spec (R)$), and $Y \to \Spec (R)$ has equidimensional fibers, we see that the irrelevant ideal defines a codimension $\geq 2$ subset of $\Spec (S)$.  Thus $S' \to S_Z'$ is quasi-\'etale.  Taking completions proves the claim.
    
    In particular, since we have a maximum size of a quasi-\'etale cover of $S$ (the reciprocal perfectoid signature), we also have one for $Y$.   The result follows as it does in the local case.
\end{proof}

\begin{remark}
    In \autoref{thm.FinitenessFunGroupsForBigOpenInLogFano} we did not actually need the hypothesis that $(Y, \Delta + \epsilon D)$ is globally $\bigplus$-regular for \emph{every} Cartier divisor $D$.  Set $Y_1 = \Proj S_1$ with notation as in the proof of \autoref{clm.BCM-RegularityOfCone}.  Consider the finite surjective map $\pi : Y_1 \to Y \to \bP^d_{A'}$ (where $A' = W(k)\llbracket y_1, \dots, y_s\rrbracket$).  It is enough to take $D$ to be the pullback to $Y$ of an effective divisor $H$ on $\bP^d_{A'}$ so that the induced $Y_1 \setminus \pi^* H \to \bP^d_{A'} \setminus H$ is \'etale and so that $(Y, \Delta + \epsilon D)$ is globally $\bigplus$-regular for some $\epsilon > 0$.
\end{remark}

We can also recover the following analog of the main result of a paper of Greb-Kebekus-Peternell in characteristic zero \cite{GrebKebekusPeternellEtaleFundamentalGroupsKLT}.  We use the result of Stibitz \cite{StibitzEtaleCoversAndLocalFunGroups} to let us reduce the question to verifying finiteness of local \'etale fundamental groups (also compare with \cite{BhattGabberOlssonFinitenessOfFundamental} which also recovered the result of Greb-Kebekus-Peternell in characteristic zero via reduction to positive characteristic).  Finally, also see \cite{BhattCarvajalRojasGrafSchwedeTuckerFunGroups} for a an analogous result in characteristic $p > 0$.

There is one slight issue with directly applying Stibitz's result: the geometric points $x$ of a scheme $X$ need not have perfect or algebraically closed residue field (and hence neither do their completions).  Hence we need to base change the completions of the stalks at geometric points so that they have algebraically closed residue field.  

\begin{theorem}[\cf {\cite[Theorem 1]{StibitzEtaleCoversAndLocalFunGroups}, \cite[Theorem 1.1]{GrebKebekusPeternellEtaleFundamentalGroupsKLT}}]
    \label{thm.GrebKebekusPeternellVersion}
    Suppose $X$ is a normal integral Noetherian scheme surjective and of finite type over a DVR $V$ of mixed characteristic $(0, p)$ and let $Z \subseteq X$ denote the singular locus.  
    
    Consider  a geometric point $x \in X$ of mixed characteristic and we set $R_x' = \widehat{\cO_{X,x}}$ the $\m_{x}$-adic completion of $\cO_{X,x}$, with maximal ideal $\m' \subset R_x'$ and residue field $k' = R_x'/\m'$ with perfection $k = k'_{\perf}$.  Further let $C(k') \subseteq R_{x}'$ be a coefficient (Cohen) ring of the residue field $k'$ (which is separably closed since $x \in X$ is a geometric point), choose a map $C(k') \to C(k)$, and set $R_x = R_x' \widehat{\otimes}_{C({k'})} C(k)$ where the completion over the tensor product is $\m'$-adic.  
    Note the complete local ring $(R_x, \m, k)$ has perfect and hence algebraically closed residue field $k$.

    Suppose now that:
    \begin{enumerate}
        \item $X \times_V V[1/p]$ is of klt-type (that is there exists a boundary $\bQ$-divisor $\Delta$ that makes it klt), and 
        \item For each geometric point of mixed characteristic, there exists an associated $R_x$ as above which is of BCM-regular type (see \autoref{def.StrongBCM-Regularity}\autoref{def.StrongBCM-Regularity.d}, this means there exists a boundary $\bQ$-divisor $\Delta$ on $\Spec (R_x)$ making the pair BCM-regular).
    \end{enumerate}
    Then for every tower of quasi-\'etale generically Galois finite covers:
    \[
        X \leftarrow X_1 \leftarrow X_2 \leftarrow \dots
    \]
    the maps $X_{i+1} \to X_i$ are \'etale for $i \gg 0$.
\end{theorem}
\begin{proof}
    We show that $X$ satisfies \cite[Theorem 1(i)]{StibitzEtaleCoversAndLocalFunGroups} and adopt its notation.  In particular, it is enough to show that if $x \in X$ is a geometric point, then $\pi_1^{\text{\'et}}(X_x \setminus Z_x)$ is finite.  However, $X \times_V V[1/p]$ is finite type over the field $V[1/p]$ of characteristic zero.  Hence, the result follows from \cite[Corollary 1]{StibitzEtaleCoversAndLocalFunGroups} if $\cO_{X,x}$ has equal characteristic zero.

    Now, for $x \in X$ a geometric point of mixed characteristic, we first recall that $\pi_1^{\text{\'et}}(X_x \setminus Z_x) = \pi_1^{\text{\'et}}(\widehat{X_x} \setminus \widehat{Z_x})$ from \cite[Corollarie, page 26]{ElkikSolutionsDEquationsCoefficientsHenselian} where here $\widehat{X_x} := \Spec R'_x$ and where $\widehat{Z_x}$ is the singular locus (or equivalently the pullback of $Z$ to $\widehat{X_x}$).    
    
    We claim that $R_x'$ is normal.  We do this argument in several steps.  First we show that $R_x' \to R_x$ is faithfully flat.  Note that $C(k') \to C(k)$ is flat as $C(k)$ is torsion free over the DVR $C(k')$.  Thus $R_x' = R_x' \otimes_{C(k')} C(k') \to R_x' \otimes C(k)$ is flat by base change.  Hence $(R_x' \otimes C(k)) / \m'^n(R_x' \otimes C(k))$ is also flat over $R_x' / \m'^n$.  Now we can take inverse limits and apply \cite[\href{https://stacks.math.columbia.edu/tag/0912}{Tag 0912}]{stacks-project} and deduce that $R_x' \to R_x$ is flat as well.  It is thus faithfully flat as $R_x' \to R_x$ map is local.  But then $R_x'$ is normal by faithfully flat descent \cite[\href{https://stacks.math.columbia.edu/tag/033G}{Tag 033G}]{stacks-project}.

    We next claim that if $(R'_x, \m', k') \subseteq (S', \fran', l')$ is a finite local quasi-\'etale extension with $S'$ normal and integral, then the base change is also quasi-\'etale and $S$ is a normal domain:
    \[
        R_x \subseteq S := S' \widehat\otimes_{R'_x} R_x = S' \otimes_{R'_x} R_x
    \]
    where the identification of the tensor product and $\m'$-adically-complete tensor product follows since $R'_x \subseteq S'$ is finite.
    Once we observe that $R_x \subseteq S$ is a finite local and quasi-\'etale extension of normal domains, we are done since we already know that finite local domain extensions of $R_x$ can be bounded in generic rank by $1 \over s_{\perfd}^{\underline{x}}(R_x)$ and hence the generic rank of $R_x' \subseteq S'$ has the same bound.  
    
    Now we observe that $S$ is local. 
    Indeed since certainly $\frn' S = (\sqrt{\m'S'}) S$ is contained in each maximal ideal of $S$ (which must lie over $\m = \m' R$), it suffices to show that the maximal ideal $\frak{n}'\subseteq S'$ extends to a maximal ideal of $S$.
    
    \begin{claim}
    $S/\frak{n}'S$ is a field.
    \end{claim}
    \begin{proof}[Proof of Claim]
    We first observe that as $R_x$ is of BCM-regular type, in particular weakly BCM-regular, we have that $R'_x$ is also weakly BCM-regular since $R'_x\to R_x$ is faithfully flat: any perfectoid big Cohen-Macaulay $(R_x)^+$-algebra is also a perfectoid big Cohen-Macaulay $(R'_x)^+$-algebra and thus by faithful flatness, $R'_x$ is pure inside all (sufficiently large) perfectoid big Cohen-Macaulay $(R'_x)^+$-algebras.
    Now by \autoref{lem.QuasiEtaleStuffForWeakBCMRegular}, we know that $(R'_x, \frak{m}')\to (S', \frak{n}')$ is a split extension. Thus, $R'_x/\frak{m}' \to S'/\frak{n}'$ is separable, and hence they both equal to $k'$, see \autoref{lem.residuefieldsep} below. Now, we compute $S/\frak{n}'S =S'/\frak{n}'\otimes_{R_x'/\frak{m}'}R_x/\frak{m}'R_x = k'\otimes_{k'} k = k$, and this proves the claim.
    \end{proof}

    Our next goal is to show that $S$ is normal.
    We do this by showing that $S$ is $(\mathrm{R}_1)$ and $(\mathrm{S}_2)$. To see that $S$ is $(\mathrm{R}_1)$. We observe the map $R_x\to S$ is finite quasi-\'etale since it is obtained from a finite quasi-\'etale map via a faithfully flat base change ($R'_x \to S'$ base changed by $\otimes_{R'_x} R_x$). By assumption, $R_x$ is normal, and in particular, it is $(\mathrm{R}_1)$. Therefore, $S$ is also $(\mathrm{R}_1)$. Next we show that $S$ is $(\mathrm{S}_2)$. By \cite[\href{https://stacks.math.columbia.edu/tag/0339}{Tag 0339}]{stacks-project}, it suffices to show that all fiber rings $S\otimes_{S'} k(\mathfrak{p})$ are $(\mathrm{S}_2)$ where $\mathfrak{p}\subset S'$ is a prime ideal. In fact, it suffices to check all closed fibers are geometrically Cohen-Macaulay, see \cite[Theroem 4.1]{AvramovFoxbyGrothendiecklLocalization}. However, the closed fiber of this extension is $S/\frak{n}'S\otimes_{S'/\frak{n}'} k'=k$, and it is geometrically Cohen-Macaulay since $\dim k = 0$. 
    
    Finally observe that $R_x \subseteq S$ is quasi-\'etale by base change, and thus the result follows.
\end{proof}

In the above proof, we also used the following lemma which we believe is well known but for which we do not know a reference, compare also with \cite[Lemma 2.15]{CarvajalRojasSchwedeTuckerFundamentalGroup}.

\begin{lemma}\label{lem.residuefieldsep}
    Suppose $(R, \m) \subseteq (S, \frak{n})$ is a finite split quasi-\'etale extension of Noetherian normal local domains.  Then the residue field extension $R/\m \subseteq S/\frak{n}$ is separable.
\end{lemma}
\begin{proof}
    Let $\Tr : S \to R$ denote the trace map. Since $R \subseteq S$ is quasi-\'etale, $\Tr$ is an $S$-module generator of $\Hom_R(S, R)$.  Since $R \subseteq S$ is split, there exists some surjective map in $\Hom_R(S, R)$ and hence $\Tr$ must be surjective.  Now, because $\Tr(\frak{n}) \subseteq \m$ (see \autoref{lem.TraceForNormalDomains}), we obtain an induced map $\overline{\Tr} : S/\frak{n} \to R/\m$ which must also be surjective (and in particular nonzero).
    
    By \cite[Proposition 3.1]{AuslanderRimRamificationIndexAndMultiplicity}, $\overline{\Tr}$ is a multiple of the field-trace $S/\frak{n} \to R/\m$.  In particular, the field-trace $S/\frak{n} \to R/\m$ is nonzero and so the extension $R/\m \subseteq S/\frak{n}$ is separable as desired.
\end{proof}

\begin{remark}
    There are two ways to potentially improve the statement of \autoref{thm.GrebKebekusPeternellVersion}.  Presumably, in view of Stibitz, one may actually take $V$ to be an excellent local of dimension $\leq 2$.  The problem then becomes that the characteristic zero scheme $X \times_V V[1/p]$ is not of finite type over a field.  Presumably the various ways to study the local fundamental group generalize to this setting in view of \cite{takumi}, but verifying this would take us rather far afield.

    Additionally, it would be natural to expect that if $R_x'$ is of BCM-regular type, then so is $R_x$ (since analogous results hold in positive characteristic for strong $F$-regularity).  Again, we do not try to prove this here.
\end{remark}

We also have the following application to divisor class groups.  Compare with \cite{PolstraATheoremAboutMCM,MartinTheNumberOfTorsionDivisors} and \cite[Corollary 5.1]{CarvajalRojasFiniteTorsors} in characteristic $p > 0$, and see \cite{Ishiro.LocalLogRegularRingsVsToric} as well as \cite{ShimomotoEtAlPerfectoidTowersAndTheirTilts} for the log regular case in mixed characteristic.  Our approach in particular is inspired by the argument found in the paper of Carvajal-Rojas, which itself is inspired by the arguments of Martin and Polstra. But we first need a lemma on a version of countable prime avoidance.

\begin{lemma}
\label{lem. countably prime avoidance}
Let $(R,\m)$ be a Noetherian complete local domain and let $\{P_n\}_{n\in\mathbb{N}}$ be countably many primes, none of which is $\m$. Then we have $\m\notin \m^2 \cup (\cup_{j=1}^\infty P_j)$. 
\end{lemma}
\begin{proof}
Without loss of generality, we may assume that there is no containment among the $P_j$'s. We will construct an element $x\in \m$ that is not in the union. First of all, we pick $x_1\in \m -\m^2$. We then pick $x_2,x_3,\dots$ inductively as follows. If $x_1\notin P_1$, then we set $x_2=0$, otherwise, we choose $x_2 \in \m^2 -P_1$. Suppose $x_1,\dots,x_n$ has been chosen, if $x_1+\cdots +x_n\notin P_n$, then we set $x_{n+1}=0$, otherwise, we choose $x_{n+1}\in (\m^{n+1}\cap P_1\cap\cdots \cap P_{n-1}) - P_n$ (this is possible because none of the ideals $\m^{n+1}, P_1,\dots, P_{n-1}$ is contained in $P_n$ and $P_n$ is a prime). Now we consider the element 
$$x:=x_1+x_2+ x_3 + \cdots. $$
Note that $x\in \m$ since $x_n\in \m^n$ for all $n$ and $R$ is complete, and clearly, $x\notin \m^2$ since $x_1\notin \m^2$. It remains to check that $x\notin P_j$ for all $j$. But by construction, $x_{j+2}, x_{j+3}, \dots \in P_j$. So it is enough to show that $x_1+x_2+\cdots +x_{j+1}\notin P_j$. If $x_1+x_2+\cdots+x_j\notin P_j$, then $x_{j+1}=0$ by construction so $x_1+x_2+\cdots +x_{j+1}\notin P_j$. Otherwise, $x_1+x_2+\cdots+x_j\in P_j$ and we have chosen $x_{j+1}\notin P_j$. Thus $x_1+x_2+\cdots +x_{j+1}\notin P_j$ which completes the proof.
\end{proof}

We can now prove our result on finiteness of divisor class groups.

\begin{theorem}
    \label{thm.DivisorClassGroupApplication}
    Let $X=\Spec(R)$ where $(R,\m)$ is a Noetherian complete normal local domain of mixed characteristic $(0,p)$ with algebraically closed residue field. Further suppose that there exists $\Delta \geq 0$ such that $K_R + \Delta$ is $\bQ$-Cartier and such that $(R, \Delta)$ is BCM-regular. Then for some choices of $\underline{x}$, the torsion part of the divisor class group $\mathrm{Cl}(R)_{\mathrm{tors}}$ of $R$ satisfies 
    \[
        \big|\mathrm{Cl}(R)_{\mathrm{tors}}\big| \leq {1 \over s_{\perfd}^{\underline{x}}(R)},
    \]
    in particular it is finite.
\end{theorem}

\begin{proof}
    Suppose for the moment that $|\mathrm{Cl}(R)_{\mathrm{tors}}|$ is infinite.  Pick $D_1$ a Weil divisor of Cartier-index $m_1 > 0$.  Let $S_1 = \oplus_{i = 0}^{m-1} R(-iD)$ using an isomorphism $R(-mD) \cong R$ to provide the ring structure.  Note $(S_1, \frn_1)$ is also local where $\frn_1 = \m \oplus R(-D) \oplus \dots \oplus R(-(m-1)D)$, see for instance \cite[Lemma A.4]{MaPolstra.F-singularitiesCommutativeAlgebra}, and let $f_1 : \Spec (S_1) \to \Spec (R)$ be the induced map.  We call such extensions \emph{cyclic covers of optimal index} for the remainder of the proof.  Note they are quasi-\'etale after inverting $p$ since $R(-mD) \cong R$.  By \cite[Corollary 2.6]{TomariWatanabeNormalZrGradedRings}, the kernel of the map
    \[
        {\mathrm{Cl}}(R) \xrightarrow{f_1^*} {\mathrm {Cl}}(S_1)
    \] 
    is a cyclic group generated by the class of $D_1$.  By our hypothesis, there exists a $\bQ$-divisor $D_2$ on $R$ such that $f_1^* D_2$ is not Cartier.  
    We repeat the construction to obtain $S_2 \supseteq S_1$, a cyclic cover of optimal index $m_2$ for $f_1^* D_2 \subseteq \Spec S_1$.  Let $f_2 : \Spec (S_2) \to \Spec (R)$ denote the composition.  We see that the kernel of 
    \[
        \mathrm{Cl}(R) \xrightarrow{f_2^*} \mathrm{Cl}(S_2)
    \]
    is also finite (as it is a composition of maps with finite kernels) and so the process continues and we have a chain of local rings $(S_i, \frn_i)$
    \[
        R \subseteq S_1 \subseteq S_2 \subseteq \dots 
    \]
    where each $S_{i+1}$ is a cyclic cover of optimal index over $S_i$.

    \begin{claim}
        Pick a Noether-Cohen normalization $W(k)\llbracket z_2, \dots, z_d \rrbracket = A \subseteq R$. For each $i > 0$ there exists $0 \neq g_i \in A$ such that $R(D_i)[1/g_i] \cong R[1/g_i]$. Furthermore, we can choose $x_2,\dots,x_d \in \m-\m^2$ such that $pg_i, x_2, \dots,x_d$ is a system of parameters of $A$ for every $i$. 
    \end{claim}
    \begin{proof}
    After modifying $D_i$ by a principal divisor we may assume that $R(D_i)\subseteq R$ is a pure height one ideal. Let $W_i$ be the complement of the minimal primes of $R(D_i)$. We know that $W_i^{-1}R$ is a semi-local Dedekind domain and hence a PID so that $W_i^{-1}R(D_i) \cong W_i^{-1} R$. Thus there exists $0\neq g_i\in W_i$ such that $R(D_i)[1/g_i] \cong R[1/g_i]$. For the second assertion, note that by applying \autoref{lem. countably prime avoidance}, we can choose $x_2\in \m-\m^2$ that avoids all minimal primes of each $pg_i$ (which are countably many), and then we can choose $x_3\in \m-\m^2$ that avoids all minimal primes of each $(pg_i, x_2)$ (which are countably many), etc. Thus, we can find $x_2,\dots,x_d$ such that $pg_i, x_2, \dots,x_d$ is a system of parameters of $A$ for every $i$. 
    \end{proof}

    For the $\underline{x} = p, x_2, \dots, x_d$ from the claim, we see that we can apply \autoref{prop.GeneralTransformationRuleForSpecialSOP} for each finite extension $R \subseteq S_1 \subseteq S_2 \subseteq \dots$ (using different $g_i$ for different extensions).  Hence 
    \[
        (\prod_{i = 1}^n m_i) s^{\underline{x}}_{\perfd}(R) = s^{\underline{x}}_{\perfd}(S_n).
    \]
    as all these rings have the same residue field.  As $m_i \geq 2$, this contradicts the fact that $s^{\underline{x}}_{\perfd}(R) > 0$ and $s^{\underline{x}}_{\perfd}(S_n) \leq 1$.  Thus the torsion part of the class group is finite.  

    With this in hand, we can fix $R \subseteq S$ with induced $f : \Spec (S) \to \Spec (R)$ (a finite composition of optimal index cyclic covers $R = S_0 \subseteq S_1 \subseteq \dots \subseteq S_l = S$) so that $\mathrm{Cl}(R)_{\mathrm{tors}} = \ker f^*$.  Each extension $S_i \subseteq S_{i+1}$ has a projection onto the degree-0 piece $\pi : S_{i+1} \to S_i$ which generates $\Hom_{S_i}(S_{i+1}, S_{i})$ as an $S_i$-module and sends $\frn_{i+1}$ into $\frn_i$ using our explicit description of $\frn_{i+1}$ (just as we did with $(S_1, \frn_1)$ at the start of the proof, see \cite[Lemma A.4]{MaPolstra.F-singularitiesCommutativeAlgebra}).  Composing these yields a map $S \to R$ satisfying the conditions  (a)-(c) of \autoref{lem.I_inftyCompare}.  Picking $g$ as in the statement of \autoref{prop.GeneralTransformationRuleForSpecialSOP}, we can obtain $x_2, \dots, x_d$ so that 
    \[
        s_{\perfd}^{\underline{x}}(S) = [K(S) : K(R)] \cdot s_{\perfd}^{\underline{x}}(R)
    \]
    But $[K(S) : K(R)] = |\mathrm{Cl}(R)_{\mathrm{tors}}|$ by repeated application of \cite[Corollary 2.6]{TomariWatanabeNormalZrGradedRings}.  The result follows since again $s_{\perfd}^{\underline{x}}(S) \leq 1$.
\end{proof}

One could also expect stronger results in terms of the local Nori fundamental group scheme as studied in \cite{EsnaultViehwegSurfaceSingularitiesDominated}.  See \cite{CarvajalRojasFiniteTorsors} for what is currently known in this direction in characteristic $p > 0$.
\section{Further directions and questions}
\label{sec.FurtherQuestions}

In this section we discuss some open questions, and give some definitions of variants that we plan to study in future work.

\fakesubsection{The choice of $\underline{x}$ and $A \subseteq R$}

Perhaps the most important open question is showing that the invariants we have defined are independent of the choices we have made.

\begin{conjecture}
    \label{conj.Independence}
    The perfectoid signature $s^{\underline{x}}_{\perfd}(R)$ and Hilbert-Kunz multiplicity $e^{\underline{x}}_{\perfd}(J; R)$ are independent of the choice of the Noether-Cohen normalization $A \subseteq R$ and more specifically the system of parameters $\underline{x}$.
\end{conjecture}

There is some evidence for this, notably in characteristic $p > 0$ where it is not difficult to see that our definitions coincide with $F$-signature and Hilbert-Kunz multiplicity, \autoref{section:positivecharacterstic}.  That perfectoid signature is independent of the choice of $\underline{x}$ also follows for finite quotient singularities where $p$ does not divide the index by our transformation rule \autoref{sec.TransformationRule}, see for instance \autoref{cor.DuValSingularities}. On the other hand, perfectoid Hilbert-Kunz multiplicity $e_{\perfd}^{\underline{x}} (J; R)$ does not depend on $\underline{x}$ or $A$ when $J$ is generated by a system of parameters of $R$, see \autoref{cor.perfdHKparameterideal}.   For some additional related discussion on showing normalized length is independent of related choices in greater generality see \cite[Remark 14.5.69(ii)]{GabberRameroAlmostRingsAndPerfectoidSpaces}.

Of course, there are other ways one might avoid this issue, for instance taking the supremum over all such choices of $\underline{x} = \{p, x_2, \dots, x_d\}$,
\[
    s^{\sup}_{\perfd}(R) := \sup_{\{\underline{x}\} \subseteq R} s_{\perfd}^{\underline{x}}(R). 
\]
and likewise an infimum for perfectoid Hilbert-Kunz multiplicity.
Indeed, with this choice it is immediate for $\bQ$-Gorenstein $R$, that $s^{\sup}_{\perfd}(R) > 0$ if and only if $R$ is BCM-regular.  One also obtains transformation rules $s^{\sup}_{\perfd}(R) \cdot [K(S) : K(R)] \leq s^{\sup}_{\perfd}(S)$ which goes the correct way for the applications to local fundamental groups.

\fakesubsection{More general residue fields}

    If $(R, \m)$ is a complete local domain with possibly non-perfect residue field, we believe one can define an appropriate normalized length with respect to a sequence of parameters of $A \subseteq R$.  The essential idea is to use a larger Noether-Cohen regular subring $A' \supseteq A$ with perfect residue field (or in the $F$-finite residue field case, one can use a sequence of $A_e$).  We plan to explore this in a future work.  Of course, based on the characteristic $p > 0$ picture, one should also be able to compute perfectoid signature simply by first base changing to a perfect residue field.


\fakesubsection{Localization, faithfully flat extension, and semi-continuity}

We begin with a simple question, granting for the moment \autoref{conj.Independence}.

\begin{conjecture}[Localization]
    Suppose $(R, \m)$ is a Noetherian complete local domain with characteristic $p > 0$ residue field and $Q \in \Spec R$ is a prime containing $p$.  Then 
    \[
        s_{\perfd}(R) \leq s_{\perfd}(\widehat{R_Q}) \;\;\text{ and } \;\; e_{\perfd}(R) \geq e_{\perfd}(\widehat{R_Q}).
    \]
\end{conjecture}

We also expect the following behavior under faithfully flat extensions.

\begin{conjecture}[Faithfully flat extensions]
    Suppose $(R, \m)\to (S,\mathfrak{n})$ is a flat local extension of Noetherian complete local domains with characteristic $p > 0$ residue field.  Then 
    \[
        s_{\perfd}(R) \geq s_{\perfd}(S) \;\;\text{ and } \;\; e_{\perfd}(R) \leq e_{\perfd}(S).
    \]
\end{conjecture}

Of course, we expect more precise behavior on $\Spec$.

\begin{conjecture}[Semi-continuity]
    Suppose that $T$ is locally equidimensional excellent ring with $p$ in its Jacobson radical.  Let $Z = V(p) \subseteq \Spec (T)$ be locus where $p$ vanishes.  Then the functions 
    \[
        Q \in Z \mapsto e_{\perfd}(\widehat{T_Q}) \;\;\text{ and }\;\; Q \in Z \mapsto s_{\perfd}(\widehat{T_Q})
    \]
    are upper and lower semi-continuous respectively.
\end{conjecture}

The corresponding results are true in equal characteristic $p > 0$ by \cite{HanesNotesOnHKFunction,SmirnovUpperSemiContinuityOfHK,PolstraUniformBoundsInFFiniteRingsAndLowerSemicontinuity,Lyu.UniformboundsExcellentFpAlgebras}, \cf \cite{PolstraTuckerCombinedApproach,DeStefaniPolstraYaoGlobalizing}.

One can also ask how these functions behave when inverting $p$ (and so localizing to characteristic zero). It would be interesting if there are connections with the normalized volume in characteristic zero.  For some discussion on the comparison between normalized volume and $F$-signature, see \cite{LiLiuXuAGuidedTourNormalizedVolume}.  

\fakesubsection{Perfectoid signature of pairs}

In this paper we did not define perfectoid signature of pairs.  We provide below one potential definition of perfectoid signature of pairs $(R, \Delta)$ that we plan to explore in forthcoming work.  We state the definition in the Gorenstein case for simplicity here.

\begin{conjdef}
    Suppose $(R, \m)$ is a Noetherian complete normal  local domain of residue characteristic ring $p > 0$.  Suppose $\Delta = {1 \over n} \Div(f) \geq 0$ is a $\bQ$-Cartier $\bQ$-divisor.  Fix $S$ to be a finite domain extension of $R$ such that $f^{1/n} \in S$.  We set
    \[
        J_{\infty}^{S,\Delta} = \big\{x \in S^{\Ainfty}_{\perfd} \;|\; R \xrightarrow{1 \mapsto xf^{1/n}} S^{\Ainfty}_{\perfd} \text{ is not split}\big\}
    \]
    and then define:
    \[
        s_{\perfd}^{\underline{x}}(R, \Delta) = {1 \over [K(S) : K(R)]}\lambda_{\infty}\big(S^{\Ainfty}_{\perfd} / J_{\infty}^{S,\Delta} \big).
    \]
    Via the method of the proof of \autoref{lem.I_inftyCompare} we believe that this can be seen to be independent of the choice of $S$.
\end{conjdef}

We expect many of the results of this paper will generalize to this setting, compare with \cite{BlickleSchwedeTuckerFSigPairs1}.  This definition can readily be adapted to other sorts of pairs such as $(R, \fra^t)$ or $(R, [f_i]^t)$ as well, as defined for instance in  \cite{MaSchwedePerfectoidTestideal,RobinsonToricBCM,SatoTakagiArithmeticDeformations,MurayamaSymbolicTestIdeal}.


\bibliographystyle{skalpha}
\bibliography{main,preprints}

\end{document}